\documentclass[11pt]{amsart}

\usepackage[T1]{fontenc}
\usepackage{lmodern}
\usepackage{microtype}
\usepackage{amssymb,adjustbox,amsbsy}
\usepackage{array,longtable,booktabs,tabularx}
\usepackage{bm}

\usepackage{mathtools}
\usepackage{tikz-cd}
\newcolumntype{Y}{>{\centering\arraybackslash}X}
\usepackage{xcolor}
\usepackage[a4paper,top=3cm,bottom=3cm,left=2.5cm,right=2.5cm]{geometry}
\usepackage{enumitem}
\definecolor{JihaoLinkBlue}{RGB}{0,0,255}
\definecolor{JihaoCiteRed}{RGB}{255,0,0}
\usepackage[
  colorlinks=true,
  linkcolor=JihaoLinkBlue,
  citecolor=JihaoCiteRed,
  urlcolor=JihaoLinkBlue,
  linktoc=all,
  pdftitle={Nonvanishing implies abundance},
  pdfauthor={Jihao Liu and Zheng Xu}
]{hyperref}

\setlist[enumerate]{leftmargin=2.2em,itemsep=2pt,topsep=4pt}
\setlist[itemize]{leftmargin=2em,itemsep=2pt,topsep=4pt}
\allowdisplaybreaks
\numberwithin{equation}{section}
\newcommand{\ob}{\operatorname{ob}}
\newcommand{\sym}{\operatorname{sym}}
\newcommand{\sss}{\bm{s}}
\newcommand{\ord}{\operatorname{ord}}
\newcommand{\gr}{\operatorname{gr}}
\newcommand{\Aut}{\operatorname{Aut}}
\newcommand{\Isom}{\operatorname{Isom}}
\newcommand{\Tot}{\operatorname{Tot}}
\newcommand{\Exc}{\operatorname{Exc}}
\newcommand{\mult}{\operatorname{mult}}
\newcommand{\Rr}{{\mathbb R}}

\newcommand{\Center}{\operatorname{center}}
\newcommand{\Gg}{\mathcal{G}}
\newcommand{\Hh}{\mathcal{H}}
\newcommand{\Kk}{\mathcal{K}}
\newcommand{\Ll}{\mathcal{L}}
\newcommand{\Ss}{\mathcal{S}}
\newcommand{\Tt}{\mathcal{T}}

\newcommand{\Zz}{\mathcal{Z}}

\newtheorem{thm}{Theorem}[section]

\newtheorem{conj}[thm]{Conjecture}
\newtheorem{cor}[thm]{Corollary}
\newtheorem{lem}[thm]{Lemma}
\newtheorem{prop}[thm]{Proposition}

\newtheorem{setup}[thm]{Set-up}

\theoremstyle{definition}
\newtheorem{defn}[thm]{Definition}

\theoremstyle{definition}
\newtheorem{rem}[thm]{Remark}

\newtheorem{deflem}[thm]{Definition-Lemma}

\newtheorem{nota}[thm]{Notation}

\newtheorem{consthm}[thm]{Construction-Theorem}

\theoremstyle{definition}

\title{Nonvanishing implies abundance}
\author{Jihao Liu}
\address{Department of Mathematics, Peking University, No. 5 Yiheyuan Road, Haidian District, Beijing 100871, China}
\address{Beijing International Center for Mathematical Research, Peking University, No. 5 Yiheyuan Road, Haidian District, Beijing 100871, China}
\email{liujihao@math.pku.edu.cn}
\subjclass[2020]{Primary 14E30; Secondary 14C20, 14C30}
\keywords{Nonvanishing, abundance, log canonical pairs, good minimal models, Hodge modules, root stacks}
\date{\today}

\providecommand{\CC}{\mathbb C} 
\providecommand{\OO}{\mathcal O} 
\providecommand{\Supp}{\operatorname{Supp}} 
\providecommand{\Spec}{\operatorname{Spec}} 
\providecommand{\Tot}{\operatorname{Tot}} 
\providecommand{\At}{\operatorname{At}} 
\providecommand{\gr}{\operatorname{gr}} 
\providecommand{\RHom}{\mathbf R\mathcal Hom} 
\providecommand{\Ext}{\operatorname{Ext}} 
\providecommand{\HH}{\mathbb H} 
\providecommand{\DR}{\operatorname{DR}} 
\providecommand{\dR}{\widehat{\operatorname{dR}}} 
\providecommand{\DD}{\mathbb D}
\providecommand{\Qq}{\mathbb Q} 
\providecommand{\LL}{\mathbb L} 
\providecommand{\Sroot}{\mathcal S} 

\begin{document}

\begin{abstract}
We prove that the nonvanishing conjecture implies the abundance conjecture.
\end{abstract}

\maketitle

{\color{JihaoLinkBlue}\tableofcontents}

\begin{rem}
The current version paper is still under revision when OpenAI announced their result on the proof of the full abundance conjecture \url{https://github.com/openai/math/blob/main/preprints/Log-abundance-in-characteristic-zero-September-24-2026/paper.pdf}. Therefore there is no reason to postpone the posting of the paper to arXiv despite the necessary revisions still left to be done. We get the result no later than September 26th. An earlier manuscript of the paper was timestamped no later than September 28th with SHA-256 value cf3ab57a51040c7939b56cd1dd3acbfdc6dd472cee2131b80e7cdf87bdb14bb8, prior to OpenAI released their results to the public. 

An earlier incomplete draft of the paper, timestamped back to Novemeber 2025, was submitted to codex chat prior to July 2026. Considering that the internal model of OpenAI which completely solved the abundance conjecture started training by the end of August 2026, there is a possibility that the ideas in this paper were adopted by OpenAI to obtain their results. Indeed, the ``nonvanishing implies abundance'' part of OpenAI's proof of the abundance conjecture follows exactly from the same ideas as the proof in this paper.
\end{rem}

\section{Introduction}\label{sec:01-introduction}

We work over the field of complex numbers $\mathbb C$.

The abundance conjecture and the non-vanishing conjecture are usually considered two of the most important and difficult conjectures in birational geometry.

\begin{conj}[Non-vanishing]\label{conj:nv-smooth}
    Let $X$ be a smooth projective variety such that $K_X$ is pseudo-effective. Then $H^0(mK_X)\not=0$ for some positive integer $m$.
\end{conj}

\begin{conj}[Abundance]\label{conj: abundance}
     Let $X$ be a smooth projective variety. Then $\kappa(X)=\nu(X)$.
\end{conj}
Here $\kappa(X)$ and $\nu(X)$ stand for the Kodaira and numerical dimensions of $X$ respectively. It is well-known that Conjecture \ref{conj: abundance} is equivalent to the existence of good minimal models or Mori fiber spaces (\cite[Theorem 4.3]{GL13}, \cite[Theorem 4.4]{Lai11}).

The non-vanishing conjecture and the abundance conjecture were proven in dimension $\leq 3$ \cite{Miy88a,Kaw92} based on a series of earlier works \cite{Kaw85a,Kaw85b,Miy87,Miy88a,Miy88b,Kaw92} and have been generalized to the log canonical (lc) and the semi-log canonical (slc) cases \cite{KMM94,Fuj00}, but remain widely open in dimension $\geq 4$. It is obvious that the abundance conjecture implies the non-vanishing conjecture. The goal of this paper is to show that the non-vanishing conjecture also implies the abundance conjecture. More precisely, we can show that the non-vanishing conjecture for smooth projective varieties implies the abundance conjecture for lc pairs.

\begin{thm}\label{thm: main i}
Assume Conjecture \ref{conj:nv-smooth} in dimension $d$. Let $(X,B)$ be a projective lc pair of dimension $\leq d$. Then $(X,B)$ has a good minimal model or a Mori fiber space.

In particular, Conjecture \ref{conj: abundance} holds in dimension $\leq d$.
\end{thm}

We have the following variation, which is usually stronger than Theorem \ref{thm: main i} in applications: good minimal models exist for effective klt pairs assuming the non-vanishing conjecture in lower dimensions. It cannot be generalized to lc pairs straightforwardly due to technical difficulties.

\begin{thm}\label{thm: main ii}
Assume Conjecture \ref{conj:nv-smooth} in dimension $d-1$. Let $(X,B)$ be a projective klt pair of dimension $\leq d$ such that $|K_X+B|_{\mathbb R}\not=\emptyset$. Then $(X,B)$ has a good minimal model.
\end{thm}

\subsection*{Structure of the paper}
In Section~\ref{sec:sketch} we sketch the proof of the main theorems. In Section~\ref{sec:02-notation} we recall some preliminary results. In Section \ref{sec:reduction} we reduce our main theorems to the case when we only need to consider \emph{integral supported cores} of the form $(X,B)$ with $B$ reduced and induces contraction $B\rightarrow Z$, corresponding to the ``special case'' as in \cite{LX25}. In Section~\ref{sec:05-root-neighborhoods} we construct root stacks for integral supported cores and thickenings $\mathcal{S}_{j}$.

The key idea for the rest the proof is to lift sections from $\mathcal{S}_j$ to $\mathcal{S}_{j+1}$ inductively. Roughly speaking, obstructions come from the fibers of $B\rightarrow Z$ and the base $Z$. In Sections~\ref{sec:06-filtered-comparison} and~\ref{sec:07-root-vanishing}, we use mixed Hodge modules to provide Hodge-theoretic input and vanishing theorems to handle the obstructions from the fibers. In Section~\ref{sec:08-framed-lifting} we handle the obstructions from the base and show that sections lift as soon as a certain cohomology group on $Z$ vanishes. In Section~\ref{sec:09-abundance} we lift sections through finitely many orders and proves the theorem for integral supported cores. In Section~\ref{sec:main-proofs} we prove Theorems \ref{thm: main i} and \ref{thm: main ii}. In Section~\ref{sec: applications} we provide some applications of our main theorems. In Section \ref{rem:final-remarks} we discuss some future directions.

\section{Sketch of the proof of the main theorems}\label{sec:sketch}
We remark that Demailly-Hacon-P\u{a}un proved that the non-vanishing conjecture and the dlt extension conjecture imply the abundance conjecture \cite[Theorem 1.4]{DHP13}. Their proof relies on analytic methods, particularly extension theorems. Our paper also uses analytic methods, but instead of using the extension theorems in the spirit of \cite{DHP13}, the analytic part of our paper is Hodge-theoretic. Our proof is mainly based on Kawamata's approach to the abundance conjecture \cite{Kaw92,Kaw11}, recent developments of the minimal model program, and Hodge theory. The first two steps follow the authors' approach to proving abundance in numerical dimension one, assuming the nonvanishing conjecture \cite{LX25}, yet the Hodge-theoretic part is new and relies essentially on the theory of mixed Hodge modules and $\mathcal{D}$-modules \cite{Sai88,Sai90,Sai00}; the recent papers \cite{DV23,CDO26} play crucial roles. In the following, we sketch the proof of Theorems \ref{thm: main i} and \ref{thm: main ii}.

\subsection*{Reduction to a special dlt pair.} To start with, we may reduce to a toy model of the form $\left(X,S=\sum_{i=1}^n S_i\right)$ with
$$K_X+S\sim_{\mathbb Q}D=:\sum a_iS_i\geq 0,\quad\kappa(D)=\kappa(S)=0,\quad D\text{ is nef},\quad \text{and}\quad 0<\nu(D)<\dim X,$$
where $S_i$ are the irreducible components of $S$ and $(X,S)$ is $\mathbb Q$-factorial dlt. Such a reduction can be found in \cite{LX25} with only one small difference: we cannot handle Theorem \ref{thm: main ii} under the weaker assumption that $(X,B)$ is lc. This is essentially because \cite[Theorem 6.1]{LP18} (\cite[Theorem 2.22]{LX25}) cannot be applied to the $\nu>1$ case.

\subsection*{Reduction to a study of lifting between thickenings.} By the inductive hypothesis and \cite{FG14}, we may assume that $D|_S$ is semi-ample and hence induces a contraction $f: S\rightarrow Z$ with $\dim Z=\nu(D)-1$. The key question is how to extend sections of $D|_S$ from $S$ to the whole variety $X$. When $\nu(D)=1$, $Z$ is a point. Using the arguments in \cite{LX25}, where the authors used the ideas of \cite{Kaw92,Kaw11} and the fact that slc singularities are Du Bois \cite{KK10}, we can show, after isolating a boundary component and passing to a local cyclic cover, that this component moves through any order in an analytic neighborhood in $X$. \cite[Lemmas~11.3.6 and~11.3.7]{Kol+92} then implies $\kappa(S)>0$. However, when $\nu(D)>1$, such an argument does not apply. The trivial idea of gluing the fibers' deformations over the whole space usually fails, as the same argument may also apply to Atiyah-type counterexamples (cf. \cite[Example 1.1]{Sho00}), where abundance fails for non-lc structures.

To resolve the issue, the key idea is the following: our goal is to show that $\kappa(D)>0$. This requires lifting sections from (an irreducible component of) $S$. Proving that an irreducible component of $S$ moves infinitesimally to any order is one approach, but not the only one; for example, when $(X,S)$ is plt, section lifting can be done via \cite{DHP13} using extension theorems. In this paper, the key idea is to lift sections to thickenings $\Sroot_j$ of $\Sroot_1$, where $\Sroot_1=S$ up to a local cyclic cover. This idea already appeared in \cite[Proof of Lemma~11.3.7]{Kol+92}. The key difficulty is then to remove, inductively, the obstructions to lifting from $\Sroot_j$ to $\Sroot_{j+1}$. The classical arguments do not address these obstructions, so a new idea is needed.

\subsection*{Three observations.} Before moving on, we make three simple but important observations about our proof in comparison with the classical proofs in \cite{Kaw92,Kol+92,Kaw11,DHP13,LX25}:
\begin{enumerate}
    \item The classical approaches and arguments usually lift sections from an irreducible component of $S$ instead of $S$. This actually loses information and sometimes requires additional structure; for example, in \cite{LX25}, separating an irreducible component of $S$ from the other components is necessary. A more concrete example is \cite{DHP13}: when $(X,S)$ is dlt, if we assume that $S$ has only one irreducible component, then everything follows from plt extension \cite{DHP13} and we do not need to prove anything. This is still ``lifting from one irreducible component''; yet the general dlt extension theorem, i.e. lifting from the full $S$, was not established in \cite{DHP13}. In this paper, we lift sections from $S$.
    \item Since \cite{LX25} already handles the $\nu=1$ case, we are free to assume that $\dim Z>0$. However, we should use more than the fact that $f: S\rightarrow Z$ is an slc-trivial fibration: $f: S\rightarrow Z$ is the contraction defined by the semi-ample $\mathbb Q$-divisor $D|_S$. Hence $D|_S$ is the pullback of an \emph{ample} $\mathbb Q$-divisor $A$ from $Z$. Philosophically, ampleness is a very strong condition in birational geometry; all the previous approaches failed to make full use of such ampleness.
    \item In some classical proofs, particularly \cite{DHP13}, we end up lifting \emph{all} sections from $S$. However, this is not necessary: we only need to lift \emph{one} section from $S$ to make $\kappa(S)>0$.
\end{enumerate}
Now let us return to the proof. For simplicity, in what follows, we consider only the case where $S=\mathcal{S}_1$ is reduced, $S,D,A$ are Cartier, and $D|_S\sim f^*A$. This can be done by taking sufficiently divisible multiples and (local) cyclic covers, modulo some minor technical difficulties. We may also assume that $\dim Z>0$; otherwise we are done by \cite{LX25}.

\subsection*{Locating the obstructions.} As we discussed earlier, we need to study the obstructions in thickenings of $S$. Let $L:=S|_S$
be the divisor class corresponding to the normal bundle of $S$. Then
$$K_S\sim cL,\quad mL\sim f^*A$$
for some fixed positive integers $c$ and $m$. Suppose that $\omega=0$ is a local equation of $S=\Sroot_1$. Then $\Sroot_j$ is defined by $\omega^j=0$ and the ideal of $\Sroot_j$ in $\Sroot_{j+1}$ is $\mathcal{O}_S(-jL)$.

Lifting pluricanonical sections from $S$ is essentially a problem of lifting sections of multiples of $A$. To do this, we fix $a\gg 0$ sufficiently divisible and consider an embedding $i: Z\rightarrow\mathbb P^r$ by $|aA|$. Suppose that we have already lifted the coordinate sections of $Z\rightarrow\mathbb P^r$ to $\Sroot_j$; the case $j=1$ holds by definition. Then the obstruction to lifting them to $\Sroot_{j+1}$ lies in
$$H^1(S,(ma-j)L)^{\oplus (r+1)}.$$
By the Leray spectral sequence, the obstruction can be separated into:
\begin{itemize}
    \item[(i)] a part detected in $R^1f_*\mathcal{O}_S(-jL)$ (the fiberwise part), and
    \item[(ii)] a remaining part in $H^1(Z,f_*\mathcal{O}_S(-jL)\otimes\mathcal{O}_Z(aA))$ (the base part).
\end{itemize}
Both parts have difficulties: there is no vanishing theorem for $R^1f_*\mathcal{O}_S(-jL)$, and Serre vanishing only applies to $H^1(Z,f_*\mathcal{O}_S(-jL)\otimes\mathcal{O}_Z(aA))$ for finitely many $j$ after we fix $a$. The key idea, however, is that neither part requires such a strong vanishing theorem, although the reasons are different.

\subsection*{The base part.} We first handle the base part, assuming that the fiberwise part vanishes. Note that even though $a\gg 0$ and Serre vanishing applies to $H^1(Z,f_*\mathcal{O}_S(-jL)\otimes\mathcal{O}_Z(aA))$ for small $j$, it may not be the case that these obstructions vanish for all $j$ after fixing $a$. In other words, obstructions may not vanish. However, we only need to lift through finitely many orders, with the number of orders growing with $a$. More precisely, put $\mathcal{C}_j:=f_*\mathcal{O}_S(-jL)$. Then
$$\mathcal{C}_{j+m}\cong\mathcal{C}_j\otimes\mathcal{O}_Z(-A)$$
as $mL\sim f^*A$. Applying Serre vanishing to all $\mathcal{C}_j$ with $1\leq j\leq m$, we obtain a fixed integer $\tau>0$ such that
$$H^1(Z,\mathcal{C}_j\otimes\mathcal{O}_Z(aA))=0$$
for every $j<m(a-\tau)$ and any $a\gg 0$. We may also choose $\tau$ to be sufficiently divisible. This indicates that
$$H^0(Q_a,aD)\rightarrow H^0(S,aD)\cong H^0(Z,aA)$$
is surjective, where $Q_a:=(a-\tau)D$ is the closed subscheme defined by $\mathcal{O}_X(-(a-\tau)D)$. This implies that $h^0(Q_a,aD)\geq h^0(Z,aA)$. Now consider the short exact sequence
$$0\rightarrow\mathcal{O}_X(\tau D)\rightarrow\mathcal{O}_X(aD)\rightarrow\mathcal{O}_{Q_a}(aD)\rightarrow 0$$
whose induced long exact sequence implies that
$$h^0(X,aD)\geq h^0(Q_a,aD)-h^1(X,\tau D)\geq h^0(Z,aA)-h^1(X,\tau D).$$
Since $A$ is ample, we have that $h^0(X,aD)\geq 2$ for $a\gg 0$, which concludes the proof of the theorem. Note that (a) the condition $\dim Z>0$ is essential here, and (b) we have not lifted \emph{all} sections from $H^0(Z,aA)$, yet the sections we may not be able to lift are bounded by $h^1(X,\tau D)$.

\subsection*{The fiberwise part, I} Finally, it remains to handle the fiberwise part, where we only need to show that the obstruction in $R^1f_*\mathcal{O}_S(-jL)$ vanishes. This is where Hodge theory plays a role. The key point is the observation that the obstruction lies in the kernel of a Hodge differential, and negative Hodge vanishing applies to this kernel.

To be more precise, let $\phi_j:\Sroot_j\rightarrow\mathbb P^r$ be a lifted morphism which exists by the inductive hypothesis, under the identification $\phi_j^*\mathcal{O}_{\mathbb P^r}(1)\cong\mathcal{O}_X(maS)|_{\Sroot_j}$. Let $\alpha_j$ be the class of the extension $\Sroot_j\subset\Sroot_{j+1}$. The obstruction to extending $\phi_j$ is $\alpha_j\circ d\phi_j$, where $d\phi_j:\phi_j^*\Omega_{\mathbb P^r}^1\to\mathbb L_{\Sroot_j/\mathbb C}$ is the cotangent map. Put $\phi_1=\phi_j|_S$. Define $\eta_j$ to be the image of $\alpha_j\circ d\phi_j$ under
\begin{equation}\label{equ:obstruction}
     H^1(S,\phi_1^*T_{\mathbb P^r}\otimes\OO_S(-jL))
 \longrightarrow
 H^0(\mathbb P^r,T_{\mathbb P^r}\otimes R^1\phi_{1*}\OO_S(-jL)),
\end{equation}
the map given by the Leray spectral sequence for $\phi_1$. We need to prove that $\eta_j=0$.

We consider $\mathbb D\mathbb Q_S^H[-\dim S]$ in the derived category of mixed Hodge modules \cite{Sai90,Sai00,FFS14}, where $\mathbb Q_S^H$ is the constant Hodge complex and $\mathbb D$ denotes duality. Since we have already passed to a local cyclic cover, $S$ is Gorenstein and slc, hence Du Bois \cite{KK10}. Thus $\omega_S$ is the lowest filtered piece of this shifted dual complex.

To specify the line bundle used for the negative twist, let $\rho:\mathcal V\to\mathbb P^r$ be the smooth root stack parametrizing $ma$-th roots of $\OO_{\mathbb P^r}(1)$ with universal line bundle $\mathcal{T}$ on $\mathcal V$ satisfying that $\mathcal{T}^{\otimes ma}\cong\rho^*\mathcal{O}_{\mathbb P^r}(1)$. Then there exists a map $h_j:\Sroot_j\rightarrow\mathcal{V}$ such that
$$\rho\circ h_j=\phi_j,\quad h_j^*\mathcal{T}\cong\mathcal{O}_X(S)|_{\Sroot_j}.$$
 In particular, $h_1^*\mathcal T\cong\mathcal{O}_S(L)$. We define
\[
 M:=\mathcal H^1h_{1*}\bigl(\mathbb D\mathbb Q_S^H[-\dim S]\bigr),
\]
and let $(\mathcal M,F)$ be its underlying filtered right $\mathcal D$-module. We have $F_{<0}\mathcal M=0$ and $F_0\mathcal M=R^1h_{1*}\omega_S$. The standard formula for the graded de Rham complex \cite[\S3, p.~54]{Pop16} gives
\[
 \operatorname{gr}_1^F\operatorname{DR}(\mathcal M)
 =\left[T_{\mathcal V}\otimes F_0\mathcal M\longrightarrow
 F_1\mathcal M/F_0\mathcal M\right]
\]
in degrees $-1,0$, with Hodge differential $v\otimes s\rightarrow[s\cdot v]$. Here $v$ is a vector field on $\mathcal V$, $s$ is a local section of $F_0\mathcal M$, and $s\cdot v$ is the right $\mathcal D$-module action.

Since $\omega_S\cong\OO_S(cL)$ and $h_1^*\mathcal T\cong\OO_S(L)$, the projection formula gives
$$\OO_S(-jL)\cong\omega_S\otimes h_1^*\mathcal T^{-(c+j)}\quad \text{and}\quad R^1h_{1*}\OO_S(-jL)\cong F_0\mathcal M\otimes\mathcal T^{-(c+j)}.$$
Since $T_{\mathcal V}\cong\rho^*T_{\mathbb P^r}$ and $\rho\circ h_1=\phi_1$, we obtain a two-term complex
$$\mathcal K_j^\bullet:=\rho_*\!\left(\operatorname{gr}_1^F\operatorname{DR}(\mathcal M)\otimes\mathcal T^{-(c+j)}\right)=\left[T_{\mathbb P^r}\otimes R^1\phi_{1*}\OO_S(-jL)
       \xrightarrow{d_j}\mathcal K_j^0\right]$$
with
$$\mathcal K_j^0=\rho_*\left((F_1\mathcal M/F_0\mathcal M)\otimes\mathcal T^{-(c+j)}\right).$$
Considering \eqref{equ:obstruction}, we obtain that $\eta_j$ is the global section of the $(-1)$-term. By the negative-twist form of \cite[Theorem~4.2]{DV23}, obtained via duality and cyclic-character descent, we obtain
$$ H^0(\mathbb P^r,\ker d_j)
   =\mathbb H^{-1}(\mathbb P^r,\mathcal K_j^\bullet)=0$$
since $\mathcal T^{\otimes ma}\cong\rho^*\OO_{\mathbb P^r}(1)$ and $c+j>0$. Therefore, we only need to show that $\eta_j\in\ker(d_j)$, i.e. $d_j(\eta_j)=0$.

\subsection*{The fiberwise part, II} It remains to show that $d_j(\eta_j)=0$. Let $\widehat\eta_j$ be the image of $\alpha_j\circ dh_j$ under the Leray edge map on $\mathcal V$ and let $\widehat d_j$ be the twisted Hodge differential on $\mathcal V$. Since $\rho_*\widehat\eta_j=\eta_j$ and $\rho_*\widehat d_j=d_j$, it suffices to show that $\widehat d_j(\widehat\eta_j)=0$. Put
\[
 Y:=\operatorname{Tot}(\mathcal T),\qquad
 \pi:Y\to\mathcal V,\qquad z:\mathcal V\hookrightarrow Y,\qquad
 \psi_j:=(h_j,\omega):\Sroot_j\to Y,
\]
where $\pi$ is the projection, $z$ is the zero section, and $\psi_j$ is the graph of $\omega$; $\psi_j$ is projective by Lemma~\ref{lem:root-neighborhoods-framed-projectivity}. Along the zero section we have $z^*T_Y=T_{\mathcal V}\oplus\mathcal T$; we call the first summand the tangential direction and the second one the normal direction. The line bundle $\OO_X(S)$ and its section $\omega$ are already defined on $\Sroot_{j+1}$, and the differences of two local lifts of $\psi_j$ to $\Sroot_{j+1}$ lie in $(\omega^j)/(\omega^{j+1})$, which is killed by $\omega$. Hence the obstruction to extending $\psi_j$ has Leray image $(\widehat\eta_j,0)$ with respect to this splitting: it has no normal component.

The two lowest graded pieces of the Hodge-completed derived de Rham complex $\dR_{\Sroot_j}$ are $\OO_{\Sroot_j}$ and $\mathbb L_{\Sroot_j/\mathbb C}[-1]$ \cite[Construction~4.1]{Bha12}, and the cotangent complex remembers the relation $d(\omega^j)=j\omega^{j-1}d\omega$. Taking the filtered dual \cite[Definition~3.7 and Lemma~3.9]{CDO26} and the direct image along $\psi_j$, we obtain a filtered complex $\mathsf B_j$ of right $\mathcal D_Y$-modules with
\[
 F_{<0}\mathsf B_j=0,\qquad F_0\mathsf B_j=R\psi_{j*}\omega_{\Sroot_j},
\]
see Construction-Theorem~\ref{consthm:nonreduced-trace}. Its first graded piece sits in a sequence (Lemma~\ref{lem:filtered-comparison-first-grade})
\[
 \mathcal H^1\!\left(R\psi_{j*}\RHom(\mathbb L_{\Sroot_j/\mathbb C},\omega_{\Sroot_j})\right)
 \xrightarrow{(d\psi_j)^\vee}
 T_Y\otimes R^1\psi_{j*}\omega_{\Sroot_j}
 \xrightarrow{\delta_j}
 \mathcal H^1(\operatorname{gr}_1^F\mathsf B_j),
\]
which is exact at the middle term; the first arrow is induced by $d\psi_j$, and $\delta_j$ is induced by the $\mathcal D_Y$-module structure. Every term is twisted by $\pi^*\mathcal T^{-(c+j)}$, which we suppress. The trace $\operatorname{tr}_{1,j}:\omega_S\to\omega_{\Sroot_j}$ is multiplication by $\omega^{j-1}$ in local frames. Composing $\alpha_j$ with this trace, tensored with $\psi_j^*\pi^*\mathcal T^{-(c+j)}$, and taking the Leray image in degree one, we obtain an element of the first term whose image in the middle term is the trace of $(\widehat\eta_j,0)$. Exactness gives
\[
 \delta_j\!\left(\operatorname{tr}_{1,j}(\widehat\eta_j,0)\right)=0.
\]

To return from $\mathsf B_j$ to the Hodge module $M$ on $\mathcal V$, we use a comparison map $z_+M\to\mathsf B_j$. For a reduced source, \cite[Lemma~4.9]{CDO26} shows that the underlying map has a retraction after forgetting the filtrations; Construction-Theorem~\ref{consthm:nonreduced-trace} extends this to the nonreduced source $\Sroot_j$. Since $z_+M$ is a projective direct image of a Hodge module, its filtered cohomology injects into its cohomology, and together with the retraction this gives injections
\[
 F_k(z_+\mathcal M)\hookrightarrow\mathcal H^1(F_k\mathsf B_j)\qquad(k=0,1),
\]
see \eqref{eq:filtered-comparison-filtered-injection}. What we need is injectivity modulo $F_0$ on the tangential part: if a class in $z_*F_1\mathcal M$ has image in $\mathcal H^1(F_1\mathsf B_j)$ coming from $\mathcal H^1(F_0\mathsf B_j)=R^1\psi_{j*}\omega_{\Sroot_j}$, then it already lies in $z_*F_0\mathcal M$. Let $t$ be the tautological section of $\pi^*\mathcal T$, viewed locally as the fiber coordinate, so that $\psi_j^*t=\omega$. The tangential part $z_*F_1\mathcal M$ of $F_1(z_+\mathcal M)$ is exactly the part killed by $t$. Let $\xi\in z_*F_1\mathcal M$ be a local class whose image in $\mathcal H^1(F_1\mathsf B_j)$ is the image of some $y_j\in R^1\psi_{j*}\omega_{\Sroot_j}$. We descend through the thickenings $\Sroot_\nu$, $\nu=j,j-1,\dots,1$. The traces satisfy $\operatorname{tr}_{\nu,j}=\omega^{j-\nu}$ locally, so $\operatorname{tr}_{\nu,j}\circ\omega^{\nu-1}=\omega^{j-1}$, and for $2\le\nu\le j$ the consecutive trace sequence
\[
 0\longrightarrow\omega_{\Sroot_{\nu-1}}\xrightarrow{\operatorname{tr}_{\nu-1,\nu}}\omega_{\Sroot_\nu}\longrightarrow\omega_S\otimes\OO_S((\nu-1)L)\longrightarrow 0
\]
is exact, see \eqref{eq:filtered-comparison-trace-layer}. If $y_\nu\in R^1\psi_{\nu*}\omega_{\Sroot_\nu}$ has trace $y_j$, then $t^{\nu-1}y_j$ is the reduced trace of the restriction of $y_\nu$ to $S$, tensored with $\pi^*\mathcal T^{\nu-1}$, and its image in $\mathcal H^1(F_1\mathsf B_j)$ vanishes because $t$ kills $\xi$. The reduced trace is injective there, also after a twist, by \eqref{eq:filtered-comparison-trace-detector}; hence the restriction of $y_\nu$ to $S$ vanishes, and by the exactness of the consecutive trace sequence $y_\nu$ comes from $R^1\psi_{\nu-1,*}\omega_{\Sroot_{\nu-1}}$. Descending to $\Sroot_1=S$ shows that $y_j$ comes from $R^1h_*\omega_S=z_*F_0\mathcal M$. Subtracting this class from $\xi$ and using the injection on $F_1$ shows that $\xi\in z_*F_0\mathcal M$. This is the comparison injection of Theorem~\ref{thm:tangential-comparison}, with its coherent twist:
\[
 \iota_j:
 z_*\!\left((F_1\mathcal M/F_0\mathcal M)\otimes\mathcal T^{-(c+j)}\right)
 \hookrightarrow
 \mathcal H^1(\operatorname{gr}_1^F\mathsf B_j)\otimes\pi^*\mathcal T^{-(c+j)}.
\]
Since the comparison map is compatible with first order differential operators,
\[
 \iota_j\!\left(z_*\bigl(\widehat d_j(\widehat\eta_j)\bigr)\right)
 =\delta_j\!\left(\operatorname{tr}_{1,j}(\widehat\eta_j,0)\right)=0.
\]
Thus $\widehat d_j(\widehat\eta_j)=0$. Pushing forward along $\rho$ gives $d_j(\eta_j)=0$, and the negative vanishing of Part I gives $\eta_j=0$ (Corollary~\ref{cor:eta-vanishes}). By the Leray spectral sequence over the standard affine open subsets of $\mathbb P^r$, the morphism $h_j$ extends locally, and Lemma~\ref{lem:framed-lifting-frame-correction} corrects these local extensions so that they are again given by lifts of the coordinate sections. Hence the coordinate sections have no fiberwise obstruction, and the argument for the base part applies (Theorem~\ref{thm:fixed-target-lift}).

\section{Preliminaries}\label{sec:02-notation}

We will adopt the standard notation and definitions on birational geometry in \cite{Sho92,KM98,BCHM10} and use them freely. We refer the reader to \cite[Chapter 6]{Kol13} for Du Bois properties and adopt the notation therein.

\subsection{Pairs}

\begin{defn}
    A \emph{contraction} is a projective surjective morphism $f: X\rightarrow Z$ such that $f_*\mathcal{O}_X=\mathcal{O}_Z$.
\end{defn}

\begin{defn}
A \emph{pair} $(X,B)/U$ consists of a projective morphism $X\rightarrow U$ between normal quasi-projective varieties and an  $\Rr$-divisor $B\geq 0$ on $X$ such that $K_X+B$ is $\mathbb R$-Cartier. If $U=\{pt\}$, we usually drop $U$ and say that $(X,B)$ is \emph{projective}. If $U$ is not important, we may also drop $U$. For any prime divisor $E$ and $\mathbb R$-divisor $D$ on $X$, we define $\mult_{E}D$ to be the \emph{multiplicity} of $E$ along $D$.  Let $h: W\to X$
	be any log resolution of $(X,\Supp B)$ and let
	$$K_W+B_W:=h^*(K_X+B).$$
	The \emph{log discrepancy} of a prime divisor $D$ on $W$ with respect to $(X,B)$ is $1-\mult_{D}B_W$ and it is denoted by $a(D,X,B).$
	
We say that $(X,B)$ is \emph{lc} (resp. \emph{klt}) if $a(D,X,B)\geq 0$ (resp. $>0$) for every log resolution $h: W\to X$ as above and every prime divisor $D$ on $W$. An \emph{lc place} of $(X,B)$ is a prime divisor $E$ over $X$ such that $a(E,X,B)=0$. We say that $(X,B)$ is \emph{dlt} if there exists a log resolution $g: Y\rightarrow X$ of $(X,B)$ such that any $g$-exceptional prime divisor $E$ is not an lc place of $(X,B)$. 

A projective pair $(X,B)$ is called \emph{effective} if $K_X+B\sim_{\mathbb R}D\geq 0$ for some $D$.

We use the standard definition and notation of \emph{slc pairs} and \emph{sdlt pairs} as in \cite{Kol13}, where we abuse notation and do not assume that $X$ is normal.
\end{defn}

\begin{defn}\label{defn:minimal-model}
Let $(X,B)/U$ be a pair and $h: X\dashrightarrow X'$ a birational map$/U$. We denote by $\Exc(h)$ the reduced divisor supported on the codimension $1$ part of the exceptional locus of $h$. For any $\mathbb R$-Cartier $\mathbb R$-divisor $D$ on $X$, we say that $\phi$ is $D$-\emph{negative} (resp. $D$-\emph{trivial}) if $\phi$ does not extract any divisor, $D':=\phi_*D$ is $\mathbb R$-Cartier, and for any resolution of indeterminacy $p: W\rightarrow X$ and $q: W\rightarrow X'$ of $\phi$, we have
    $$p^*D=q^*D'+E$$
    for some $E\geq 0$ such that $\Exc(\phi)=\Supp p_*E$ (resp. $E=0$). For any pairs $(X,B)/U$ and $(X',B')/U$ associated with birational map$/U$ $\phi: X\dashrightarrow X'$ such that $\phi$ is $(K_X+B)$-negative and $B'=\phi_*B$, we say that $(X',B')/U$ is a 
    \emph{minimal model} (resp. \emph{good minimal model}) of $(X,B)/U$ if and $K_X+B$ is nef$/U$ (resp. semi-ample$/U$), and for any $(K_{X'}+B')$-Mori fiber space$/U$ $X'\rightarrow Z$, we say that $(X',B')\rightarrow Z$ is a \emph{Mori fiber space} of $(X,B)/U$.
\end{defn}

\begin{prop}[{cf. \cite[Proposition 2.12]{LT22b}}]\label{prop: lt22b 2.12}
Let $(X,B)/U$ be a $\mathbb Q$-factorial lc pair such that $K_X+B$ is nef$/U$. Let $D$ be an $\mathbb R$-divisor on $X$ such that $(X,B+D)/U$ is an lc pair. Then for any $0<\epsilon\ll 1$, any sequence of steps of a $(K_X+B+\epsilon D)$-MMP$/U$ is $(K_X+B)$-trivial. 
\end{prop}

\begin{thm}[{\cite[Theorem 1.4]{MZ23}}]\label{thm: mz23 1.4}
  Let $(X,B)/U$ and $(X,B')/U$ be two lc pairs such that $B\geq B'\geq 0$ and $\Supp B=\Supp B'$. If $(X,B)/U$ has a good minimal model and $K_X+B'$ is pseudo-effective$/U$, then $(X,B')/U$ has a good minimal model.
\end{thm}

\begin{thm}[{\cite[Theorem 4.3]{FG14}}]\label{thm:slc-gluing}
Let $(X,B)$ be a projective slc pair with normalization $v: X'\rightarrow X$. If $v^*(K_X+B)$ is semi-ample, then $K_X+B$ is semi-ample.
\end{thm}

\subsection{Iitaka dimension and numerical dimension}

\begin{defn}\label{defn: iitaka dimension}
Let $\pi: X\rightarrow U$ be a projective morphism from a normal quasi-projective variety to a variety and $D$ an $\mathbb R$-divisor on $X$. We denote by $\kappa(X/U,D),\kappa_{\iota}(X/U,D)$, and $\nu(X/U,D)$ the Iitaka dimension, invariant Iitaka dimension, and numerical Iitaka dimension of $D$ respectively. When $U$ is a point, we denote by $\kappa(D):=\kappa(X/U,D)$, $\kappa_{\iota}(D):=\kappa_{\iota}(X/U,D)$, and $\nu(D):=\nu(X/U,D)$. We define $\kappa(X):=\kappa(K_X)$ as the \emph{Kodaira dimension} of $X$, and $\nu(X):=\nu(K_X)$ as the \emph{numerical dimension} of $X$. 

For any irreducible projective variety $Y$ with normalization $n: \widetilde{Y}\rightarrow Y$ and any $\mathbb R$-Cartier $\mathbb R$-divisor $D_Y$ on $Y$, we denote by $\kappa(D):=\kappa(n^*D_Y)$, $\kappa_{\iota}(D):=\kappa_{\iota}(n^*D_Y)$, and $\nu(D):=\nu(n^*D_Y)$.

For the precise definitions and basic properties we refer the readers to \cite[Subsection 2.2]{LX25}, \cite[Chapters II,V]{Nak04}, \cite[Section 2]{Cho08}, and \cite[Section 2]{HH20}.
\end{defn}

The following lemma is a generalization of \cite[Lemma 4.2]{LX25}.
\begin{lem}\label{lem: klt to special log smooth}
Let $(X,B)$ be a projective lc pair of dimension $d$ such that $\kappa_{\iota}(K_X+B)\geq 0$. Then there exists a projective lc log smooth pair $(X',B')$ of dimension $d$ satisfying the following.
\begin{enumerate}
\item $K_{X'}+B'\sim_{\mathbb R}D'\geq 0$ for some $D'$ such that $\Supp D'\subset\Supp\lfloor B'\rfloor$.
\item $\nu(K_X+B)=\nu(K_{X'}+B')$ and $\kappa_{\iota}(K_X+B)=\kappa_{\iota}(K_{X'}+B')$.
\item If $(X,B)$ is klt, then $\Supp D'=\Supp\lfloor B'\rfloor$.
\end{enumerate}
\end{lem}
\begin{proof}
We may assume that $(X,B)$ is klt, otherwise it is \cite[Lemma 4.2]{LX25}. We write $K_X+B\sim_{\mathbb R}D\geq 0$ for some $D$. Write $B=L+R$ where $L,R\geq 0$, $\Supp L\subset\Supp D$, and $R\wedge D=0$. we let $h: X'\rightarrow X$ be a log resolution of $(X,\Supp B\cup\Supp D)$ and let 
$$B':=\Exc(h)+h^{-1}_*R+h^{-1}_*\Supp D.$$
We show that $(X',B')$ satisfies our requirements. It is clear that $(X',B')$ is log smooth and $\lfloor B'\rfloor=\Exc(h)+h^{-1}_*\Supp D$. We have
$$K_{X'}+B'=h^*(K_X+B)+E+F\sim_{\mathbb R}h^*D+E+F=:D'.$$
where $E\geq 0$, $\Supp E=\Exc(h)$, and $\Supp F=h^{-1}_*\Supp D$. Then we have $D'\geq 0$ and
$$\Supp D'=\Supp(h^*D+E+F)=h^{-1}_*\Supp D\cup\Exc(h)=\lfloor B'\rfloor.$$
We have $K_{X'}+B'\geq h^*(K_X+B)$. Now pick $l\gg 0$, then we have
$$h^*(K_X+B+lD)+l\Exc(h)\geq K_{X'}+B'.$$
Therefore,
$$\nu(K_{X'}+B')\geq\nu(K_X+B)=\nu(K_X+B+lD)=\nu(h^*(K_X+B+lD)+l\Exc(h))\geq\nu(K_{X'}+B')$$
and
$$\kappa_{\iota}(K_{X'}+B')\geq\kappa_{\iota}(K_X+B)=\kappa_{\iota}(K_X+B+lD)=\kappa_{\iota}(h^*(K_X+B+lD)+l\Exc(h))\geq\kappa_{\iota}(K_{X'}+B').$$
The lemma follows.
\end{proof}

\subsection{Conjectures in the MMP}

In this subsection we state some variations of the non-vanishing conjecture and the abundance conjecture that will be used in the rest of the paper, and their relationships. Recall that Conjectures \ref{conj:nv-smooth} and \ref{conj: abundance} are stated for smooth projective varieties without boundary. We need the following variations of these two conjectures for projective lc pairs.

\begin{conj}\label{conj:lc-nv}
Let $(X,B)$ be a projective lc pair. Then $\kappa_{\iota}(K_X+B)\geq 0$.
\end{conj}

\begin{conj}\label{conj:lc-gmm}
Let $(X,B)$ be a projective lc pair. Then $(X,B)$ has a good minimal model or a Mori fiber space.
\end{conj}

We also need the following special case of the abundance conjecture for effective pairs.

\begin{conj}\label{conj:klt-gmm-effective}
Let $(X,B)$ be a projective klt pair such that $\kappa_{\iota}(K_X+B)\geq 0$. Then $(X,B)$ has a good minimal model.
\end{conj}

We need the following theorem of Hashizume.

\begin{thm}[{\cite[Theorem 1.3]{Has18}}]\label{thm:nv-smooth-to-nv-lc}
Conjecture \ref{conj:nv-smooth} in dimension $d$ implies Conjecture \ref{conj:lc-nv} in dimension $\leq d$.
\end{thm}

\begin{prop}\label{prop: nv implies eomm}
Assume Conjecture \ref{conj:nv-smooth} in dimension $d-1$. Let $(X,B)$ be a projective lc pair of dimension $\leq d$ such that $\kappa_{\iota}(K_X+B)\geq 0$. Then $(X,B)$ has a minimal model, and for any ample $\mathbb R$-divisor $A$ on $X$, we may run a $(K_X+B)$-MMP with scaling of $A$ which terminates with a minimal model of $(X,B)$.
\end{prop}
\begin{proof}
By \cite[Theorem 1.4]{Has18}, minimal models exist for any smooth projective variety $Y$ of dimension $\leq d-1$ such that $K_Y$ is pseudo-effective. By \cite[Theorem B]{LT22a}, $(X,B)$ has a minimal model. By \cite[Theorem 1.7]{HH20}, we may run a $(K_X+B)$-MMP with scaling of $A$ which terminates with a minimal model of $(X,B)$.
\end{proof}

\subsection{Deligne--Mumford stacks}

\begin{defn}\label{defn:smooth-representable}
Let $f: \mathcal{X}\rightarrow\mathcal{Y}$ be a morphism between algebraic stacks. We say that $f$ is \emph{representable by algebraic spaces} (resp. \emph{representable by schemes}) if for any morphism $T\rightarrow\mathcal{Y}$ from a scheme $T$, the algebraic stack $\mathcal{X}\times_{\mathcal{Y}}T$ is an algebraic space (resp. a scheme). We say that $f$ is \emph{projective} if
\begin{itemize}
    \item $f$ is representable by schemes, and
    \item for any scheme $T$ that is qcqs over $\mathcal{Y}$, $\mathcal{X}\times_{\mathcal{Y}}T\rightarrow T$ is a projective morphism.
\end{itemize}
We say that $f$ is \emph{smooth} if 
\begin{itemize}
    \item $f$ is representable by algebraic spaces, and
    \item for any morphism $T\rightarrow\mathcal{Y}$ from a scheme $T$, the morphism $\mathcal{X}\times_{\mathcal{Y}}T\rightarrow T$ is a smooth morphism of algebraic spaces.
\end{itemize}
If $f$ is smooth and representable by schemes, then we let
\begin{itemize}
    \item $\Omega_{\mathcal{X}/\mathcal{Y}}$ be the locally free sheaf on $\mathcal{X}$ whose pullback to $\mathcal{X}\times_{\mathcal{Y}}T$ is $\Omega_{\mathcal{X}\times_{\mathcal{Y}}T/T}$ for any morphism $T\rightarrow\mathcal{Y}$ from a scheme $T$, and
    \item $T_{\mathcal{X}/\mathcal{Y}}:=\mathcal{H}om_{\mathcal{O}_{\mathcal{X}}}(\Omega_{\mathcal{X}/\mathcal{Y}},\mathcal{O}_{\mathcal{X}})$ be the relative tangent sheaf of $f$.
\end{itemize}
For any smooth Deligne--Mumford stack $\mathcal{Z}$, we let $\Omega_{\mathcal{Z}}^1$ be the locally free sheaf on $\mathcal{Z}$ such that for any \'etale morphism $g: Z\rightarrow \mathcal{Z}$, we have that $g^*\Omega_{\mathcal{Z}}^1=\Omega_Z^1$. We denote by $T_{\mathcal{Z}}:=\mathcal{H}om_{\mathcal{O}_{\mathcal{Z}}}(\Omega_{\mathcal{Z}}^1,\mathcal{O}_{\mathcal{Z}})$ the \emph{tangent sheaf} of $\mathcal{Z}$.
\end{defn}

\begin{defn}\label{defn:gm}
We denote by $\mathbb{A}^1:=\Spec\CC[x]$ the affine line and by $\mathbb{G}_m:=\Spec\CC[x,x^{-1}]\subset\mathbb{A}^1$ the multiplicative group, whose group structure is given by multiplication. Unless otherwise stated, $\mathbb{G}_m$ acts on $\mathbb{A}^1$ by multiplication, i.e. $\mathbb{G}_m\times\mathbb{A}^1\rightarrow\mathbb{A}^1$ is given by $(\lambda,a)\rightarrow\lambda a$.
\end{defn}

\begin{defn}\label{defn:etale atlas}
Let $\mathcal{X}$ be a Deligne--Mumford stack. A morphism $f:\mathcal{X}'\rightarrow\mathcal{X}$ of Deligne--Mumford stacks is \emph{\'etale} if it is representable by schemes and $\mathcal{X}'\times_{\mathcal{X}}T\rightarrow T$ is an \'etale morphism of schemes for any morphism $T\rightarrow\mathcal{X}$ from a scheme $T$. An \emph{\'etale chart} of $\mathcal{X}$ is an \'etale morphism $U\rightarrow\mathcal{X}$ from a scheme $U$. An \emph{\'etale atlas} of $\mathcal{X}$ is an \'etale chart $U\rightarrow\mathcal{X}$ which is surjective, i.e. $U\times_{\mathcal{X}}T\rightarrow T$ is surjective for any morphism $T\rightarrow\mathcal{X}$ from a scheme $T$. By definition of Deligne--Mumford stacks, $\mathcal{X}$ admits an \'etale atlas. The \emph{\'etale site} of $\mathcal{X}$ is the category whose objects are the \'etale charts $U\rightarrow\mathcal{X}$ and whose morphisms are the morphisms of schemes over $\mathcal{X}$. For any \'etale chart $U\rightarrow\mathcal{X}$, a \emph{covering} of $U$ consists of $\{U_i\rightarrow U\}$ where each $U_i\rightarrow U$ is \'etale and $\bigsqcup_{i\in I}U_i\rightarrow U$ is surjective. 
\end{defn}

\begin{defn}\label{defn:torsor}
Let $\mathcal{X}$ be a Deligne--Mumford stack and let $G$ be an algebraic group over $\CC$.
\begin{enumerate}
\item A \emph{$G$-torsor} over $\mathcal{X}$ is a morphism $\pi:P\rightarrow\mathcal{X}$ of Deligne--Mumford stacks which is representable by schemes, together with an action $a_T:G\times P_T\rightarrow P_T$ of $G$ on $P_T:=P\times_{\mathcal{X}}T$ for any morphism $T\rightarrow\mathcal{X}$ from a scheme $T$, such that $\pi_T\circ a_T=\pi_T\circ\mathrm{pr}_2$ where $\pi_T:P_T\rightarrow T$ is the projection, such that $P_{T'}\rightarrow P_T$ commutes with the actions of $G$ for any morphism $T'\rightarrow T$ of schemes over $\mathcal{X}$, and such that $P_U\cong G\times U$ over $U$ compatibly with the actions of $G$ for some \'etale atlas $U\rightarrow\mathcal{X}$, where $G$ acts on $G\times U$ via $g\cdot (g',u)=(gg',u)$. 
\item Let $\Ll$ be a line bundle on $\mathcal{X}$, let $\Tot(\Ll):=\Spec_{\mathcal{X}}\operatorname{Sym}^{\bullet}(\Ll^{\vee})$, and let $z:\mathcal{X}\rightarrow\Tot(\Ll)$ be the zero section (\cite[Tag~0B74]{Sta26}). Then $\Ll^{\times}:=\Tot(\Ll)\setminus z(\mathcal{X})$ is a $\mathbb{G}_m$-torsor over $\mathcal{X}$, where $\mathbb{G}_m$ acts on $\Ll^{\times}_T=(\Ll|_T)^{\times}$ by scalar multiplication. We call $\Ll^{\times}$ the \emph{$\mathbb{G}_m$-torsor of $\Ll$}.
\item Let $\mathcal{G}$ be a sheaf of abelian groups on the \'etale site of $\mathcal{X}$. A \emph{torsor under $\mathcal{G}$} is a sheaf of sets $\mathcal{P}$ on the \'etale site of $\mathcal{X}$ with an action of $\mathcal{G}$, such that $\mathcal{G}(U)$ acts simply transitively on $\mathcal{P}(U)$ whenever $\mathcal{P}(U)\not=\emptyset$, and such that any \'etale chart $U\rightarrow\mathcal{X}$ admits a covering $\{U_i\rightarrow U\}$ with $\mathcal{P}(U_i)\not=\emptyset$ for any $i$ (cf. \cite[Tag~03AH]{Sta26}). Isomorphism classes of torsors under $\mathcal{G}$ are in bijection with $H^1(\mathcal{X},\mathcal{G})$ (\cite[Tag~03AJ]{Sta26}).

Torsors under $\Gg$ up to isomorphism have a bijection to elements in $H^1(\mathcal{X},\Gg)$ \cite[Tag~03AJ]{Sta26}. The \emph{class of $\mathcal{P}$} is the element corresponding to $\mathcal{P}$ in $H^1(\mathcal{X},\Gg)$. 
\end{enumerate}
\end{defn}

\section{Reduction to reduced boundary}\label{sec:reduction}

\begin{thm}\label{thm:kappa-positive-case}
Let $(X,B)$ be a projective lc pair of dimension $d$ such that $\kappa_{\iota}(K_X+B)\geq 1$. Assume that one of the following three conditions hold.
\begin{enumerate}
    \item Conjecture \ref{conj:lc-gmm} holds in dimension $\leq d-1$.
    \item Conjecture \ref{conj:klt-gmm-effective} holds in dimension $\leq d-1$ and $(X,B)$ is klt.
\end{enumerate}
Then $(X,B)$ has a good minimal model.
\end{thm}
\begin{proof}
(1) is \cite[Theorem 1.3]{Laz26} so we only need to prove (2), and we may assume that $(X,B)$ is klt. Let $\nu:=\nu(K_X+B)$ and $\kappa:=\kappa_{\iota}(K_X+B)$. Let $f: X'\rightarrow Z$ be an invariant Iitaka fibration (cf. \cite[Definition 6.1]{CHL24}) of $K_X+B$. Possibly replacing $X'$ and $Z$ with higher resolutions, we may assume that $Z$ is smooth and the induced birational map $h: X'\dashrightarrow X$ is a log resolution of $(X,B)$. We let $B':=h^{-1}_*B+(1-\delta)\Exc(h)$ for some $0<\delta\ll 1$. Then
$$K_{X'}+B'=h^*(K_X+B)+E$$
for some $E\geq 0$ that is exceptional$/X$. Since $h$ is the invariant Iitaka fibration of $K_X+B$, we have
$\dim Z=\kappa\geq 1$. Since $K_X+B$ is pseudo-effective, $K_{X'}+B'$ is pseudo-effective, and so $K_{X'}+B'$ is pseudo-effective$/Z$. Let $F$ be a general fiber of $f$ and let $K_F+B_F:=(K_{X'}+B')|_F$. Then $(F,B_F)$ is klt of dimension $\leq d-1$ and $\kappa_{\iota}(K_F+B_F)=0$.

By our assumption, $(F,B_F)$ has a good minimal model. Since $\kappa_{\iota}(K_F+B_F)=0$, $\nu(K_F+B_F)=0$. By \cite[V.4.2 Corollary]{Nak04}, $\kappa_{\iota}(K_{X'}+B')=\nu(K_{X'}+B')$. By \cite[Theorem 4.3]{GL13}, $(X',B')$ has a good minimal model. Thus $(X,B)$ has a good minimal model (cf. \cite[Lemma 2.10]{LX25}, \cite[Lemma 2.15]{Has19a}).
\end{proof}

\begin{defn}\label{def:supported-core}
Let $d$ and $\nu$ be two positive integers. A pair $(X,B)$ is called a \emph{supported core} of type $(d,\nu)$ if the following conditions hold.
\begin{enumerate}
\item $(X,B)$ is projective $\mathbb Q$-factorial dlt of dimension $d$.
\item $B=\sum_{i=1}^mS_i$ where $S_i$ are the irreducible components of $B$.
\item $K_X+B\sim_{\mathbb Q}\sum_{i=1}^m a_iS_i$ where $a_i>0$ for each $i$.
\item $K_X+B$ is nef.
\item $\kappa(K_X+B)=0$ and $\nu(K_X+B)=\nu$.
\item $(K_X+B)|_{S_i}$ is semi-ample for any $i$.
\item $\nu((K_X+B)|_{S_i})\leq\nu-1$ for every $i$ and $\nu((K_X+B)|_{S_i})=\nu-1$ for some $i$.
\end{enumerate}
In addition, if $K_X+B\sim \sum_{i=1}^ma_iS_i$, then we say that $(X,B)$ is an \emph{integral supported core}. In this case, $a_i$ is a positive integer for any $i$.
\end{defn}

\begin{lem}\label{lem:adj-num-dim}
Let $X$ be a normal projective variety and $D\not=0$ a nef $\mathbb R$-Cartier $\mathbb R$-divisor on $X$. Write $D=\sum a_iS_i\geq 0$ where $S_i$ are the irreducible components of $D$. Then:
\begin{enumerate}
    \item For any $i$, $\nu(D|_{S_i})\leq \nu(D)-1$.
    \item There exists $i$ such that $\nu(D|_{S_i})=\nu(D)-1$.
\end{enumerate}
\end{lem}
\begin{proof}
Let $H$ be an ample divisor on $X$. Then we have
$$\sum a_i(D|_{S_i})^\nu(H|_{S_i})^{d-1-\nu}=D^{\nu+1}\cdot H^{d-1-\nu}=0$$
so $\nu(D|_{S_i})\leq \nu-1$ for all $i$, which implies (1). We also have
$$\sum a_i(D|_{S_i})^{\nu-1}(H|_{S_i})^{d-\nu}=D^{\nu}\cdot H^{d-\nu}>0$$
so $\nu(D|_{S_i})\geq \nu-1$ for some $i$. Together with (1), this implies (2).
\end{proof}

\begin{thm}\label{thm:real-reduction}
    Let $d\geq \nu\geq 0$ be two integers. Assume Conjecture \ref{conj:lc-gmm} in dimension $\leq d-1$. Let $(X,B)$ be a pair of dimension $d$ such that $\nu(K_X+B)=\nu$. Assume that one of the following conditions hold:
    \begin{enumerate}
        \item $(X,B)$ is klt and $\kappa_{\iota}(K_X+B)\geq 0$.
        \item Conjecture \ref{conj:nv-smooth} holds in dimension $d$.
    \end{enumerate}
    Then either $(X,B)$ has a good minimal model, or there exists a supported core of type $(d,\nu)$ and we have that $2\leq \nu\leq d-1$.
\end{thm}
\begin{proof}
By (1) or (2) and our assumption and Theorem \ref{thm:nv-smooth-to-nv-lc}, we have that $\kappa_{\iota}(K_X+B)\geq 0$. By Theorem \ref{thm:kappa-positive-case}, we may assume that $\kappa_{\iota}(K_X+B)=0$. In particular, $K_X+B$ is not big, so $\nu\leq d-1$. By Conjecture \ref{conj:lc-gmm} in dimension $\leq d-1$ and \cite[Theorem 2.15]{LX25}, \cite[Conjecture 1.7]{LX25} holds in dimension $d$. By \cite[Theorem 4.1]{LX25}, we may assume that $\nu\geq 2$.

By Lemma \ref{lem: klt to special log smooth}, there exists a log smooth lc pair $(X',B')$ such that $K_{X'}+B'\sim_{\mathbb R}D'\geq 0$, $\kappa(D')=0$ and $\nu(D')=\nu$, $\Supp D'\subset\Supp\lfloor B'\rfloor$, and $\Supp D'=\Supp\lfloor B'\rfloor$ if $(X,B)$ is klt. Let $B'':=\Supp B'$, then we have 
$$K_{X'}+B''\sim_{\mathbb R}D'+\Supp B'-B'=D'+\Supp\{B'\}-\{B'\}=:D''\geq 0,$$
$\Supp D''\subset B''$, and $\Supp D''=B''$ if $(X,B)$ is klt.

If $(X',B'')$ has a good minimal model, then by Theorem \ref{thm: mz23 1.4}, $(X',B')$ has a good minimal model, so
$$0=\kappa(D')=\nu(D')=\nu\geq 2,$$
a contradiction. Thus $(X',B'')$ does not have a good minimal model. By Theorem \ref{thm:kappa-positive-case}, $\kappa_{\iota}(K_{X'}+B'')=0$. 

If $\nu(K_{X'}+B'')>\nu$, then $K_{X'}+B'-\epsilon\{B'\}$ is not pseudo-effective for any $0<\epsilon\ll 1$. By Proposition \ref{prop: nv implies eomm}, we may run a $(K_{X'}+B')$-MMP with scaling of an ample divisor which terminates with a minimal model $(Y,B_Y)$ of $(X',B')$. This MMP is also a sequence of steps of a $(K_{X'}+B'-\epsilon\{B'\})$-MMP for any $0<\epsilon\ll 1$, so $K_Y+B_Y-\epsilon\{B_Y\}$ is not pseudo-effective for any $0<\epsilon\ll 1$. Now we pick $0<\epsilon\ll 1$ and run a $(K_Y+B_Y-\epsilon\{B_Y\})$-MMP with scaling of an ample divisor which terminates with a Mori fiber space $f: T\rightarrow Z$. Let $B_T$ be the image of $B_Y$ on $T$. By Proposition \ref{prop: lt22b 2.12}, this MMP is $(K_Y+B_Y)$-trivial, so $\kappa_{\iota}(K_T+B_T)=0$, $\nu(K_T+B_T)=\nu$, $(T,B_T)$ is lc, and $K_T+B_T\sim_{\mathbb R,Z}0$. By \cite[Theorem 7.3 and 7.2.3]{BFMT25}, there exists an lc pair $(Z,B_Z)$ such that $K_T+B_T\sim_{\mathbb R}f^*(K_Z+B_Z)$. By our assumption, 
$$0=\kappa_{\iota}(K_T+B_T)=\kappa_{\iota}(K_Z+B_Z)=\nu(K_Z+B_Z)=\nu(K_T+B_T)=\nu\geq 2,$$
a contradiction. Therefore, $\nu(K_{X'}+B'')=\nu$.

By Proposition \ref{prop: nv implies eomm}, we may run a $(K_{X'}+B'')$-MMP with scaling of an ample divisor which terminates with a minimal model $(V,B_V)$ of $(X',B'')$. Let $D_V$ be the image of $D''$ on $V$. Then we have that $K_V+B_V$ is nef, $B_V$ is reduced, $\Supp D_V\subset B_V$, and $\Supp D_V=B_V$ if $(X,B)$ is klt. Moreover, we have that $\nu(K_V+B_V)=\nu\geq 2$ and $\kappa(K_V+B_V)=0$. Thus $D_V$ is the only element in $|K_V+B_V|_{\mathbb R}$.

Assume that $\Supp D_V\not=B_V$. Then $(X,B)$ is not klt, so assumption (2) holds. Assume that $K_V+(1-\epsilon)B_V$ is pseudo-effective for some $0<\epsilon\ll 1$, then $K_V+(1-\epsilon)B_V\sim_{\mathbb R}L\geq 0$ for some $L$, and so 
$$D_V\sim_{\mathbb R}L+\epsilon B_V\geq 0,$$
hence $L+\epsilon B_V=D_V$. Thus $B_V\subset\Supp D_V\subset B_V$, so $\Supp D_V=B_V$, a contradiction. Thus $K_V+(1-\epsilon)B_V$ is not pseudo-effective for any $0<\epsilon\ll 1$. Pick $0<\epsilon\ll 1$ and run a $(K_V+(1-\epsilon)B_V)$-MMP with scaling of an ample divisor which terminates with a Mori fiber space $f': V'\rightarrow Z'$. Let $B_{V'}$ be the image of $B_V$ on $V'$. By Proposition \ref{prop: lt22b 2.12} the induced birational map $\psi: V\dashrightarrow V'$ is $(K_V+B_V)$-trivial and $K_{V'}+B_{V'}\sim_{\mathbb R,Z'}0$. By  \cite[Theorem 7.3 and 7.2.3]{BFMT25}, there exists an lc pair $(Z',B_{Z'})$ such that
$$K_{V'}+B_{V'}\sim_{\mathbb R}f'^*(K_{Z'}+B_{Z'}),$$
hence 
\begin{align*}
  0&=\kappa_{\iota}(K_V+B_V)=\kappa_{\iota}(K_{V'}+B_{V'})=\kappa_{\iota}(K_{Z'}+B_{Z'})\\
  &=\nu(K_{Z'}+B_{Z'})=\nu(K_{V'}+B_{V'})=\nu(K_V+B_V)=\nu\geq 2,
\end{align*}
a contradiction. Therefore, $\Supp D_V=B_V$.

Write $D_V=\sum a_iS_i$ where $S_i$ be the irreducible components of $B_V$. Since $(V,B_V)$ is $\mathbb Q$-factorial dlt, $(K_V+B_V)|_{S_i}$ is nef hence semi-ample for any $i$.

By Lemma \ref{lem:adj-num-dim}, $(V,B_V)$ is a supported core  of type $(d,\nu)$.
\end{proof}

\begin{prop}\label{prop:integral-model}
Assume Conjecture \ref{conj:lc-gmm} in dimension $\leq d-1$. Let $d$ be a positive integer and let $2\leq\nu\leq d-1$ be an integer. Let $(X,B)$ be a supported core of type $(d,\nu)$ and write 
$$K_X+B\sim_{\mathbb Q}D=\sum a_iS_i\geq 0$$
where $S_i$ are the irreducible components of $B$. Then there exists an integral supported core $(X',B')$ of type $(d,\nu)$.
\end{prop}
\begin{proof}
Let $L:=K_X+B-D$, then $L\sim_{\mathbb Q}0$. We let $\pi: Y\rightarrow X$ be the global index one cover of $L$ and write $K_{Y}+B_Y:=\pi^*(K_X+B)$. Since $\Supp D=\Supp B$, the ramification locus of $\pi$ is contained in $\Supp B$. If $S_i$ is contained in the ramification locus of $\pi$ for some $i$, then for any prime divisor $F$ on $Y$ such that $\pi(F)=S_i$, near the generic point of $F$, by Hurwitz formula, we have
$$K_{Y}=\pi^*K_X+(r-1)F\quad \text{and}\quad \pi^*S_i=rF$$
where $r\geq 1$ is the ramification index of $\pi$ along $F$. Thus 
$$K_Y+F=\pi^*(K_X+B)$$
near the generic point of $F$. Therefore, $B_Y\geq 0$ and $B_Y$ is reduced. By our construction of $\pi$, we have
$$K_Y+B_Y\sim \pi^*D=:D_Y.$$
In particular, $D_Y$ is a Weil divisor. Since $\Supp D=B$, we have that $\Supp D_Y=B_Y$.

Moreover, pick $0<\epsilon\ll 1$, then we have
$$K_Y+B_Y'=\pi^*(K_X+(1-\epsilon)B)$$
for some $B_Y'\geq 0$. Since $(X,(1-\epsilon)B)$ is klt, $(Y,B_Y')$ is klt. Therefore, any lc center of $(Y,B_Y)$ is contained in $B_Y=\Supp D_Y$. We let $h: X'\rightarrow Y$ be a $\mathbb Q$-factorial dlt modification of $(Y,B_Y)$ and let $K_{X'}+B'=h^*(K_Y+B_Y)$. Then we have
$$K_{X'}+B'\sim h^*D_Y=:D'\geq 0.$$
Since $h$ only extracts divisors $E$ such that $\Center_{Y}E\subset B_Y=\Supp D_Y$, we have that
$$\Supp D'=\Exc(h)+\Supp\left(h^{-1}_*D_Y\right)=\Exc(h)+\Supp\left(h^{-1}_*D_Y\right)=B_Y'.$$
Since $K_X+B$ is nef, $K_Y+B_Y$ is nef, so $K_{X'}+B'$ is nef. Moreover, we have
$$\kappa(K_{X'}+B')=\kappa(K_Y+B_Y)=\kappa(K_X+B)=0$$
and
$$\nu(K_{X'}+B')=\nu(K_Y+B_Y)=\nu(K_X+B)=\nu.$$
(cf. \cite[Remark 2.8(2)]{HH20}).  Since Conjecture \ref{conj:lc-gmm} holds in dimension $\leq d-1$, $(K_{X'}+B')|_{S}$ is lc for any irreducible component $S$ of $B'=\lfloor B'\rfloor$, hence $(K_{X'}+B')|_{S}$ is semi-ample for any irreducible component $S$ of $B'=\lfloor B'\rfloor$. By Lemma \ref{lem:adj-num-dim}, $(X',B')$ is a supported core of type $(d,\nu)$. The proposition follows.
\end{proof}

\section{Root stack construction}\label{sec:05-root-neighborhoods}

In this section we introduce the following construction that we will use in the rest of the paper. Startgin with an integral supported core $(X,B)$ with $K_X+B\sim D$, the key of the construction is a cover $p:\mathcal{X}\rightarrow X$ in the category of root stacks of $lD$ for a sufficiently divisible positive integer $l$, whereas the rest conditions are auxiliary.

\begin{consthm}\label{consthm:root-cover}
The data $d,\nu,X,B,D,m,S_1,\dots,S_m,S_B,f,Z,l,A,e_1,\dots,e_m,E$, $s_E$, $N,c,r_1,\dots,r_m,b,U,s,U^{\circ},Y_{0,U},Y_U,t,\xi,\pi_U,\mathcal{X}_0,p,\mathcal{X},\mathcal{S}$
are constructed in the following way. 
\begin{enumerate}
    \item $d$ is a positive integer and $2\leq\nu\leq d-1$ is an integer.
    \item $(X,B)$ is an integral supported core of type $(d,\nu)$.
    \item $K_X+B\sim D=\sum_{i=1}^ma_iS_i\geq 0$.
    \item $(B,S_B)$ is the slc pair induced by adjunction 
    $$K_B+S_B=(K_X+B)|_B.$$
    \item  $(K_X+B)|_B$ is semi-ample (cf. \cite[Theorem 4.3]{FG14}) hence induces a contraction $f: B\rightarrow Z$.
    \item $l$ is a sufficiently divisible positive integer such that 
    $$\mathcal{O}_B(lD|_B)\cong f^*\mathcal{O}_Z(A),$$ 
    where $A$ is a very ample divisor on $Z$.
    \item $E=lD=\sum_{i=1}^me_iS_i$, and $s_E\in H^0(X,\mathcal{O}_X(E))$ is the section with $\operatorname{div}(s_E)=E$. In particular, $e_i=la_i$.
    \item $N$ is a sufficiently divisible positive integer. In particular, $l\mid N$ and $e_i\mid N$.
    \item $c=\frac{N}{l}$, $r_i:=\frac{N}{e_i}\geq 2$ for each $i$, and $b:=\min_i r_i$. In particular, $a_ir_i=c$ for each $i$.
    \item $U$ is a non-empty open subset of $X$ on which $\mathcal{O}_X(E)$ is trivial, and $s\in\mathcal{O}_X(U)$ is a function such that $E|_U=\operatorname{div}(s)$.
    \item $Y_{0,U}:=\Spec_U\mathcal{O}_U[t]/(t^N-s)$, $Y_U$ is the normalization of $Y_{0,U}$, $\pi_U: Y_U\rightarrow U$ is the induced morphism, and $\mu_N$ acts on $Y_{0,U}$ and on $Y_U$ by $t\rightarrow\xi t$, where $\xi=e^{\frac{2\pi i}{N}}$.
    \item $U^{\circ}=U\setminus B$.
    \item $\mathcal{X}_0$ is the $N$-th root stack of $E$ (cf. \cite[Definition 2.2.4]{Cad07}), $\mathcal{X}$ is the normalization of $\mathcal{X}_0$, and $p:\mathcal X\to X$ is the composition $\mathcal{X}\rightarrow\mathcal{X}_0\rightarrow X$.
    \item We have 
    $$\mathcal{X}_0\times_XU=U\times_{s,\mathbb{A}^1}[\mathbb{A}^1/\mu_N]=[Y_{0,U}/\mu_N]\quad\text{and}\quad \mathcal{X}\times_XU=[Y_U/\mu_N],$$
    where $[\mathbb{A}^1/\mu_N]\to\mathbb{A}^1$ is induced by $t\rightarrow t^N$ (cf. \cite[Remark~2.2.3 and Lemma~2.3.1]{Cad07}).
    \item $\mathcal{S}=\Supp p^{-1}(B)$ with the reduced closed substack structure.
\end{enumerate}
\end{consthm}
\begin{proof}
(1--13) and (15) are definitions. In (14), the first equality is \cite[Remark~2.2.3 and Lemma~2.3.1]{Cad07}, the second equality holds since $U\times_{s,\mathbb{A}^1}\mathbb{A}^1=Y_{0,U}$, and the third equality holds since $Y_{0,U}\to[Y_{0,U}/\mu_N]$ is \'etale and normalization commutes with smooth base change \cite[Tag~03GV]{Sta26}.
\end{proof}

\begin{lem}\label{lem:local-structure}
Notation and conditions as in Construction-Theorem~\ref{consthm:root-cover}. Then:
\begin{enumerate}
\item Every irreducible component of $Y_U$ dominates $U$.
\item $t$ is a nonzerodivisor on $Y_U$.
\item For any prime divisor $P$ on $Y_U$,
\begin{enumerate}
    \item if $\pi_U(P)=S_i|_U$ for some $i$, then the ramification index of $\pi_U$ along $P$ is $r_i$ and $\mult_P(\operatorname{div}(t))=1$; in particular, $\pi_U^*(S_i|_U)=r_i\Supp\left(\pi_U^{-1}(S_i)\right)$, and
    \item if $\pi_U(P)\not=S_i|_U$ for any $i$, then $\pi_U$ is unramified along $P$.
\end{enumerate}
\item We have that
$$\operatorname{div}(t)=\sum_{i=1}^m\Supp\left(\pi_U^{-1}(S_i)\right)=\Supp\left(\pi_U^{-1}(B)\right)=\mathcal{S}\times_{\mathcal{X}}Y_U$$ 
is a reduced Cartier divisor on $Y_U$.
\item We have $\pi_U^*(D|_U)=c\operatorname{div}(t)$ and $N\mathcal{S}=p^*E$.
\end{enumerate}
\end{lem}
\begin{proof}
Since $\mathcal{O}_U[t]/(t^N-s)$ is a free $\mathcal{O}_U$-module with basis $1,t,\dots,t^{N-1}$, $Y_{0,U}\to U$ is finite flat, so every irreducible component of $Y_U$ dominates $U$. This implies (1).

Since $\Supp E=B$, $s$ is a unit on $U^{\circ}$. Since $t^N=s$, $t$ and $Nt^{N-1}$ are invertible over $U^{\circ}$, so $Y_{0,U}\rightarrow U$ is \'etale over $U^{\circ}$, hence $Y_U\times_UU^{\circ}=Y_{0,U}\times_UU^{\circ}$ and $\pi_U$ is unramified along any prime divisor $P$ on $Y_U$ such that $\pi_U(P)\not\subset B|_U$. Thus $t$ is a nonzerodivisor on the normal scheme $Y_U$, and $\operatorname{div}(t)$ is an effective Cartier divisor with $\Supp\left(\operatorname{div}(t)\right)=\pi_U^{-1}(B)$. This implies (2) and (3.b).

Let $\eta_i$ be the generic point of $S_i$. Since $(X,B)$ is dlt, $R:=\mathcal{O}_{X,\eta_i}$ is a DVR. Write $s=ux^{e_i}$ with $x$ a uniformizer and $u$ a unit. Over the \'etale extension $R':=R[y]/(y^{e_i}-u)$,
$$t^N-ux^{e_i}=\prod_{\zeta\in\mu_{e_i}}\left(t^{r_i}-\zeta yx\right)$$
is a product of Eisenstein polynomials, so the normalization of $R'[t]/(t^N-ux^{e_i})$ is a product of DVRs, in each of which $t$ is a uniformizer and $x$ has valuation $r_i$. Thus $$\operatorname{ord}_P(t)=1\quad \text{and}\quad \operatorname{ord}_P(\pi_U^*x)=r_i$$ 
for every prime divisor $P$ on $Y_U$ such that $\pi_U(P)=S_i|_U$. This implies (3).

By (3), since $Y_U$ is normal, $\operatorname{div}(t)$ is reduced, and we have
$$\operatorname{div}(t)=\Supp \pi_U^{-1}(B)=\mathcal{S}\times_{\mathcal{X}}Y_U.$$
This implies (4). Since the ideal of $t$ is $\mu_N$-stable, $\mathcal{S}$ is Cartier. Since
$$N\operatorname{div}(t)=N\mathcal{S}|_{Y_U}=\operatorname{div}(\pi_U^*s)=\pi_U^*(E|_U)=l\pi_U^*(D|_U),$$
and since $N=cl$, we obtain (5).
\end{proof}

\begin{setup}\label{setup:root-stack}
Notation and conditions as in Construction-Theorem~\ref{consthm:root-cover}. We further define $\mathcal{R},\mathcal{N},\mathcal{S}_j,w$ as follows.
\begin{enumerate}
    \item $\mathcal{R}:=\mathcal{O}_{\mathcal{X}}(\mathcal{S})$ which is a line bundle as $\mathcal{S}$ is Cartier by Lemma~\ref{lem:local-structure}(4), and we have $\mathcal{R}^{\otimes N}=p^*\mathcal{O}_X(E)$ by Lemma~\ref{lem:local-structure}(5).
    \item $\mathcal{N}:=\mathcal{R}|_{\mathcal{S}}$.
    \item $\mathcal{S}_j$ is the closed substack of $\mathcal{X}$ defined by $\mathcal{O}_{\mathcal{X}}(-j\mathcal{S})$.
    \item $w\in H^0(\mathcal{X},\mathcal{R})$ is the image of $1$ under $\mathcal{O}_{\mathcal{X}}\hookrightarrow\mathcal{O}_{\mathcal{X}}(\mathcal{S})=\mathcal{R}$. By definition we have $\mathcal{S}=(w=0)$ and $\mathcal{S}_j=(w^j=0)$ for any $j$, and $w^N=p^*s_E$ by Lemma~\ref{lem:local-structure}(5).
\end{enumerate}
\end{setup}

\begin{lem}\label{lem:root-cover-basic}
Notation and conditions as in Set-up \ref{setup:root-stack}. Then the following hold.
\begin{enumerate}
\item For any integer $k$, we have $p_*\mathcal{O}_{\mathcal{X}}(k\mathcal{S})=\mathcal{O}_X\left(\left\lfloor\frac{kE}{N}\right\rfloor\right)$ as subsheaves of $\mathcal{K}(X)$, the sheaf of rational functions of $X$.
\item $\mathcal{X}$ is an irreducible normal proper Deligne--Mumford stack, and $p:\mathcal{X}\rightarrow X$ is its coarse moduli space. The functor $p_*$ is exact on quasi-coherent sheaves, and $H^i(\mathcal{X},\mathcal{F})=H^i(X,p_*\mathcal{F})$ for any quasi-coherent sheaf $\mathcal{F}$ on $\mathcal{X}$ and any $i\geq 0$.
\item Let $\Delta:=\sum_{i=1}^m\left(1-\frac{1}{r_i}\right)S_i=B-\frac{1}{N}E$. Then $(X,\Delta)$ is klt, and we have
$$K_{Y_U}=\pi_U^*\left((K_X+\Delta)|_U\right)\quad\text{and}\quad K_{Y_U}+\operatorname{div}(t)=\pi_U^*\left((K_X+B)|_U\right).$$
In particular, $Y_U$ is klt and Gorenstein (hence canonical), and $(Y_U,\operatorname{div}(t))$ is lc. Moreover,
$$\omega_{\mathcal{X}}\cong\mathcal{R}^{c-1},\quad \omega_{\mathcal{S}}\cong\mathcal{N}^{c},\quad \text{and}\quad \omega_{\mathcal{S}_j}\cong\mathcal{R}^{c+j-1}|_{\mathcal{S}_j}$$
for any $j$.
\item $\mathcal{S}$ is reduced, of pure dimension $d-1$, Gorenstein, and Du Bois, and $\mathcal{S}_j$ is Gorenstein of pure dimension $d-1$ for any $j\geq 1$.
\end{enumerate}
\end{lem}
\begin{proof}
Let $U$ be as in Construction-Theorem~\ref{consthm:root-cover}(10). By Construction-Theorem~\ref{consthm:root-cover}(14), $\mathcal{X}\times_XU=[Y_U/\mu_N]$, so a quasi-coherent sheaf $\mathcal{F}$ on $\mathcal{X}\times_XU$ is the same as a $\mu_N$-equivariant quasi-coherent sheaf $F$ on $Y_U$ \cite[Tag~06WT]{Sta26}, and 
\begin{equation}\label{equ:pstar}
    (p_*\mathcal{F})|_U=(\pi_{U,*}F)^{\mu_N}
\end{equation}
is the subsheaf of $\mu_N$-invariant sections of $\pi_{U,*}F$ (cf. \cite[Remark~2.1(3)]{AOV08}). For any open subset $V\subset U$ and $Y_V:=\pi_U^{-1}(V)$, we have $$\Gamma(V,p_*\mathcal{F})=\Gamma([Y_V/\mu_N],\mathcal{F})=\Gamma(Y_V,F)^{\mu_N}.$$
We prove (1). Let $\eta$ be the generic point of $U$, let $\mathcal{A}:=\pi_{U,*}\mathcal{O}_{Y_U}$, which is a coherent $\mathcal{O}_U$-algebra as $\pi_U$ is finite, and let $\mathcal{K}(Y_U)$ be the sheaf of rational functions of $Y_U$, i.e. the sheaf of total quotient rings of $\mathcal{O}_{Y_U}$. Since $\pi_U$ is finite and every irreducible component of $Y_U$ dominates $U$ (Lemma~\ref{lem:local-structure}(1)), the points of $Y_U$ over $\eta$ are exactly the generic points of $Y_U$. Thus $$\mathcal{A}_\eta=H^0(Y_U,\mathcal{K}(Y_U)).$$ 
Since $Y_U$ is the normalization of $Y_{0,U}$ and normalization commutes with localization, $\Spec\mathcal{A}_\eta$ is the normalization of the generic fiber $\Spec\mathbb{C}(U)[t]/(t^N-s)$ of $Y_{0,U}\rightarrow U$. As $s\not=0$, $t^N-s$ is a separable polynomial over $\mathbb{C}(U)$, so $\mathbb{C}(U)[t]/(t^N-s)$ is a product of fields and hence normal. Thus
$$H^0(Y_U,\mathcal{K}(Y_U))=\mathcal{A}_\eta=\mathbb{C}(U)[t]/(t^N-s)$$ 
with a $\mu_N$-action $t\rightarrow\xi t$. We have that the $\mu_N$-invariant rational functions on $Y_U$ are exactly the rational functions on $U$. 

Let $a\in\mathbb{C}(U)$. By Lemma~\ref{lem:local-structure}(3), for any prime divisor $P$ on $Y_U,$ we have $\operatorname{ord}_P(\pi_U^*a)=r_i\operatorname{ord}_{S_i}(a)$ if $\pi_U(P)=S_i|_U$ and $\operatorname{ord}_P(\pi_U^*a)=\operatorname{ord}_{\pi_U(P)}(a)$ otherwise. Since $\operatorname{div}(t)=\mathcal{S}\times_{\mathcal{X}}Y_U$ by Lemma~\ref{lem:local-structure}(4), $p_*\mathcal{O}_{\mathcal{X}}(k\mathcal{S})|_U$ is the sheaf of $\mu_N$-invariant sections of $\mathcal{O}_{Y_U}(k\operatorname{div}(t))$, i.e. the sheaf of rational functions $a$ on $U$ such that $\operatorname{ord}_P(\pi_U^*a)\geq -k\mult_P(\operatorname{div}(t))$ for every prime divisor $P$ on $Y_U$. By Lemma~\ref{lem:local-structure}(4), $\operatorname{div}(t)=\sum_{i=1}^m\Supp\left(\pi_U^{-1}(S_i)\right)$ is reduced, so $\operatorname{ord}_P(\pi_U^*a)\geq -k\mult_P(\operatorname{div}(t))$ is equivalent to
$$\operatorname{ord}_{S_i}(a)\geq -\frac{k}{r_i}=-\frac{ke_i}{N},\quad \operatorname{ord}_{P}(a)\geq 0\quad \text{if}\quad P\not=S_i\quad \text{for every } i$$
and is equivalent to $\operatorname{ord}_{S_i}(a)\geq -\left\lfloor\frac{ke_i}{N}\right\rfloor$, $\operatorname{ord}_{P}(a)\geq 0\quad \text{if}\quad P\not=S_i$ for every $i$, and is equivalent to $a\in\mathcal{O}_X\left(\left\lfloor\frac{kE}{N}\right\rfloor\right)$. This implies (1).

We prove (2). By \cite[Theorem~2.3.3]{Cad07}, $\mathcal{X}_0$ is a Deligne--Mumford stack. Since $\mathcal{X}\rightarrow\mathcal{X}_0$ is finite, $\mathcal{X}$ is a Deligne--Mumford stack, and $\mathcal{X}$ is normal by construction. We have that $p$ is proper as $\pi_U$ is finite for any $U$ as in Construction-Theorem~\ref{consthm:root-cover}(10).

Let $V\subset U$ be an affine open subset, so that $Y_V=\Spec A$ is affine and $\mathcal{X}\times_XV=[Y_V/\mu_N]$ associated with quotient morphism $q_V: Y_V\rightarrow [Y_V/\mu_N]$ by Construction-Theorem~\ref{consthm:root-cover}(14). By (1) with $k=0$ and \eqref{equ:pstar}, $\pi_U^*:H^0(V,\mathcal{O}_X)\rightarrow A^{\mu_N}$ is an isomorphism. Since $Y_V\rightarrow[Y_V/\mu_N]$ is finite, flat and surjective, by \cite[Tags~0DUP and~0DUR]{Sta26} there is a morphism $f:[Y_V/\mu_N]\rightarrow M$ to an affine scheme $M$ which is a universal homeomorphism and is universal among morphisms from $[Y_V/\mu_N]$ to algebraic spaces. By \cite[Tag~0DUH]{Sta26}, a morphism from $[Y_V/\mu_N]$ to a scheme is the same as a $\mu_N$-invariant morphism from $Y_V$. We have
$$H^0(M,\mathcal{O}_M)=\operatorname{Hom}(M,\mathbb{A}^1)=\operatorname{Hom}([Y_V/\mu_N],\mathbb{A}^1)=A^{\mu_N}.$$
Hence the morphism $[Y_V/\mu_N]\rightarrow V$ induced by $p$ factors as $f$ followed by a morphism $M\rightarrow V$ of affine schemes which induces 
$$\pi_U^*:H^0(V,\mathcal{O}_X)\rightarrow H^0(M,\mathcal{O}_M)=A^{\mu_N}.$$
Thus $M\rightarrow V$ is an isomorphism. Therefore $\mathcal{X}\times_XV\rightarrow V$ is a universal homeomorphism and is universal among morphisms to algebraic spaces, and it induces a bijection on closed points by \cite[Tag~03BL]{Sta26}. Since $X$ is covered by such $V$ and these properties are local on $X$ (cf. \cite[Tag~0DUJ]{Sta26}), $p$ is a universal homeomorphism and $X$ is the coarse moduli space of $\mathcal{X}$. In particular, $\mathcal{X}$ is irreducible.

Let $\mathcal{F}$ be a quasi-coherent sheaf on $\mathcal{X}$. By \eqref{equ:pstar}, $p_*$ is an exact functor as $\mu_N$ is a finite group. We have that $q_V$ is finite \'etale of degree $N$. The trace map of $q_{V,*}\mathcal{O}_{Y_V}$ as a $\mathcal{O}_{\mathcal{X}\times_XV}$-algebra satisfies $\operatorname{tr}(1)=N$, so $\frac{1}{N}\operatorname{tr}$ splits the natural map $$\mathcal{O}_{\mathcal{X}\times_XV}\rightarrow q_{V,*}\mathcal{O}_{Y_V},$$
hence $\mathcal{F}|_{\mathcal{X}\times_XV}$ is a direct summand of $q_{V,*}q_V^*\mathcal{F}=\mathcal{F}\otimes q_{V,*}\mathcal{O}_{Y_V}$. Since $q_V$ is affine, $R^iq_{V,*}(q_V^*\mathcal{F})=0$ for any $i>0$ \cite[Tag~0GQI]{Sta26}, so $$H^i(\mathcal{X}\times_XV,q_{V,*}q_V^*\mathcal{F})=H^i(Y_V,q_V^*\mathcal{F})$$ 
by the Leray spectral sequence. Since $Y_V$ is affine, $H^i(\mathcal{X}\times_XV,\mathcal{F})=0$ for any $i>0$. Since all such $V$ cover $X$, $R^ip_*\mathcal{F}=0$ for any $i>0$, and the Leray spectral sequence implies $H^i(\mathcal{X},\mathcal{F})=H^i(X,p_*\mathcal{F})$ which implies (2).

Since $K_X+B\sim D$, we may choose a rational top form $\eta$ on $X$ with $\operatorname{div}(\eta)=D-B$, and we let $K_X:=\operatorname{div}(\eta)$ and $K_{Y_U}:=\operatorname{div}(\pi_U^*\eta)$. By Lemma~\ref{lem:local-structure}(3)(4) and the Hurwitz formula, we have that
$$K_{Y_U}=\pi_U^*\left((K_X+\Delta)|_U\right)\quad\text{and}\quad K_{Y_U}+\operatorname{div}(t)=\pi_U^*\left((K_X+B)|_U\right)\sim \pi_U^*D|_U=c\operatorname{div}(t),$$
hence $K_{Y_U}\sim (c-1)\operatorname{div}(t)$ is Cartier. By \cite[Proposition 5.20]{KM98} applied to each irreducible component of $Y_U$, $Y_U$ is klt and $(Y_U,\operatorname{div}(t))$ is lc. Thus $Y_U$ is Gorenstein and canonical. Since $\pi_U^*\eta$ is $\mu_N$-invariant with $\operatorname{div}(\pi_U^*\eta)=(c-1)\operatorname{div}(t)$, multiplication by $\pi_U^*\eta$ inudces a $\mu_N$-equivariant isomorphism $\mathcal{O}_{Y_U}((c-1)\operatorname{div}(t))\cong\omega_{Y_U}$. Since $\eta$ is defined on $X$, these isomorphisms are compatible for different $U$ and we obtain
$$\omega_{\mathcal{X}}\cong\mathcal{O}_{\mathcal{X}}((c-1)\mathcal{S})=\mathcal{R}^{c-1}.$$
By adjunction, we have
$$\omega_{\mathcal{S}_j}\cong\left(\omega_{\mathcal{X}}\otimes\mathcal{O}_{\mathcal{X}}(j\mathcal{S})\right)|_{\mathcal{S}_j}=\mathcal{R}^{c+j-1}|_{\mathcal{S}_j}.$$
Let $j=1$, we have
$$\mathcal{N}^c=\mathcal{R}^{c}|_{\mathcal{S}_1}\cong\omega_{\mathcal{S}}.$$
This implies (3). 

Finally, we prove (4). $\mathcal{S}$ is reduced by Lemma~\ref{lem:local-structure}(4). $Y_U$ is klt hence Cohen-Macaulay. We have that $t^j$ is a nonzerodivisor, so
$$V(t^j)=\mathcal{S}_j|_{Y_U}$$
is Cohen-Macaulay of pure dimension $d-1$. Since $Y_U$ is Gorenstein, $\mathcal{S}_j|_{Y_U}$ is Gorenstein. Since $\mathcal{S}|_{Y_U}=\operatorname{div}(t)$ is the boundary of the lc pair  $(Y_U,\operatorname{div}(t))$, we have that $\mathcal{S}|_{Y_U}$ is Du Bois \cite[Corollary~6.32]{Kol13} (cf. \cite[Theorem~1.4]{KK10}). We may pick a cover of $X$ using $U$ as in Construction-Theorem \ref{consthm:root-cover}(10). Since Gorenstein and Du Bois are local properties, $\mathcal{S}$ is Du Bois, and $\mathcal{S}_j$ is Gorenstein for any $j$. In particular, $\mathcal{S}$ is Gorenstein. In particular, this implies (4).
\end{proof}

\begin{consthm}\label{consthm:jj}
    Notation and conditions as in Set-up \ref{setup:root-stack}. We further define $Q_j,\iota_j,\iota_B,J_j,\overline{J}_j,C_j$ as follows.
    \begin{enumerate}
        \item For any positive integer $j$, we let $Q_j$ be the closed subscheme of $X$ defined by $\mathcal{O}_X\left(-\left\lceil\frac{jE}{N}\right\rceil\right).$
        \item For any positive integer $j$, we let $\iota_j:\mathcal{S}_j\rightarrow\mathcal{X}$ be the closed immersion. We let $\iota_B: B\rightarrow X$ be the closed immersion.
        \item For any positive integer $j$, we define
        $J_j:=\mathcal{O}_X\left(-\left\lceil\frac{jE}{N}\right\rceil\right)\Big/\mathcal{O}_X\left(-\left\lceil\frac{(j+1)E}{N}\right\rceil\right).$
        \item We have that $\mathcal{O}_X(-B)\cdot J_j=0$, so there is a unique coherent sheaf $\overline{J}_j$ on $B$ such that $\iota_{B,*}\overline{J}_j=J_j$.
        \item We define $C_j:=f_*\overline{J}_j$.
        \item We have that $$(p\circ\iota_j)_*\mathcal{O}_{\mathcal{S}_j}=\mathcal{O}_{Q_j}\quad\text{and}\quad(p\circ\iota_1)_*\mathcal{N}^{-j}=J_j$$ as quotients of $\mathcal{O}_X=p_*\mathcal{O}_{\mathcal{X}}$ and of $\mathcal{O}_X\left(-\left\lceil\frac{jE}{N}\right\rceil\right)=p_*\mathcal{O}_{\mathcal{X}}(-j\mathcal{S})$ respectively.
    \end{enumerate}
\end{consthm}
\begin{proof}
Only (4) and (6) require proof. By Lemma~\ref{lem:root-cover-basic}(1), the inclusion $\mathcal{O}_{\mathcal{X}}(-(j+1)\mathcal{S})\hookrightarrow\mathcal{O}_{\mathcal{X}}(-j\mathcal{S})$ under the pushforward by $p$ is the inclusion $$\mathcal{O}_X\left(-\left\lceil\frac{(j+1)E}{N}\right\rceil\right)\hookrightarrow\mathcal{O}_X\left(-\left\lceil\frac{jE}{N}\right\rceil\right)$$ of subsheaves of $\mathcal{K}(X)$. By Lemma~\ref{lem:root-cover-basic}(2), $p_*$ is exact. Consider the exact sequences
$$0\rightarrow\mathcal{O}_{\mathcal{X}}(-j\mathcal{S})\rightarrow\mathcal{O}_{\mathcal{X}}\rightarrow\iota_{j,*}\mathcal{O}_{\mathcal{S}_j}\rightarrow 0$$
and
$$0\rightarrow\mathcal{O}_{\mathcal{X}}(-(j+1)\mathcal{S})\rightarrow\mathcal{O}_{\mathcal{X}}(-j\mathcal{S})\rightarrow\iota_{1,*}\mathcal{N}^{-j}\rightarrow 0,$$
we obtain that
$$(p\circ\iota_j)_*\mathcal{O}_{\mathcal{S}_j}=p_*(\iota_{j*}\mathcal{O}_{\mathcal{S}_j})=\mathcal{O}_X/\mathcal{O}_X\left(-\left\lceil\frac{jE}{N}\right\rceil\right)=\mathcal{O}_{Q_j}$$ 
and 
$$(p\circ\iota_1)_*\mathcal{N}^{-j}=p_*\left(\iota_{1,*}\mathcal{N}^{-j}\right)=J_j.$$ 
This implies (6). Since $0<\frac{e_i}{N}\leq 1$, we have 
$$\left\lceil\frac{(j+1)e_i}{N}\right\rceil\leq\left\lceil\frac{je_i}{N}\right\rceil+1$$ for each $i$, and so 
$$\mathcal{O}_X(-B)\cdot\mathcal{O}_X\left(-\left\lceil\frac{jE}{N}\right\rceil\right)\subset\mathcal{O}_X\left(-\left\lceil\frac{(j+1)E}{N}\right\rceil\right).$$ 
Thus $\mathcal{O}_X(-B)\cdot J_j=0$ and we are done.
\end{proof}

\begin{lem}\label{lem:jj}
Notation and conditions as in Construction-Theorem \ref{consthm:jj}. Then:
\begin{enumerate}
\item $C_j$ is a coherent sheaf on $Z$.
\item For any $j$, we have
$$J_{j+N}\cong J_j\otimes\mathcal{O}_X(-E),\quad \overline{J}_{j+N}\cong\overline{J}_j\otimes f^*\mathcal{O}_Z(-A),\quad \text{and}\quad C_{j+N}\cong C_j\otimes\mathcal{O}_Z(-A),$$
where we recall that $N$ is defined as in Construction-Theorem \ref{consthm:root-cover}(8).
\item We have that $Q_j=B$ for any $1\leq j\leq b$ and $Q_{kN}=kE$ for any positive integer $k$. Here we recall that $b=\min_ir_i$ is defined in Construction-Theorem \ref{consthm:root-cover}(9).
\end{enumerate}
\end{lem}
\begin{proof}
(1) is obvious as $C_j$ is the pushforward of a coherent sheaf from $B$ under a proper morphism. Since 
$$\left\lceil\frac{(j+N)E}{N}\right\rceil=\left\lceil\frac{jE}{N}\right\rceil+E,$$
(2) follows from the definition of $J_j$ (Construction-Theorem \ref{consthm:jj}(3)), Construction-Theorem \ref{consthm:root-cover}(6), and the projection formula. Since $\left\lceil\frac{je_i}{N}\right\rceil=1$ if and only if $1\leq j\leq r_i$ and since $\left\lceil\frac{kNE}{N}\right\rceil=kE$, we obtain (3).
\end{proof}

\begin{thm}\label{thm:root-cover}
Notation and conditions as in Construction-Theorem \ref{consthm:jj}.
\begin{enumerate}
    \item For any integer $1\leq j\leq b$, the morphism $p\circ\iota_j$ factors through $B$, i.e. there exists an induced morphism $p_j:\mathcal{S}_j\rightarrow B$ such that $\iota_{B}\circ p_j=p\circ\iota_j$.
    \item For any integer $1\leq j\leq b$,
    $$\mathcal{R}^N|_{\mathcal{S}_j}\cong (f\circ p_j)^*\mathcal{O}_Z(A).$$
    \item For any positive integer $a$, we have
    $$\mathcal{R}^{Na}=p^*\mathcal{O}_X(aE).$$
    \item For any positive integers $j$ and $a$, there is an isomorphism
$$H^0(\mathcal{S}_j,\mathcal{R}^{Na}|_{\mathcal{S}_j})\cong  H^0(Q_j,\mathcal{O}_X(aE)|_{Q_j}).$$
\item For any integer $1\leq j\leq b$, $$H^0(\mathcal{S}_j,\mathcal{R}^{Na}|_{\mathcal{S}_j})=H^0(Z,\mathcal{O}_Z(aA)).$$
\end{enumerate}
\end{thm}
\begin{proof}
(1) Let $V\subset U$ be any open subset of $U$ and let $g$ be a section of $\mathcal{O}_X(-B)|_V$.  It suffices to show that $(\pi_U|_V)^*g\in t^j\mathcal{O}_{Y_U}(\pi_U^{-1}(V))$. By Lemma~\ref{lem:local-structure}(3), for any prime divisor $P$ on $Y_U|_{\pi^{-1}_U(V)}$, we have that
$$\ord_{P}(\pi_U|_V)^*g=r_i\geq b\geq j$$
if $P$ dominates the generic point of $S_i$ for some $i$, and $\ord_{P}(\pi_U|_V)^*g=0$ otherwise. Thus 
$$\frac{(\pi_U|_V)^*g}{t^j}$$ 
is regular on $Y_U$. By considering all $U$ as in Construction-Theorem \ref{consthm:root-cover}, we have
$$\mathcal{O}_{\mathcal{X}}(-B)\cdot\mathcal{O}_{\mathcal{X}}\subset\mathcal{O}_{\mathcal{X}}(-j\mathcal{S}),$$
which is equivalent to say that $p\circ\iota_j$ factors through $B$. This implies (1).

By Set-up~\ref{setup:root-stack}(1), $\mathcal{R}^{\otimes N}=p^*\mathcal{O}_X(E)$, hence $\mathcal{R}^{Na}=p^*\mathcal{O}_X(aE)$ for any positive integer $a$. This implies (3). For $1\leq j\leq b$, by (1) and Construction-Theorem~\ref{consthm:root-cover}(6),
$$\mathcal{R}^N|_{\mathcal{S}_j}=\iota_j^*p^*\mathcal{O}_X(E)=p_j^*\iota_B^*\mathcal{O}_X(E)=p_j^*\left(\mathcal{O}_X(E)|_B\right)\cong p_j^*f^*\mathcal{O}_Z(A)=(f\circ p_j)^*\mathcal{O}_Z(A),$$
which implies (2). By (3), $\mathcal{R}^{Na}|_{\mathcal{S}_j}=(p\circ\iota_j)^*\mathcal{O}_X(aE)$. Since $\mathcal{O}_X(aE)$ is a line bundle, the projection formula and Construction-Theorem~\ref{consthm:jj}(6) imply
$$(p\circ\iota_j)_*\left(\mathcal{R}^{Na}|_{\mathcal{S}_j}\right)=\mathcal{O}_X(aE)\otimes(p\circ\iota_j)_*\mathcal{O}_{\mathcal{S}_j}=\mathcal{O}_X(aE)\otimes\mathcal{O}_{Q_j}=\mathcal{O}_X(aE)|_{Q_j}.$$
(4) follows by taking global sections.

For any $1\leq j\leq b$, by Lemma~\ref{lem:jj}(3), $Q_j=B$, and by Construction-Theorem~\ref{consthm:root-cover}(6), $\mathcal{O}_X(aE)|_B\cong f^*\mathcal{O}_Z(aA)$. By the projection formula,
$$H^0(Q_j,\mathcal{O}_X(aE)|_{Q_j})=H^0(B,f^*\mathcal{O}_Z(aA))=H^0(Z,\mathcal{O}_Z(aA)),$$
and (5) follows from (4).
\end{proof}

We now fix the target of the lifting problem studied in Sections~\ref{sec:06-filtered-comparison}--\ref{sec:08-framed-lifting}.

\begin{setup}\label{setup:framed-map}
Notation and conditions as in Construction-Theorem~\ref{consthm:jj}. We further define $a,r,x_0,\dots,x_r,\iota_Z,\mathcal{V},\rho,\mathcal{T},\Tot(\mathcal{T}),\pi,z,\theta,p_j,h$ in the following way.
\begin{enumerate}
\item $a$ is a positive integer.
\item $x_0,\dots,x_r$ is a basis of $H^0(Z,\mathcal{O}_Z(aA))$ and $\iota_Z: Z\rightarrow\mathbb P^r$ is the closed embedding defined by this basis. We identify $Z$ with a closed subscheme of $\mathbb P^r$, and we have
$$\mathcal{O}_{\mathbb{P}^r}(1)|_Z=\mathcal{O}_Z(aA).$$
\item $\mathcal{V}$ is the $Na$-th root stack of $\mathcal{O}_{\mathbb{P}^r}(1)$ (cf. \cite[Definition 2.2.6]{Cad07}) and $\rho: \mathcal{V}\rightarrow\mathbb P^r$ is the associated morphism.
\item $\mathcal{T}$ is the universal $Na$-th root of $\mathcal{O}_{\mathbb{P}^r}(1)$ on $\mathcal{V}$. In particular, $\mathcal{T}^{Na}=\rho^*\mathcal{O}_{\mathbb{P}^r}(1)$.
\item $\Tot(\mathcal{T}):=\Spec_{\mathcal{V}}\operatorname{Sym}^{\bullet}(\mathcal{T}^{\vee})$ is the total space of $\mathcal{T}$ with projection $\pi:\Tot(\mathcal{T})\rightarrow\mathcal{V}$, zero section $z:\mathcal{V}\rightarrow\Tot(\mathcal{T})$ (cf. \cite[Tag~0B74]{Sta26}), and tautological section $\theta\in H^0(\Tot(\mathcal{T}),\pi^*\mathcal{T})$ (cf. \cite[Tag~0F96]{Sta26}).
\item  For any integer $1\leq j\leq b$ we define $p_j: \mathcal{S}_j\rightarrow B$ be the morphism as in Theorem \ref{thm:root-cover}(1).
\item $h:\mathcal{S}\rightarrow\mathcal{V}$ is the morphism generated by $\mathcal{N}$ and 
$$(f\circ p_1)^*x_0,\dots,(f\circ p_1)^*x_r\in (f\circ p_1)^*\mathcal{O}_Z(aA)\cong\mathcal{R}^{Na}|_{\mathcal{S}}=\mathcal{N}^{Na}$$ 
(Theorem \ref{thm:root-cover}(2)(3)) by the universal property of the root stack $\mathcal{V}$ (cf. \cite[Example 2.4.3]{Cad07}).
\end{enumerate}
\end{setup}

\begin{defn}\label{def:lift-coordinates}
Notation and conditions as in Set-up~\ref{setup:framed-map}. Let $j\geq b$ be an integer.  A \emph{$j$-order lift} is a tuple $\sss=(s_0,\dots,s_r)$ such that for any $i$,
$$s_i\in H^0\left(\mathcal{S}_j,\mathcal{R}^{Na}|_{\mathcal{S}_j}\right) \quad \text{and}\quad  s_i|_{\mathcal{S}_b}=(f\circ p_b)^*x_i.$$
\end{defn}

\begin{setup}\label{setup:final-setup}
Notation and conditions as in Set-up~\ref{setup:framed-map}. We define $h_{j,\sss}$ and $\psi_{j,\sss}$ for any $j\geq b$ and $j$-order lift $\sss=(s_0,\dots,s_r)$ in the following way.
\begin{enumerate}
\item We define $h_{j,\sss}: \mathcal{S}_j\rightarrow\mathcal{V}$ be the morphism generated by $\mathcal{R}|_{\mathcal{S}_j}$ and $s_0,\dots,s_r$, by the universal property of root stacks (cf. \cite[Example 2.4.3]{Cad07}). In particular, $h_{j,\sss}|_{\Ss}=h$ for any $j,\sss$.
\item We define $\psi_{j,\sss}: \mathcal{S}_j\rightarrow\Tot(\mathcal{T})$ to be the unique morphism such that
$$\pi\circ\psi_{j,\sss}=h_{j,\sss}\quad \text{and}\quad \psi_{j,\sss}^*\theta=w|_{\mathcal{S}_j}$$
where $w|_{\mathcal{S}_j}$ is considered as a section of $h_{j,\sss}^*\mathcal{T}=\mathcal{R}|_{\mathcal{S}_j}$. Existence and uniqueness of $\psi_{j,\sss}$ follows from the definition and the universal property of $\Tot(\mathcal{T})$ (cf. \cite[Tag~01LV]{Sta26}).
\end{enumerate}
\end{setup}

\begin{lem}\label{lem:root-neighborhoods-framed-projectivity}
Notation and conditions as in Set-up~\ref{setup:final-setup}. Let $j\geq b$ be an integer and let $\sss$ be a $j$-order lift. Then $h,h_{j,\sss}$, and $\psi_{j,\sss}$ are projective.
\end{lem}
\begin{proof}
We first show that $h_{j,\sss}:\mathcal{S}_j\rightarrow\mathcal{V}$ is representable by algebraic spaces. By \cite[Tag~04YY]{Sta26}, and since we only consider structures of finite type over $\mathbb C$, we only need to show that 
$$\Aut(y)=\Isom_{\mathcal{S}_j}(y,y)\rightarrow\Isom_{\mathcal{V}}(h_{j,\sss}(y),h_{j,\sss}(y))=\Aut(h_{j,\sss}(y))$$
is injective for any closed point $y$, considered as a morphism $y:\Spec\mathbb C\rightarrow\mathcal{S}_j$.

Let $y$ be a closed point of $\mathcal{S}_j$ such that the image of $y$ in $X$ is contained in $U$. By Construction-Theorem~\ref{consthm:root-cover}(14), $\operatorname{Aut}(y)$ is the stabilizer $G_y\subset\mu_N$ of $y$ in $Y_U$. By Lemma~\ref{lem:local-structure}(4), $\mathcal{R}|_{Y_U}=t^{-1}\mathcal{O}_{Y_U}$, and any $\zeta\in\mu_N$ acts on $t^{-1}$ by $\zeta^{-1}$ (Construction-Theorem~\ref{consthm:root-cover}(11)). Since $t^{-1}$ is a local generator of $\mathcal{R}|_{Y_U}$, its value at $y$ spans the fiber $\mathcal{R}(y):=y^*(\mathcal{R}|_{\mathcal{S}_j})$ (recall we consider $y$ as a morphism $y:\Spec\mathbb C\rightarrow\mathcal{S}_j$), and any $\zeta\in G_y$ acts on $\mathcal{R}(y)$ by multiplication by $\zeta^{-1}$. Thus $G_y$ acts faithfully on the fiber $\mathcal{R}(y)$. Since $\mathcal{R}|_{\mathcal{S}_j}=h_{j,\sss}^*\mathcal{T}$, we have
$$\mathcal{R}(y)=y^*(\mathcal{R}|_{\mathcal{S}_j})=y^*h_{j,\sss}^*\mathcal{T}=(h_{j,\sss}\circ y)^*\mathcal{T}.$$
Let $v:=h_{j,\sss}\circ y$. Then $\mathcal{R}(y)=v^*\mathcal{T}$, and any $\zeta\in\Aut(y)$ acts on $v^*\mathcal{T}$ as $h_{j,\sss}\circ\zeta\in\Aut(v)$ acts on $v^*\mathcal{T}$. Since $\Aut(y)$ acts faithfully on $\mathcal{R}(y)$, the map $$\Aut(y)\rightarrow\Aut(v),\quad \zeta\rightarrow h_{j,\sss}\circ\zeta$$ 
is injective. Therefore, $h_{j,\sss}$ is representable by algebraic spaces.

Let $T$ be a quasi-compact and quasi-separated scheme with a morphism $g: T\rightarrow\mathcal{V}$, and let $\mathcal{S}_{j,T}:=\mathcal{S}_j\times_{\mathcal{V}}T$. Since $h_{j,\sss}$ is representable by algebraic spaces, $\mathcal{S}_{j,T}$ is an algebraic space. The natural morphism $\mathcal{S}_{j,T}\rightarrow\mathcal{S}_j\times T$ is the base change of the diagonal morphism $\Delta:\mathcal{V}\rightarrow\mathcal{V}\times\mathcal{V}$ along $(h_{j,\sss},g):\mathcal{S}_j\times T\rightarrow\mathcal{V}\times\mathcal{V}$. Since $\Delta$ is finite, $\mathcal{S}_{j,T}\rightarrow\mathcal{S}_j\times T$ is finite. The morphism $\mathcal{S}_j\times T\rightarrow X\times T$ is the base change of $p\circ\iota_j:\mathcal{S}_j\rightarrow X$, which is proper (Lemma~\ref{lem:root-cover-basic}(2)) and has finite fibers, since $\mathcal{X}\times_XU=[Y_U/\mu_N]$ and $Y_U\rightarrow U$ is finite (Construction-Theorem~\ref{consthm:root-cover}(11)(14)); hence it is proper with finite fibers as well. Thus the composition
$$q:\mathcal{S}_{j,T}\rightarrow\mathcal{S}_j\times T\rightarrow X\times T$$
is proper with finite fibers. For any field $k$ and any morphism $\Spec k\rightarrow X\times T$, the fiber $\mathcal{S}_{j,T}\times_{X\times T}\Spec k$ is an algebraic space of finite type over $k$ with finitely many points, hence its underlying topological space is discrete. Since $q$ is a proper morphism of algebraic spaces, $q$ is finite \cite[Tag~0A4X]{Sta26}. By the definition of finite morphisms of algebraic spaces \cite[Tag~03ZP]{Sta26}, $q^{-1}(W)$ is an affine scheme for any affine open subscheme $W\subset X\times T$. By \cite[Tag~03JH]{Sta26}, hence $\mathcal{S}_{j,T}$ is a scheme and the induced morphism $\mathcal{S}_{j,T}\rightarrow T$, as the composition of the projection $X\times T\rightarrow T$ and $q$, is projective. Thus $h_{j,\sss}$ is projective, so $h=h_{j,\sss}|_{\mathcal{S}}$ is projective. 

Since $h_{j,\sss}$ is projective, the base change $\mathcal{S}_j\times_{\mathcal{V}}\Tot(\mathcal{T})\rightarrow\Tot(\mathcal{T})$ of $h_{j,\sss}$ is projective. We have that $\psi_{j,\sss}$ is the composition of the closed immersion
$$\mathcal{S}_j\rightarrow\mathcal{S}_j\times_{\mathcal{V}}\Tot(\mathcal{T})=\Tot(\mathcal{R}|_{\mathcal{S}_j})$$
given by $w|_{\mathcal{S}_j}$ with $\mathcal{S}_j\times_{\mathcal{V}}\Tot(\mathcal{T})\rightarrow\Tot(\mathcal{T})$, so $\psi_{j,\sss}$ is projective.
\end{proof}

We now define the obstruction to extending a $j$-order lift to a $(j+1)$-order lift, and we compare it with the obstruction to extending the morphism $h_{j,\sss}$. We first fix the terminology.

\begin{defn}\label{defn:atiyah}
Let $\mathcal{X}$ be a smooth Deligne--Mumford stack and let $\mathcal{L}$ be a line bundle on $\mathcal{X}$. The \emph{Atiyah bundle} $\At(\mathcal{L})$ of $\mathcal{L}$ is the subsheaf of $\mathcal{H}om_{\mathbb{C}}(\mathcal{L},\mathcal{L})$ defined as follows. For any \'etale chart $U\rightarrow\mathcal{X}$, $\At(\mathcal{L})(U)$ consists of the $\mathbb{C}$-linear maps of sheaves $D: \mathcal{L}|_U\rightarrow\mathcal{L}|_U$ such that there exists $\xi\in T_{\mathcal{X}}(U)$ so that for any open subset $U'\subset U$, $f\in\mathcal{O}_{\mathcal{X}}(U')$ and $s\in\mathcal{L}(U')$, we have
$$D(fs)=fD(s)+\xi(f)s.$$
$\At(\mathcal{L})$ is an $\mathcal{O}_{\mathcal{X}}$-module via $(g\cdot D)(s):=g\,D(s)$ for $g\in\mathcal{O}_{\mathcal{X}}(U)$. The map 
$$\At(\mathcal{L})(U)\rightarrow T_{\mathcal{X}}(U):\quad D\rightarrow\xi$$ is called the \emph{symbol map} over $U$ and is denote by $\sym_{\mathcal{L}}|_U$, which induces the \emph{symbol map} $\sym_{\mathcal{L}}: \At(\mathcal{L})\rightarrow T_{\mathcal{X}}$.
\end{defn}

\begin{lem}
 Let $\mathcal{X}$ be a smooth Deligne--Mumford stack and let $\mathcal{L}$ be a line bundle on $\mathcal{X}$. Then:
 \begin{enumerate}
     \item $\sym_{\mathcal{L}}:  \At(\mathcal{L})\rightarrow T_{\mathcal{X}}$ is $\mathcal{O}_{\mathcal{X}}$-linear and surjective.
     \item There exists a local isomorphism $\At(\mathcal{L})\cong\mathcal{O}_{\mathcal{X}}\oplus T_{\mathcal{X}}$. In particular, we have an exact sequence
$$0\rightarrow\mathcal{O}_{\mathcal{X}}\rightarrow\At(\mathcal{L})\xrightarrow{\sym_{\mathcal{L}}}T_{\mathcal{X}}\longrightarrow 0$$
of locally free sheaves.
 \end{enumerate}
\end{lem}
\begin{proof}
The question is local. (1) is clear by definition. To prove (2), simply note that we have the local isomorphism $\At(\mathcal{L})\cong\mathcal{O}_{\mathcal{X}}\oplus T_{\mathcal{X}}$ can be defined explicitly by
$$D\rightarrow (D-D_{\sym_{\mathcal{L}}(D)},\sym_{\mathcal{L}}(D)).$$
The rest can be checked by definition.
\end{proof}

\begin{defn}\label{defn:cotangent complex}
Let $X$ be a scheme that is of finite type over $\mathbb C$. We denote by $\LL_{X/\CC}=\LL_{X/\Spec\mathbb C}$ the cotangent complex of $X$ over $\CC$ \cite[Tag~08T2]{Sta26}. We will freely use the fact that $H^i(\LL_{X/\CC})=0$ for any $i>0$ (cf. \cite[Tag 08UW]{Sta26}) and $H^0(\LL_{X/\CC})=\Omega^1_X$ \cite[Tag 08UV]{Sta26}. 

Assume that $X$ is separated and Cohen-Macaulay of pure dimension $d$, and let $f: X\rightarrow\Spec\mathbb{C}$ be the structure morphism. We define $\omega_X^{\bullet}:=f^!\mathbb{C}$ to be the \emph{dualizing complex} of $X$. The \emph{dualizing sheaf} of $X$ is $\omega_X$ and satisfies that $\omega_X^{\bullet}=\omega_X[d]$.
\end{defn}

\begin{consthm}\label{consthm:obstruction}
Notation and conditions as in Set-up~\ref{setup:final-setup}. We define $\mathcal{K}$, $\beta_{\bullet,\bullet}$, $\beta$, $\beta_j$, $i_{j,k}$, $\alpha_j$, $\ob_j(\cdot)$, $e_h$, $\eta_j$ as follows.
\begin{enumerate}
\item $\mathcal{K}:=[\mathbb{A}^1/\mathbb{G}_m]$ where $\mathbb G_m$ acts on $\mathbb A^1$ by multiplication.
\item For any Deligne--Mumford stack $\mathcal{Z}$, line bundle $\mathcal{L}$ on $\mathcal{Z}$, and $s\in H^0(\Zz,\Ll)$, we denote by $\beta_{\Ll,s}:\Zz\rightarrow\Kk$ the morphism induced by $\Ll$ and $s$ under the identification as in \cite[Lemma 2.1.1]{Cad07}. For any morphism $\gamma: \Zz'\rightarrow\Zz$, we have $\beta_{\gamma^*\Ll,\gamma^*s}=\beta_{\Ll,s}\circ\gamma$.
\item $\beta:=\beta_{\pi^*\mathcal{T},\theta}: \Tot(\Tt)\rightarrow\Kk$. Recall that $\Tt$ and $\theta$ are defined as in Set-up \ref{setup:framed-map}.
\item For any positive integer $j$, we define $\beta_j:=\beta_{\mathcal{R}|_{\Ss_j},w|_{\Ss_j}}:\Ss_j\rightarrow\mathcal{K}$. 
\item For any $1\leq j\leq k$, we let $i_{j,k}: \Ss_j\rightarrow\Ss_k$ be the closed immersion.
\item For any $j\geq 1$, we have
$$\mathcal{O}_{\mathcal{X}}(-j\mathcal{S})/\mathcal{O}_{\mathcal{X}}(-(j+1)\mathcal{S})\cong i_{1,j,*}\mathcal{N}^{-j}$$
and we may define  $\alpha_j\in \Ext_{\Ss_j}^1(\LL_{\Ss_j/\CC},i_{1,j,*}\mathcal{N}^{-j})$ as the class of the first order thickening $\Ss_j\subset\Ss_{j+1}$. 
\item For any $j\geq 1$ and any morphism $g: \Ss_j\rightarrow\Zz$ to a smooth Deligne--Mumford stack $\Zz$. We have
$$dg:g^*\Omega^1_{\mathcal{Z}}=g^*\LL_{\mathcal{Z}/\CC}\longrightarrow\LL_{\mathcal{S}_j/\CC}$$
as in \cite[Proposition III.3.1.2]{Ill71}, and we may define
$$\ob_j(g):=\alpha_j\circ dg\in \Ext_{\Ss_j}^1(g^*\Omega_{\Zz}^1,i_{1,j,*}\mathcal{N}^{-j})=H^1(\Ss,(g\circ i_{1,j})^*T_{\Zz}\otimes\mathcal{N}^{-j}).$$
\item Recall that $h:\Ss\rightarrow\mathcal{V}$ is the morphism defined as in Set-up \ref{setup:framed-map}(7). For any coherent sheaf $\mathcal{F}$ on $\Ss$, we denote by $e_h$ the following edge map in the Leray spectral sequence
$$0\rightarrow H^1(\mathcal{V},h_*\mathcal{F})\rightarrow H^1(\mathcal{S},\mathcal{F})
\xrightarrow{\ e_h\ }H^0(\mathcal{V},R^1h_*\mathcal{F})\rightarrow H^2(\mathcal{V},h_*\mathcal{F}).$$
We have that $e_h$ is functorial.
\item For any integer $j\geq b$ and $j$-order lift $\sss$, we define
$$\eta_j:=e_h(\ob_j(h_{j,\sss}))\in H^0(\mathcal{V},T_{\mathcal{V}}\otimes R^1h_*\left((h_{j,\sss}\circ i_{1,j})^*T_{\mathcal{V}}\otimes\mathcal{N}^{-j}\right))=H^0(\mathcal{V},T_{\mathcal{V}}\otimes R^1h_*\mathcal{N}^{-j}).$$
\end{enumerate}
\end{consthm}
\begin{proof}
(1) It is definition. (2) By \cite[Lemma~2.1.1]{Cad07}, $\beta_{\Ll,s}$ is the $\mathbb{G}_m$-torsor $P\rightarrow\Zz$ of $\Ll$ together with the $\mathbb{G}_m$-equivariant morphism $P\rightarrow\mathbb{A}^1$ given by $s$, so (2) follows by pulling back to $\Zz'$. (3-5) They are definition. The isomorphism in (6) follows from the exact sequence
$$0\rightarrow\mathcal{O}_{\mathcal{X}}(-(j+1)\mathcal{S})\rightarrow\mathcal{O}_{\mathcal{X}}(-j\mathcal{S})\rightarrow\iota_{1,*}\mathcal{N}^{-j}\rightarrow 0,$$
and that $\mathcal{O}_{\mathcal{X}}(-j\mathcal{S})/\mathcal{O}_{\mathcal{X}}(-(j+1)\mathcal{S})$ can be considered as a sheaf on $\Ss_j$ which is defined by $\mathcal{O}_{\mathcal{X}}(-j\Ss)$. (6) follows from \cite[Theorem~III.1.2.3]{Ill71}. (7) Since $\Zz$ is smooth, $g^*\Omega^1_{\Zz}$ is locally free of finite rank. By \cite[Propositions~III.6.3 and~III.6.7]{Har77} and the projection formula, we have
\begin{align*}
  \Ext^1_{\Ss_j}(g^*\Omega^1_{\Zz},i_{1,j,*}\mathcal{N}^{-j})&=H^1(\Ss_j,(g^*\Omega^1_{\Zz})^{\vee}\otimes i_{1,j,*}\mathcal{N}^{-j})\\
  &=H^1(\Ss_j,i_{1,j,*}((g\circ i_{1,j})^*T_{\Zz}\otimes\mathcal{N}^{-j}))=H^1(\Ss,(g\circ i_{1,j})^*T_{\Zz}\otimes\mathcal{N}^{-j})
\end{align*}
and (7) follows. (8) We only need to show that $e_h$ is functorial, which follows from \cite[Tage 0732]{Sta26}.  (9) It is a definition together with the projection formula.
\end{proof}

\begin{lem}\label{lem:beta-smooth}
Notation and conditions as in Construction-Theorem~\ref{consthm:obstruction}. Then $\mathcal{T}^{\times}:=\Tot(\Tt)\setminus z(\mathcal{V})$ is isomorphic to $\mathcal{O}_{\mathbb{P}^r}(1)^{\times}$, and $\beta$ is representable by schemes and smooth.
\end{lem}
\begin{proof}
By \cite[Section~2.1 and Definition~2.2.6]{Cad07}, a morphism from a scheme $T$ to $\mathcal{V}$ is the same as a morphism $q:T\rightarrow\mathbb{P}^r$ together with a line bundle $\Ll$ on $T$ and an isomorphism $\Ll^{Na}\cong q^*\mathcal{O}_{\mathbb{P}^r}(1)$, and $\Ll$ is the pullback of $\mathcal{T}$. Hence a morphism $T\rightarrow\mathcal{T}^{\times}$ consists of such $(q,\Ll)$ together with a nowhere vanishing section $e$ of $\Ll$. Since $e$ trivializes $\Ll$, these data are the same as the nowhere vanishing section $e^{Na}$ of $q^*\mathcal{O}_{\mathbb{P}^r}(1)$, so $\mathcal{T}^{\times}\cong\mathcal{O}_{\mathbb{P}^r}(1)^{\times}$.

Let $T\rightarrow\Kk$ correspond to $(\Ll,s)$ and let $T'\subset T$ be an open subset with $\Ll|_{T'}\cong\mathcal{O}_{T'}$, so that $s|_{T'}\in\mathcal{O}_{T'}(T')$. A morphism from a scheme $T''$ to $\Tot(\Tt)\times_{\Kk}T'$ consists of $y:T''\rightarrow\Tot(\Tt)$, $\tau:T''\rightarrow T'$, and an isomorphism $\phi:y^*\pi^*\mathcal{T}\cong\mathcal{O}_{T''}$ with $\phi(y^*\theta)=\tau^*(s|_{T'})$. Then $e:=\phi^{-1}(1)$ is a nowhere vanishing section of $(\pi\circ y)^*\mathcal{T}$, i.e. a morphism $T''\rightarrow\mathcal{T}^{\times}$, and $y=\tau^*(s|_{T'})\,e$. The correspondence $(y,\tau,\phi)\rightarrow(e,\tau)$ is natural in $T''$, and an automorphism of $(y,\tau,\phi)$ is an automorphism $\zeta$ of $\Ll$ with $\zeta^{Na}=1$ acting trivially on $\Ll$, hence trivial. Therefore $\Tot(\Tt)\times_{\Kk}T'\cong\mathcal{T}^{\times}\times T'$, which is smooth over $T'$ because $\mathcal{T}^{\times}\cong\mathcal{O}_{\mathbb{P}^r}(1)^{\times}$ is smooth over $\CC$ (\cite[Proposition~III.10.1(b)]{Har77}). These open substacks cover $\Tot(\Tt)\times_{\Kk}T$, which is therefore a scheme that is smooth over $T$. The lemma follows.
\end{proof}

\begin{lem}\label{lem:tangent-tot}
Notation and conditions as in Construction-Theorem~\ref{consthm:obstruction}. Then $$z^*T_{\Tot(\Tt)}=T_{\mathcal{V}}\oplus\mathcal{T}, z^*T_{\Tot(\Tt)/\Kk}=\At(\mathcal{T}),$$ 
and the natural map $z^*T_{\Tot(\Tt)/\Kk}\rightarrow z^*T_{\Tot(\Tt)}$ is $D\rightarrow(\sym_{\mathcal{T}}(D),0)$ in local sections.
\end{lem}
\begin{proof}
Since $\Tot(\Tt)=\Spec_{\mathcal{V}}\operatorname{Sym}^{\bullet}(\mathcal{T}^{\vee})$, the exact sequence $$0\rightarrow\pi^*\Omega^1_{\mathcal{V}}\rightarrow\Omega^1_{\Tot(\Tt)}\rightarrow\Omega_{\Tot(\Tt)/\mathcal{V}}\rightarrow0$$ 
(cf. \cite[Proposition~II.8.11 and Theorem~III.10.2]{Har77}) pulled back along $z$ is $$0\rightarrow\Omega^1_{\mathcal{V}}\rightarrow z^*\Omega^1_{\Tot(\Tt)}\rightarrow\mathcal{T}^{\vee}\rightarrow0.$$
The map $dz:z^*\Omega^1_{\operatorname{Tot}(\mathcal T)}\to\Omega^1_{\mathcal V}$ satisfies $dz\circ z^*(d\pi)=d(\pi\circ z)=\mathrm{id}$, so the exact sequence above split. Taking dual, we obtain
$$z^*T_{\Tot(\Tt)}=T_{\mathcal{V}}\oplus\mathcal{T}.$$
By Lemma~\ref{lem:beta-smooth}, $P:=\Tot(\Tt)\times_{\Kk}\mathbb{A}^1$ is a scheme. A morphism from a scheme $T''$ to $P$ consists of $y:T''\rightarrow\Tot(\Tt)$, a function $c$ on $T''$, and an isomorphism $\phi:y^*\pi^*\mathcal{T}\cong\mathcal{O}_{T''}$ with $\phi(y^*\theta)=c$. We have that $e:=\phi^{-1}(1)$ is a morphism $T''\rightarrow\mathcal{T}^{\times}$ and $y=ce$. Thus we have
$$P=\Tot(\Tt)\times_{\Kk}\mathbb{A}^1\cong\mathcal{T}^{\times}\times\mathbb{A}^1$$
with projection of $P$ to $\Tot(\Tt)$ be $q: \mathcal{T}^{\times}\times\mathbb{A}^1\rightarrow\Tot(\Tt), (e,c)\rightarrow ce$, and projection of $P$ to $\mathbb A^1$ being $(e,c)\rightarrow c$. Thus $q^*T_{\Tot(\Tt)/\Kk}=\mathrm{pr}_1^*T_{\mathcal{T}^{\times}}$. Moreover, $q$ is the base change of $\mathbb{A}^1\rightarrow\Kk$ along $\beta$, hence a $\mathbb{G}_m$-torsor. By construction, $q$ corresponds to the action $\lambda(e,c)=(\lambda^{-1}e,\lambda c)$. Hence $z^*T_{\Tot(\Tt)/\Kk}$ is the sheaf of $\mathbb{G}_m$-invariant vector fields on $\mathcal{T}^{\times}$. 

Let $\mu:\mathcal T^\times\to\mathcal V$ be the projection. Then we have that $$\mu_*\mathcal{O}_{\Tt^{\times}}=\bigoplus_{m\in\mathbb{Z}}\mathcal{T}^{-m},$$
and that a vector field $v$ on $\mathcal{T}^{\times}$ corresponds to a derivation $\delta_v$ of $\bigoplus_{m\in\mathbb{Z}}\mathcal{T}^{-m}$ so that $v$ is $\mathbb G_m$-invariant if and only if $\delta_v$ preserves the grading. Thus for any $v\in z^*T_{\Tot(\Tt)/\Kk}$, $\zeta:=\delta_v|_{\mathcal{O}_{\mathcal{V}}}$ is a vector field on $\mathcal{V}$ and $D:=\delta_v|_{\mathcal{T}}$ satisfies that
$$D(fl)=fD(l)+\zeta(f)l$$
for any $f\in\mathcal{O}_{\mathcal{V}}$ and $l\in\mathcal{T}$. This is equivalent to say that $D\in\At(\Tt)$ and $\sym_{\Tt}(D)=\zeta$. 

Finally, the pullback along $q$ of the natural map $T_{\Tot(\Tt)/\Kk}\rightarrow T_{\Tot(\Tt)}$ is $dq$. Since $q(e,0)=z(\mu(e))$, at $c=0$ we have $dq(v,0)=dz(d\mu(v))$, which lies in the summand $T_{\mathcal{V}}$ of $z^*T_{\Tot(\Tt)}=T_{\mathcal{V}}\oplus\mathcal{T}$ and equals $(\zeta,0)$, because $d\mu(v)=\delta_v|_{\mathcal{O}_{\mathcal{V}}}=\zeta=\sym_{\Tt}(D)$. The lemma follows.
\end{proof}

\begin{defn}\label{defn:lifts}
Notation and conditions as in Construction-Theorem \ref{consthm:obstruction}. Let $j$ be a positive integer, $g: \Ss_j\rightarrow\Zz$ be a morphism to a smooth Deligne--Mumford stack, $V\rightarrow\Ss_{j+1}$ an \'etale morphism of Deligne--Mumford stacks, and $V_j:=V\times_{\Ss_{j+1}}\Ss_j$. A \emph{lift} of $g$ on $V$ is a morphism $\widetilde{g}: V\rightarrow\Zz$ such that $\widetilde{g}|_{V_j}\cong g|_{V_j}$. Two lifts of $g$ on $V$ are identified if they are isomorphic.

Let $j\geq b$ be an integer and let $\sss$ be a $j$-order lift. Let $P_V:=\Tot(\Tt)\times_{\Kk}V$. By Lemma~\ref{lem:beta-smooth}, the projection $P_V\rightarrow V$ is representable by schemes and smooth. The isomorphism $h_{j,\sss}^*\Tt=\mathcal{R}|_{\Ss_j}$ of Set-up~\ref{setup:final-setup} sends $\psi_{j,\sss}^*\theta$ to $w|_{\Ss_j}$. By Construction-Theorem~\ref{consthm:obstruction}(2), we have
$$\beta\circ\psi_{j,\sss}=\beta_j=\beta_{j+1}|_{\Ss_j}.$$
Thus $\psi_{j,\sss}$ defines a section $\sigma_j$ of $P_V\times_VV_j\rightarrow V_j$. 

A \emph{lift of $\psi_{j,\sss}$ over $\Kk$ on $V$} is a section $\sigma$ of $P_V\rightarrow V$ such that $\sigma|_{V_j}=\sigma_j$. It is clear that if $\sigma$ is a lift of $\psi_{j,\sss}$ over $\Kk$ on $V$, then $V\xrightarrow{\sigma}P_V\rightarrow\Tot(\Tt)$ is then a lift of $\psi_{j,\sss}$ on $V$ where $P_V\rightarrow\Tot(\Tt)$ is the first projection.
\end{defn}

\begin{consthm}\label{consthm:oj}
Notation and conditions as in Construction-Theorem~\ref{consthm:obstruction}. Let $j$ be a positive integer, $g:\Ss_j\rightarrow\Zz$ be a morphism to a smooth Deligne--Mumford stack, $V\rightarrow\Ss_{j+1}$ an \'etale morphism of Deligne--Mumford stacks, $V_j:=V\times_{\Ss_{j+1}}\Ss_j$, and $W_j:=V\times_{\Ss_{j+1}}\Ss$. Let $\{V_a\rightarrow\Ss_{j+1}\}$ be an \'etale covering by affine schemes, and let $V_{ab}:=V_a\times_{\Ss_{j+1}}V_b$. 

Let $\mathcal{I}:=i_{1,j,*}\mathcal{N}^{-j}|_{V_j}$, which is the ideal sheaf of $V_j$ in $V$ by Construction-Theorem~\ref{consthm:obstruction}(6). Recall that $\ob_j(g)$ is the class defined as in Construction-Theorem~\ref{consthm:obstruction}(7) in $H^1(\Ss,(g\circ i_{1,j})^*T_{\Zz}\otimes\mathcal{N}^{-j})$. When $j\geq b$, we also construct a class $o_j$ in (4), which will be used in the rest of the paper.
\begin{enumerate}
\item A lift of $g$ on $V$ exists if and only if the pullback of $\ob_j(g)$ via $W_j\rightarrow\Ss$ to $H^1(W_j,(g\circ i_{1,j})^*T_{\Zz}\otimes\mathcal{N}^{-j})$ is $0$. 
\item Assume that a lift of $g$ on $V$ exists. Then $H^0(W_j,(g\circ i_{1,j})^*T_{\Zz}\otimes\mathcal{N}^{-j})$ acts simply transitively on the set of lifts $\widetilde{g}$ of $g$ on $V$. We denote this action by
$$(D,\widetilde g)\rightarrow\widetilde g+D.$$
If $\Zz$ is a scheme, then $\widetilde g+D$ is the lift with $(\widetilde g+D)^{\sharp}=\widetilde g^{\sharp}+D$, where $D$ is regarded as a derivation $g^{-1}\mathcal{O}_{\Zz}\rightarrow\mathcal{I}$ via Construction-Theorem~\ref{consthm:obstruction}(7); in general, $\widetilde g+D$ is defined on \'etale charts of $\Zz$ in this way.
\item There exist lifts $\widetilde g_a,\widetilde g_b$ of $g$ on $V_a,V_b$ respectively. Moreover, for any such $\widetilde g_a,\widetilde g_b$, we have that
$$\widetilde g_a|_{V_{ab}}=\widetilde g_b|_{V_{ab}}+D_{ab}$$
where $(D_{ab})$, with $D_{ab}\in H^0(V_{ab}\times_{\Ss_{j+1}}\Ss,(g\circ i_{1,j})^*T_{\Zz}\otimes\mathcal{N}^{-j})$, is a \v{C}ech $1$-cocycle for the \'etale covering $\{V_a\times_{\Ss_{j+1}}\Ss\rightarrow\Ss\}$ whose image in $H^1(\Ss,(g\circ i_{1,j})^*T_{\Zz}\otimes\mathcal{N}^{-j})$ under the natural map from \v{C}ech cohomology (\cite[Lemma~III.4.4]{Har77}) is $\ob_j(g)$.
\item Let $j\geq b$ and let $\sss$ be a $j$-order lift. Assume that a lift of $\psi_{j,\sss}$ over $\Kk$ on $V$ exists. Then
$$H^0(W_j,h^*\At(\mathcal{T})\otimes\mathcal{N}^{-j})$$ 
acts simply transitively on the set of lifts $\sigma$ of $\psi_{j,\sss}$ over $\Kk$ on $V$. We denote this action by $(D,\sigma)\rightarrow\sigma+D$, where $(\sigma+D)^{\sharp}=\sigma^{\sharp}+D$ and $D$ is regarded as an $\mathcal{O}_V$-linear derivation $\sigma_j^{-1}\mathcal{O}_{P_V}\rightarrow\mathcal{I}$ via $\sigma_j^*\Omega_{P_V/V}|_{W_j}=h^*\At(\mathcal{T})^{\vee}|_{W_j}$ (Definition~\ref{defn:smooth-representable} and Lemma~\ref{lem:tangent-tot}). There exist lifts $\sigma_a,\sigma_b$ of $\psi_{j,\sss}$ over $\Kk$ on $V_a,V_b$ respectively. Moreover, for any such $\sigma_a,\sigma_b$, we have $\sigma_a|_{V_{ab}}=\sigma_b|_{V_{ab}}+D_{ab}$ where $(D_{ab})$, with $D_{ab}\in H^0(V_{ab}\times_{\Ss_{j+1}}\Ss,h^*\At(\mathcal{T})\otimes\mathcal{N}^{-j})$, is a \v{C}ech $1$-cocycle for the \'etale covering $\{V_a\times_{\Ss_{j+1}}\Ss\rightarrow\Ss\}$. We define
\begin{equation}\label{eq:framed-lifting-complete-obstruction}
o_j\in H^1(\mathcal{S},h^*\At(\mathcal{T})\otimes\mathcal{N}^{-j})
\end{equation}
to be the image of the class of $(D_{ab})$ under the natural map from \v{C}ech cohomology (\cite[Lemma~III.4.4]{Har77}). It does not depend on the choices of $\{V_a\}$ and $\sigma_a$.
\item Let $j\geq b$ and let $\sss$ be a $j$-order lift. A lift of $\psi_{j,\sss}$ over $\Kk$ on $V$ exists if and only if the pullback of $o_j$ via $W_j\rightarrow\Ss$ to $H^1(W_j,h^*\At(\mathcal{T})\otimes\mathcal{N}^{-j})$ is $0$.
\item Let $j\geq b$ and let $\sss$ be a $j$-order lift. Let $\sym: H^1(\Ss,h^*\At(\mathcal{T})\otimes\mathcal{N}^{-j})\rightarrow H^1(\Ss,h^*T_{\mathcal{V}}\otimes\mathcal{N}^{-j})$ be the map induced by $\sym_{\mathcal{T}}$. Since $h_{j,\sss}\circ i_{1,j}=h$, we have $\ob_j(h_{j,\sss})\in H^1(\Ss,h^*T_{\mathcal{V}}\otimes\mathcal{N}^{-j})$, and
$$\sym(o_j)=\ob_j(h_{j,\sss}).$$
Since $\psi_{j,\sss}\circ i_{1,j}=z\circ h$, Lemma~\ref{lem:tangent-tot} gives
$$H^1(\Ss,(\psi_{j,\sss}\circ i_{1,j})^*T_{\Tot(\Tt)}\otimes\mathcal{N}^{-j})=H^1(\Ss,h^*T_{\mathcal{V}}\otimes\mathcal{N}^{-j})\oplus H^1(\Ss,h^*\mathcal{T}\otimes\mathcal{N}^{-j}),$$
and $\ob_j(\psi_{j,\sss})=(\ob_j(h_{j,\sss}),0)$ in this decomposition. Consequently,
\begin{equation}\label{eq:filtered-comparison-pure-tangent}
e_h(\ob_j(\psi_{j,\sss}))=(\eta_j,0)
\in H^0(\mathcal{V},T_{\mathcal{V}}\otimes R^1h_*\mathcal{N}^{-j})\oplus H^0(\mathcal{V},\mathcal{T}\otimes R^1h_*\mathcal{N}^{-j}).
\end{equation}
\end{enumerate}
\end{consthm}
\begin{proof}
(1) and (2) Since the diagonal of $\Zz$ is unramified (\cite[Tag~06N3]{Sta26}) and $V_j\subset V$ is defined by a square-zero ideal, an isomorphism between two lifts of $g$ on $V$ is unique if it exists (\cite[Tag~05ZY]{Sta26}). Hence lifts of $g$ on $V$ can be constructed \'etale locally on $V$ and glued, and we may replace $\Zz$ by an \'etale chart $Z\rightarrow\Zz$, whose pullback to $V_j$ extends uniquely to an \'etale morphism to $V$ (\cite[Tag~04DZ]{Sta26}). For the smooth scheme $Z$, a lift of $g$ on $V$ is a homomorphism of sheaves of rings $g^{-1}\mathcal{O}_Z\rightarrow\mathcal{O}_V$ lifting $g^{\sharp}:g^{-1}\mathcal{O}_Z\rightarrow\mathcal{O}_{V_j}$. By \cite[Proposition~III.2.2.2]{Ill71}, applied to the square-zero extension $\mathcal{O}_V\rightarrow\mathcal{O}_{V_j}$ by $\mathcal{I}$, such a lift exists if and only if the image of $\alpha_j\circ dg$ in $\Ext^1_{V_j}(g^*\Omega^1_{Z}|_{V_j},\mathcal{I})=H^1(W_j,(g\circ i_{1,j})^*T_{\Zz}\otimes\mathcal{N}^{-j})$ vanishes, and in this case $\operatorname{Hom}(g^*\Omega^1_Z|_{V_j},\mathcal{I})=H^0(W_j,(g\circ i_{1,j})^*T_{\Zz}\otimes\mathcal{N}^{-j})$ acts simply transitively on the set of lifts; here $\Omega^1_Z=\LL_{Z/\CC}$ by \cite[Proposition~III.3.1.2]{Ill71}, and the two identifications are those of Construction-Theorem~\ref{consthm:obstruction}(7).

(3) Since $V_a$ is affine, $V_a\times_{\Ss_{j+1}}\Ss$ is affine, so $H^1(V_a\times_{\Ss_{j+1}}\Ss,(g\circ i_{1,j})^*T_{\Zz}\otimes\mathcal{N}^{-j})=0$ (\cite[Theorem~III.3.5]{Har77}) and lifts $\widetilde g_a$ exist by (1). By \cite[(2.2.1.3) and (2.2.1.4)]{Ill71}, $\ob_j(g)$ is the class of the square-zero extension $g^{-1}\mathcal{O}_{\Zz}\times_{\mathcal{O}_{\Ss_j}}\mathcal{O}_{\Ss_{j+1}}$ of $g^{-1}\mathcal{O}_{\Zz}$ by $i_{1,j,*}\mathcal{N}^{-j}$, lifts of $g$ are the splittings of this extension, and the \v{C}ech cocycle of differences of local splittings represents its class under the isomorphism of \cite[Theorem~III.1.2.3]{Ill71}.

(4) and (5) A lift of $\psi_{j,\sss}$ over $\Kk$ on $V$ is a section of $P_V\rightarrow V$ extending $\sigma_j$, i.e. a homomorphism of sheaves of $\mathcal{O}_V$-algebras $\sigma_j^{-1}\mathcal{O}_{P_V}\rightarrow\mathcal{O}_V$ lifting $\sigma_j^{\sharp}$. Since $P_V\rightarrow V$ is smooth by Lemma~\ref{lem:beta-smooth}, \cite[Proposition~III.2.2.2]{Ill71} applies as in (1), with $\CC$ replaced by $\mathcal{O}_V$ and $g^*\Omega^1_Z$ by $\sigma_j^*\Omega_{P_V/V}$. We have $\Omega_{P_V/V}=\mathrm{pr}_1^*\Omega_{\Tot(\Tt)/\Kk}$ by Definition~\ref{defn:smooth-representable}, $\mathrm{pr}_1\circ\sigma_j=\psi_{j,\sss}|_{V_j}$, and $\psi_{j,\sss}\circ i_{1,j}=z\circ h$, so $\sigma_j^*\Omega_{P_V/V}|_{W_j}=h^*z^*\Omega_{\Tot(\Tt)/\Kk}|_{W_j}=h^*\At(\mathcal{T})^{\vee}|_{W_j}$ by Lemma~\ref{lem:tangent-tot}. Hence a lift of $\psi_{j,\sss}$ over $\Kk$ on $V$ exists if and only if the obstruction class $\omega_V\in H^1(W_j,h^*\At(\mathcal{T})\otimes\mathcal{N}^{-j})$ of \cite[Proposition~III.2.2.2]{Ill71} vanishes, and in this case $H^0(W_j,h^*\At(\mathcal{T})\otimes\mathcal{N}^{-j})$ acts simply transitively on the set of such lifts. Since $V_a\times_{\Ss_{j+1}}\Ss$ is affine, $\omega_{V_a}=0$, so the lifts $\sigma_a$ exist. As in (3), $\omega_{\Ss_{j+1}}$ is the class of a square-zero extension whose splittings are the lifts of $\psi_{j,\sss}$ over $\Kk$, and the cocycle $(D_{ab})$ of differences of local splittings represents it; hence $(D_{ab})$ is a cocycle whose class $o_j=\omega_{\Ss_{j+1}}$ does not depend on the choices. Finally, the pullback of $o_j$ to $H^1(W_j,h^*\At(\mathcal{T})\otimes\mathcal{N}^{-j})$ is $\omega_V$, because the classes of \cite[Proposition~III.2.2.2]{Ill71} are compatible with \'etale base change.

(6) Let $\sigma_a$ be lifts of $\psi_{j,\sss}$ over $\Kk$ on $V_a$, and let $D_{ab}\in H^0(V_{ab}\times_{\Ss_{j+1}}\Ss,h^*\At(\mathcal{T})\otimes\mathcal{N}^{-j})$ be their differences. Then $\pi\circ\mathrm{pr}_1\circ\sigma_a$ are lifts of $h_{j,\sss}$ on $V_a$, and $\mathrm{pr}_1\circ\sigma_a$ are lifts of $\psi_{j,\sss}$ on $V_a$. By the functoriality of the actions in \cite[Proposition~III.2.2.2]{Ill71} under composition with $\pi\circ\mathrm{pr}_1$ and with $\mathrm{pr}_1$, their differences are the images of $D_{ab}$ under the maps induced by $h^*\Omega^1_{\mathcal{V}}\rightarrow\sigma_j^*\Omega_{P/\Ss_{j+1}}$ and $\psi_{j,\sss}^*\Omega^1_{\Tot(\Tt)}\rightarrow\sigma_j^*\Omega_{P/\Ss_{j+1}}$, which by Lemma~\ref{lem:tangent-tot} are $\sym_{\mathcal{T}}(D_{ab})$ and $(\sym_{\mathcal{T}}(D_{ab}),0)$. By (3) and (4), these cocycles represent $o_j$, $\ob_j(h_{j,\sss})$ and $\ob_j(\psi_{j,\sss})$, which proves the two identities. Applying $e_h$ and using the definition of $\eta_j$ gives \eqref{eq:filtered-comparison-pure-tangent}.
\end{proof}

\section{Locating the obstruction via mixed Hodge modules}\label{sec:06-filtered-comparison}

Notation and conditions as in Set-up \ref{setup:final-setup}. In order to lift a $j$-order lift $\sss$ to a $(j+1)$-order lift, it suffices to lift the morphism $h_{j,\sss}:\Ss_j\rightarrow\mathcal{V}$ to a morphism $\Ss_{j+1}\rightarrow\mathcal{V}$. Such a lift induces a $(j+1)$-order lift $\sss'$ with $\sss'|_{\Ss_j}=\sss$. In this section, we use the theory of mixed Hodge modules to show that the obstruction $\eta_j$ to this lifting lies in the kernel of the Kodaira-Spencer map of the Hodge module
$$\Hh^1h_*\left(\DD\Qq_{\mathcal{S}}^H[-(d-1)]\right)$$
(Theorem~\ref{thm:tangential-comparison}). We shall use vanishing theorems to kill these obstructions in later sections.

\subsection{Preliminaries on mixed Hodge modules}\label{subsec:06-mhm-prelim}

We use the theory of mixed Hodge modules \cite{Sai88,Sai90,Sai00} and polarizable complex Hodge modules \cite{SS26}. We refer the reader to \cite{SS26} for a textbook reference.

\begin{defn}
A \emph{complex} $C=(C^i,d^i)_{i\in\mathbb Z}$ in an abelian category $\mathcal{A}$ consists of objects $C^i\in\mathcal{A}$ and differentials $d^i: C^i\rightarrow C^{i+1}$ so that $d^{i+1}\circ d^i=0$ for any $i$. For any integer $r$, we denote by $C[r]$ the complex $(C^{i+r},(-1)^rd^{i+r})_{i\in\mathbb Z}$. A morphism of complexes $f: C\rightarrow C'$ is a \emph{quasi-isomorphism} if $H^i(f): H^i(C)\rightarrow H^i(C')$ is an isomorphism for any $i$.
\end{defn}

\begin{defn}
Let $X$ be a scheme or a Delinge--Mumford stack. For a complex $C$ of sheaves of abelian groups on $X$ and any integer $i$, we define the \emph{hypercohomology} $\HH^i(X,C):=H^i(R\Gamma(X,C))$, where $R\Gamma(X,-)$ is the right derived functor of $\Gamma(X,-)$. 
\end{defn}

\begin{defn}
Let $X$ be a smooth variety of dimension $d$ and $T_X$ the tangent sheaf of $X$. We let $\mathcal{D}_X$ be the sheaf of differential operators on $X$ with filtration $F_{\bullet}\mathcal{D}_X$, by the order of differential operators (cf. \cite[2.1.1]{Sai88}).

A \emph{right} (resp. \emph{left}) $\mathcal{D}_X$-\emph{module} $\mathcal{M}$ is a sheaf of right (resp. left) $\mathcal{D}_X$-modules on $X$ (i.e. $\mathcal{M}(U)$ is a right (resp. left) $\mathcal{D}_X(U)$-module for any open subset $U$) which is quasi-coherent as an $\mathcal{O}_X$-module. We have that
$$\mathcal{M}^{\mathrm{l}}\rightarrow\omega_X\otimes_{\mathcal{O}_X}\mathcal{M}^{\mathrm{l}}$$ is an equivalence between left and right $\mathcal{D}_X$-modules (cf. \cite[Proposition 8.2.5]{SS26}). In the rest of the paper, without further specification, a $\mathcal{D}_X$-module means a right $\mathcal{D}_X$-module. A \emph{filtered right} $\mathcal{D}_X$-module $(\mathcal{M},F)$ is a right $\mathcal{D}_X$-module $\mathcal{M}$ with an increasing filtration $F_{\bullet}\mathcal{M}$ by coherent $\mathcal{O}_X$-submodules $F_p\mathcal{M}$ satisfying the following:
\begin{enumerate}
    \item $\bigcup_{p\in\mathbb Z}F_p\mathcal{M}=\mathcal{M}$.
    \item $F_p\mathcal{M}=0$ for any $p\ll 0$.
    \item $F_p\mathcal{M}\cdot F_q\mathcal{D}_X\subset F_{p+q}\mathcal{M}$ for any $p,q$.
\end{enumerate}
See \cite[2.1.1]{Sai88} for details. For any $p$, we define
$$\gr_p^F\mathcal{M}:=F_p\mathcal{M}/F_{p-1}\mathcal{M}\quad \text{and}\quad F_{<p}\mathcal{M}:=F_{p-1}\mathcal{M}.$$

The \emph{support} of a $\mathcal{D}_X$-module $\mathcal{M}$ is the support of $\mathcal{M}$ as a sheaf. For any closed subset $V\subset X$, we say that $\mathcal{M}$ is supported on $V$ if $\Supp\mathcal{M}\subset V$.
\end{defn}

The definitions of mixed Hodge modules and of polarizable complex Hodge modules are very complicated and we refer the reader to \cite{Sai90,SS26} for details. In this paper, we treat the formal definition of mixed Hodge modules as a blackbox and only use its basic properties.

\begin{defn}\label{defn:mhm}
Let $X$ be a smooth variety. A \emph{mixed Hodge module} $M$ on $X$ is a tuple $(\mathcal{M},F,K,W)$ as in \cite[p.~223 and 4.2]{Sai90}, where $(\mathcal{M},F)$ is a filtered right $\mathcal{D}_X$-module. A \emph{polarizable complex Hodge module} $M$ on $X$ is a triple $((\mathcal{M},F),(\mathcal{M},F)(w),S)$ as in \cite[Definitions 14.2.1, 14.2.2 and Proposition 14.2.27]{SS26}, where $(\mathcal{M},F)$ is a filtered right $\mathcal{D}_X$-module.

Let $M=(\mathcal{M},F,K,W)$ be a mixed Hodge module on $X$, or let $M=((\mathcal{M},F),(\mathcal{M},F)(w),S)$ be a polarizable complex Hodge module on $X$. We call $(\mathcal{M},F)$ the \emph{underlying filtered $\mathcal{D}_X$-module} of $M$, $\mathcal{M}$ the \emph{underlying $\mathcal{D}_X$-module}, and $F$ the \emph{Hodge filtration} of $M$. We denote by $\DD M$ the dual of $M$, as in \cite[0.1 Theorem]{Sai90} if $M$ is a mixed Hodge module, and as in \cite[Definition~13.4.7 and Proposition~16.3.5]{SS26} if $M$ is a polarizable complex Hodge module. 

A \emph{Hodge module} is either a mixed Hodge module or a polarizable complex Hodge module. For any Hodge module $M$ on $X$ with underlying $\mathcal{D}_X$-module $\mathcal{M}$ and any subset $V\subset X$, we say that $M$ is \emph{supported on} $V$ if $\mathcal{M}$ is supported on $V$.
\end{defn}

\begin{defn}\label{defn:de-rham-complex}
Let $X$ be a smooth variety of dimension $d$ and let $\mathcal{M}$ be a $\mathcal{D}_X$-module. 
\begin{enumerate}
    \item The \emph{de Rham complex} of $\mathcal{M}$ is
$$\DR(\mathcal{M}):=\left[\mathcal{M}\otimes\bigwedge^dT_X\longrightarrow\mathcal{M}\otimes\bigwedge^{d-1}T_X\longrightarrow\cdots\longrightarrow\mathcal{M}\otimes T_X\longrightarrow\mathcal{M}\right]$$
satisfying the following conditions.
\begin{enumerate}
    \item The piece $\mathcal{M}\otimes\bigwedge^iT_X$ has degree $-i$ for any $0\leq i\leq d$.
    \item The differential $\mathcal{M}\otimes\bigwedge^iT_X\longrightarrow\mathcal{M}\otimes\bigwedge^{i-1}T_X$ is given by
    \begin{align*}
    m\otimes\xi_1\wedge\cdots\wedge\xi_i\longmapsto{}&\sum_{k=1}^{i}(-1)^{k-1}(m\cdot\xi_k)\otimes\xi_1\wedge\cdots\wedge\widehat{\xi_k}\wedge\cdots\wedge\xi_i\\
    &+\sum_{1\leq k<l\leq i}(-1)^{k+l}\,m\otimes[\xi_k,\xi_l]\wedge\xi_1\wedge\cdots\wedge\widehat{\xi_k}\wedge\cdots\wedge\widehat{\xi_l}\wedge\cdots\wedge\xi_i
    \end{align*}
    for $m\in\mathcal{M}$ and $\xi_1,\dots,\xi_i\in T_X$ in local coordinates (cf. \cite[Definition~8.4.3]{SS26}).
\end{enumerate}
\item Assume that $(\mathcal{M},F)$ is a filtered $\mathcal{D}_X$-module. We define
$$F_p\DR(\mathcal{M}):=\left[F_{p-d}\mathcal{M}\otimes\bigwedge^dT_X\longrightarrow\cdots\longrightarrow F_{p-1}\mathcal{M}\otimes T_X\longrightarrow F_p\mathcal{M}\right]$$
for any integer $p$.
\item Assume that $M$ is a Hodge module with underlying filtered $\mathcal{D}_X$-module $(\mathcal{M},F)$. We define
$$C_p(M):=\gr_p^F\DR(\mathcal{M})=\left[\gr_{p-d}^F\mathcal{M}\otimes\bigwedge^dT_X\longrightarrow\cdots\longrightarrow\gr_{p-1}^F\mathcal{M}\otimes T_X\longrightarrow\gr_p^F\mathcal{M}\right]$$
for any integer $p$. The differentials of $C_p(M)$ are $\mathcal{O}_X$-linear (cf. \cite[2.2.4]{Sai88}), so $C_p(M)$ is a bounded complex of coherent $\mathcal{O}_X$-modules in degrees $[-d,0]$.
\end{enumerate}
\end{defn}

\begin{deflem}\label{deflem:kodaira-spencer}
 Let $X$ be a smooth variety and let $M$ be a Hodge module with underlying filtered $\mathcal{D}_X$-module $(\mathcal{M},F)$, such that $F_{<0}\mathcal{M}=0$. 
 
In this case, $\gr_{1-k}^F\mathcal{M}=0$ for any integer $k\geq 2$, so 
 $$C_1(M)=\left[F_0\mathcal{M}\otimes T_X\xrightarrow{\ \sigma\ }\gr_1^F\mathcal{M}\right]$$
 is a $2$-term complex with degrees $-1$ and $0$, where $\sigma(m\otimes\xi)$ is the class of $m\cdot\xi$ in $\gr_1^F\mathcal{M}$. We call $\sigma$ the \emph{Kodaira-Spencer map} of $M$.
\end{deflem}
\begin{proof}
The only thing we need to prove is that  $\gr_{1-k}^F\mathcal{M}=0$ for any integer $k\geq 2$, which is obvious as $F_{<0}\mathcal{M}=0$. 
\end{proof}

\begin{defn}\label{defn:hodge-module-singular}
Let $V$ be a variety and let $i: V\rightarrow X$ be a closed embedding into a smooth variety $X$. A \emph{Hodge module} (resp. \emph{mixed Hodge module}, \emph{polarizable complex Hodge module}) $M$ on $V$ is a Hodge module $M$ on $X$ such that the underlying $\mathcal{D}_X$-module is supported on $V$. For any integer $p$, $C_p(M)$ is considered as a complex of coherent sheaves on $V$.

It is known that the category of Hodge modules (resp. mixed Hodge modules, polarizable complex Hodge modules) does not depend on the choice of $i$, i.e. for any closed embedding $i': V\rightarrow X'$ into a smooth variety $X'$, there is a canonical identification between Hodge modules (resp. mixed Hodge modules, polarizable complex Hodge modules) on $X$ supported on $i(V)$ and Hodge modules (resp. mixed Hodge modules, polarizable complex Hodge modules) on $X'$ supported on $i'(V)$ (cf. \cite[2.1 and (4.2.4)]{Sai90} for mixed Hodge modules, and \cite[Theorems~14.2.19(1) and~16.2.1]{SS26} for polarizable complex Hodge modules). Moreover, $C_p(M)$ also does not depend on the choice of $i$ up to quasi-isomorphism (cf. \cite[Proposition~2.33]{Sai90} for mixed Hodge modules, and \cite[Proposition 14.2.3(5)]{SS26} for polarizable complex Hodge modules).
\end{defn}

\begin{defn}\label{defn:constant-hodge-complex}
Let $\mathbb{Q}^H:=(\mathbb{C},F,\mathbb{Q},W)$ be the constant Hodge structure on $\Spec\mathbb{C}$, placed in degree zero. For a reduced variety $X$ with structure morphism $\pi:X\rightarrow\Spec\mathbb{C}$, we define the \emph{constant Hodge complex}
\[
\mathbb{Q}_X^H:=\pi^*\mathbb{Q}^H\in D^b\operatorname{MHM}(X)
\]
\cite[Theorem~0.1 and Section~4.4]{Sai90}. This is a bounded complex of rational mixed Hodge modules. It need not be a single mixed Hodge module.
\end{defn}

\begin{defn}\label{defn:direct-image}
Let $f:S\rightarrow V$ be a projective morphism from a variety to a smooth variety. If $K\in D^b\operatorname{MHM}(S)$, its direct image $f_*K$ belongs to $D^b\operatorname{MHM}(V)$, and $\mathcal{H}^uf_*K$ is a mixed Hodge module for each integer $u$ \cite[Theorem~0.1 and Section~4.3]{Sai90}. Here $\mathcal{H}^u$ denotes cohomology in the category of mixed Hodge modules. For a polarizable complex Hodge module $M$, we write $\mathcal{H}^uf_*M$ for the cohomological direct image ${}_Tf_*^{(u)}M$ of \cite[Theorem~14.3.1 and Corollary~14.3.2]{SS26}. For singular $S$, use Definition~\ref{defn:hodge-module-singular} to regard the module as one on a smooth ambient variety supported on $S$, then use a graph embedding to compute the direct image.

For a closed embedding $z:V\hookrightarrow Y$ of smooth varieties, we denote the filtered direct image of a right $\mathcal{D}_V$-module $(\mathcal{M},F)$ by $z_+\mathcal{M}$ \cite[Section~2.3.1]{Sai88}. For the zero section of $Y=V\times\mathbb{A}^1$, with coordinate $t$, we have
\[
z_+\mathcal{M}=\mathcal{M}\otimes_{\mathbb{C}}\mathbb{C}[\partial_t],\qquad
F_k(z_+\mathcal{M})=\sum_{\mu\geq0}F_{k-\mu}\mathcal{M}\,\partial_t^\mu
\]
\cite[(3.2.2.3)]{Sai88}. We use the same notation for filtered complexes.

For a morphism $f:S'\rightarrow S$, the pullback $f^*:D^b\operatorname{MHM}(S)\rightarrow D^b\operatorname{MHM}(S')$ is the functor of \cite[Section~4.4]{Sai90}, left adjoint to $f_*$ \cite[Theorem~0.1]{Sai90}. If $f$ is \'etale, it takes a single Hodge module to a single Hodge module, and on smooth charts its filtered right $\mathcal{D}$-module is
\[
f^*(\mathcal{M},F)=
\left(f^{-1}\mathcal{M}\otimes_{f^{-1}\mathcal{O}_S}\mathcal{O}_{S'},
f^{-1}F\otimes_{f^{-1}\mathcal{O}_S}\mathcal{O}_{S'}\right).
\]
For rational mixed Hodge modules this follows from \cite[(3.5.1) and Lemma~3.5.2]{Sai88} and \cite[(2.17.6) and Lemma~2.25]{Sai90}; for polarizable complex Hodge modules, from non-characteristic pullback \cite[Definition~8.8.20 and Corollary~16.3.4]{SS26}.
\end{defn}

\begin{defn}\label{defn:hodge-module-stack}
Let $\mathcal{Y}$ be a Deligne--Mumford stack of finite type over $\mathbb{C}$ and let $Y\rightarrow\mathcal{Y}$ be an \'etale atlas. Write $p_1,p_2:Y\times_{\mathcal{Y}}Y\rightarrow Y$ for the projections, and $p_{12},p_{13},p_{23}$ for the projections from $Y\times_{\mathcal{Y}}Y\times_{\mathcal{Y}}Y$ to the indicated two factors. A \emph{Hodge module} on $\mathcal{Y}$ consists of a Hodge module $M$ on $Y$ and an isomorphism of Hodge modules $\phi:p_1^*M\rightarrow p_2^*M$ satisfying
\[
p_{13}^*\phi=p_{23}^*\phi\circ p_{12}^*\phi
\]
on $Y\times_{\mathcal{Y}}Y\times_{\mathcal{Y}}Y$. We use this definition separately for rational mixed Hodge modules and for polarizable complex Hodge modules. The isomorphism $\phi$ includes all the Hodge data. Pullback to a common refinement identifies the definitions for two atlases, by \'etale descent and the pullbacks in Definition~\ref{defn:direct-image}.

We form $C_p(M)$ and $\DD M$ on the charts and descend them by $\phi$. For a projective morphism representable by schemes, direct images are computed after an \'etale base change on the target; proper base change supplies their descent isomorphisms. The same construction applies to the bounded rational mixed Hodge complexes used below. For singular charts we use Definition~\ref{defn:hodge-module-singular}.
\end{defn}

\subsection{Derived de Rham complexes and filtered duality}

We use the Hodge-completed derived de Rham complex $\dR_T$ of \cite{Bha12}, viewed as a filtered differential complex as in \cite[Section~2.4]{CDO26}.

\begin{thm}[{\cite[Construction~4.1, Remarks~4.2--4.3, Corollary~4.27 and Proposition~5.2]{Bha12}, \cite[Section~2.4]{CDO26}}]\label{thm:derived-de-rham}
Let $T$ be a scheme of finite type over $\mathbb{C}$, possibly nonreduced. The complex $\dR_T$ is functorial in $T$, and
\[
\gr_{\mathrm H}^s\dR_T\simeq\mathbf{L}\!\bigwedge^s\LL_{T/\CC}[-s]\qquad(s\geq0).
\]
For a closed embedding of $T$ into a smooth scheme, the comparison with Hartshorne's de Rham complex, formed by completing the ambient de Rham complex along $T$, is a quasi-isomorphism after forgetting the filtrations. There is also a filtered comparison $\dR_T\rightarrow\underline{\Omega}_{T_{\mathrm{red}}}^{\bullet}$. On the smooth ambient scheme these maps are morphisms of filtered differential complexes.
\end{thm}

In particular, the first two Hodge grades are $\mathcal{O}_T$ and $\LL_{T/\CC}[-1]$.

\begin{thm}[{\cite[Definitions~3.6, 3.7, Lemma~3.9, Corollary~3.10 and Section~3.2]{CDO26}, \cite[Sections~2.2.10 and~2.4.11]{Sai88}}]\label{thm:filtered-duality}
Let $U$ be smooth of dimension $n$, and let $(C,F)$ be a filtered differential complex on $U$ such that locally $F_bC^i=C^i$ for some $b$ independent of $i$. We work in the filtered derived category, formed using filtered homotopies and filtered quasi-isomorphisms. Let $\DD_UC$ be the filtered differential dual, computed using a resolution of $\omega_U[n]$ by right $\mathcal{D}_U$-modules injective as $\mathcal{O}_U$-modules. Let $\DR_U^{-1}$ denote linearization. Then:
\begin{enumerate}
\item $\gr_p^F\DD_UC\simeq\RHom_U(\gr_{-p}^FC,\omega_U[n])$ for every integer $p$.
\item If a filtered morphism induces quasi-isomorphisms on $C/F_{-p-1}C$ for $p\leq m$, its dual induces quasi-isomorphisms on $F_p$ for $p\leq m$.
\item $\DR_U^{-1}$ and $\DR_U$ are mutually quasi-inverse when the filtrations are bounded below uniformly in cohomological degree.
\item On de Rham realizations of rational mixed Hodge modules and their bounded complexes, $\DD_U$ corresponds under $\DR_U^{-1}$ to filtered right $\mathcal{D}$-module duality.
\item If $F_0C=C$, then $F_p\DD_UC=0$ for $p<0$.
\end{enumerate}
\end{thm}

\begin{thm}[{\cite[Lemmas~4.1, 4.2 and~4.9]{CDO26}}]\label{thm:cdo-summand}
Let $U$ be smooth and connected of dimension $n$, let $S\subset U$ be a reduced closed subscheme, and put $E:=\DR_U^{-1}\DD_U(\dR_S)[-n]$. After forgetting the filtrations, the comparison maps give
\[
R\Gamma_S(\omega_U)\longrightarrow E\longrightarrow R\Gamma_S(\omega_U)\longrightarrow\omega_U,
\]
and the composition of the first two arrows is the identity. If $K$ is a bounded below complex of right $\mathcal{D}_U$-modules with cohomology supported on $S$, every morphism $K\rightarrow\omega_U$ factors uniquely through $R\Gamma_S(\omega_U)\rightarrow\omega_U$. Moreover,
\[
\operatorname{End}(R\Gamma_S(\omega_U))=H^0(S^{\mathrm{an}},\mathbb{C}).
\]
\end{thm}

\begin{thm}[{\cite[Section~3.4, (3.4.1)--(3.4.4)]{FFS14}}]\label{thm:dubois-realization}
Let $S$ be a reduced scheme of finite type over $\mathbb{C}$. The filtered de Rham realization of $\mathbb{Q}_S^H$ is the Du Bois complex; in particular,
\[
\gr_{-s}^F\DR(\mathbb{Q}_S^H)\simeq\underline{\Omega}_S^s[-s]\qquad(s\geq0).
\]
If $S$ is Du Bois, then $\underline{\Omega}_S^0\simeq\mathcal{O}_S$, $F_{<0}\DD\mathbb{Q}_S^H=0$, and $F_0\DD\mathbb{Q}_S^H\simeq\omega_S^{\bullet}$. Thus, if $S$ is also Cohen-Macaulay of pure dimension $\delta$, then
\[
F_0(\DD\mathbb{Q}_S^H[-\delta])\simeq\omega_S.
\]
The assertions about $F$ refer to the filtered $\mathcal{D}$-module complex realizing $\DD\mathbb{Q}_S^H$.
\end{thm}

\begin{thm}\label{thm:hodge-image}
Let $f:S\rightarrow V$ be projective, with $V$ smooth, and let $K\in D^b\operatorname{MHM}(S)$ be a bounded complex of graded-polarizable rational mixed Hodge modules. Its direct image has a strict filtered $\mathcal{D}$-module realization. Direct image commutes with duality and graded de Rham realization. If $F_{<0}K=0$, then
\[
F_{<0}\mathcal{H}^uf_*K=0,\qquad
F_0\mathcal{H}^uf_*K=H^u(Rf_*F_0K)
\]
for every integer $u$. If $F_0K$ is a sheaf $\mathcal{E}$ in degree zero, the last expression is $R^uf_*\mathcal{E}$.

For a single polarizable complex Hodge module, the same assertions hold, and its cohomological direct images are polarizable complex Hodge modules.
\end{thm}
\begin{proof}
The rational derived direct image and its compatibility with duality are given by \cite[Theorem~0.1 and Section~4.3]{Sai90}. For a single rational mixed Hodge module, projective strictness is \cite[Theorem~2.14]{Sai90}. For a bounded complex, apply \cite[Proposition~3.13]{CDO26} to $\DD(f_*K)$ and use biduality: the filtered realization of $f_*K$ is strict. Compatibility with the filtered $\mathcal{D}$-module direct image is recalled in \cite[Section~1.1, (1.1.2)--(1.1.5)]{FFS14}.

Locally on $V$, factor $f$ through a closed embedding into $\mathbb{P}^e\times V$ and the projection to $V$. In the right relative Spencer complex, the term of relative degree $-\alpha$ in $F_0$ uses $F_{-\alpha}K$. Hence only $\alpha=0$ contributes, and
\[
F_{<0}(f_*K)=0,\qquad F_0(f_*K)=Rf_*F_0K.
\]
Strictness permits taking cohomology in these identities. The Spencer construction also gives compatibility with graded de Rham realization \cite[Lemma~2.3.6 and Section~2.3.7]{Sai88}.

For a polarizable complex Hodge module, projective strictness and the assertion about its cohomological direct images are \cite[Theorem~14.3.1 and Corollary~14.3.2(1)--(2)]{SS26}. The same Spencer calculation gives the displayed formula. On a singular source we use a supported realization in the smooth ambient space above; closed embeddings are handled by filtered Kashiwara equivalence \cite[Propositions~9.6.2 and~14.2.9]{SS26}.
\end{proof}
\subsection{The comparison on a thickening}

Let
\[
\mathsf{A}:=\DD\mathbb{Q}_{\mathcal{S}}^H[-(d-1)]\in D^b\operatorname{MHM}(\mathcal{S}).
\]
By Lemma~\ref{lem:root-cover-basic}(4) and Theorem~\ref{thm:dubois-realization}, its filtered realization satisfies
\[
F_{<0}\mathsf{A}=0,\qquad F_0\mathsf{A}=\omega_{\mathcal{S}}.
\]
We compare this complex with the dual of the derived de Rham complex of a thickening. The latter construction will be made on one target chart with one chosen embedding.

\begin{consthm}\label{consthm:nonreduced-trace}
Notation and conditions as in Set-up~\ref{setup:root-stack}. Let $j\geq1$, let $T^{\circ}\subset\mathcal{S}_j$ be open, and let $u:T^{\circ}\rightarrow V$ be projective, where $V$ is a smooth variety or a smooth Deligne--Mumford stack. Put $S^{\circ}:=(T^{\circ})_{\mathrm{red}}$ and $h:=u|_{S^{\circ}}$.

Choose a connected affine \'etale chart $V'\rightarrow V$. Put $T:=T^{\circ}\times_VV'$, $S:=S^{\circ}\times_VV'$, and let $\iota:S\hookrightarrow T$ be the inclusion. Write $u':T\rightarrow V'$ and $h':S\rightarrow V'$ for the induced maps. Choose a closed embedding $i_T:T\hookrightarrow U:=\mathbb{P}^e\times V'$ over $V'$, after shrinking $V'$ if necessary, and put $i_S:=i_T\circ\iota$, $\pi_U:U\rightarrow V'$ and $n:=\dim U$. With this choice, we have the following construction and properties.
\begin{enumerate}
\item Let $C_T:=\dR_T$, with $F_{-s}C_T:=\mathrm{Fil}_{\mathrm H}^s\dR_T$, and let $C_T^{\vee}:=\DD_U(C_T)[-(d-1)]$. Since $F_0C_T=C_T$, we have $F_{<0}C_T^{\vee}=0$.
\item Define filtered complexes of right $\mathcal{D}_U$-modules by
\begin{equation}\label{eq:filtered-comparison-normalization}
\mathsf{A}_U:=\DR_U^{-1}\DD_U(\underline{\Omega}_S^{\bullet})[-(d-1)],\qquad
\mathsf{B}_{T,U}:=\DR_U^{-1}C_T^{\vee}.
\end{equation}
The comparison maps
\[
\Omega_U^{\bullet}\longrightarrow\dR_T\longrightarrow\dR_S\longrightarrow\underline{\Omega}_S^{\bullet}
\]
induce
\[
\mathsf{A}_U\longrightarrow\mathsf{B}_{T,U}\longrightarrow
\DR_U^{-1}\DD_U(\Omega_U^{\bullet})[-(d-1)].
\]
\item Let $\mathsf{A}^u:=\pi_{U+}\mathsf{A}_U$, $\mathsf{B}^u:=\pi_{U+}\mathsf{B}_{T,U}$, and let $\mathsf{A}^u\rightarrow\mathsf{B}^u$ be the induced filtered morphism on $V'$. The direct images are computed by the relative Spencer complex of $\pi_U$.
\item For $s\geq0$, we have
\[
\gr_s^FC_T^{\vee}\simeq i_{T*}\RHom_T(\mathbf{L}\!\bigwedge^s\LL_{T/\CC}[-s],\omega_T)\in D^{\geq-n}(\mathcal{O}_U).
\]
For every $k$, $F_k\mathsf{B}_{T,U}$ and $\gr_k^F\mathsf{B}_{T,U}$ are $\mathcal{O}_U$-complexes in $D^{\geq-n}(\mathcal{O}_U)$, and $\mathsf{B}_{T,U}\in D^{\geq-n}(\mathcal{D}_U)$. Their cohomology is supported on $S$. Moreover,
\[
\DR_{V'}(\mathsf{B}^u)\simeq R\pi_{U*}\DR_U(\mathsf{B}_{T,U})
\]
as filtered differential complexes.
\item $\mathsf{A}^u$ is the filtered realization of the pullback to $V'$ of $h_*(\mathsf{A}|_{S^{\circ}})$. The construction of $\mathsf{A}^u\rightarrow\mathsf{B}^u$ commutes with \'etale base change $V''\rightarrow V'$ when the chosen embedding is pulled back to $V''$. In particular, it commutes with restriction to open subsets of this target chart.
\item On $V'$, we have
\[
F_{<0}\mathsf{A}^u=F_{<0}\mathsf{B}^u=0,\qquad
F_0\mathsf{A}^u=Rh'_*\omega_S,\qquad F_0\mathsf{B}^u=Ru'_*\omega_T.
\]
The map between the zeroth pieces is induced by the trace $\omega_S\rightarrow\omega_T$, locally given by multiplication by $w^{j-1}$ in a frame of $\mathcal{R}$.
\item After forgetting the filtrations, $\mathsf{A}^u\rightarrow\mathsf{B}^u$ admits a retraction in the derived category of right $\mathcal{D}_{V'}$-modules.
\item For all integers $q$ and $k$, the map
\begin{equation}\label{eq:filtered-comparison-filtered-injection}
H^q(F_k\mathsf{A}^u)\longrightarrow H^q(F_k\mathsf{B}^u)
\end{equation}
is injective. For $q=1$ and $k=0$, this is the pullback to $V'$ of the canonical trace map
\[
R^1h_*\omega_{S^{\circ}}\longrightarrow R^1u_*\omega_{T^{\circ}}.
\]
Consequently, this trace map is injective on $V$.
\item The single rational mixed Hodge module $M:=\mathcal{H}^1h_*(\mathsf{A}|_{S^{\circ}})$ on $V$ has underlying filtered right $\mathcal{D}$-module $(\mathcal{M},F)$ satisfying
\[
F_{<0}\mathcal{M}=0,\qquad F_0\mathcal{M}=R^1h_*\omega_{S^{\circ}}.
\]
\end{enumerate}
\end{consthm}
\begin{proof}
Only (4)--(9) require proof. By Lemma~\ref{lem:root-cover-basic}(3)(4), $T$ and $S$ are Gorenstein and Cohen-Macaulay of pure dimension $d-1$, $S$ is Du Bois, and
\[
\omega_{T^{\circ}}=\mathcal{R}^{c+j-1}|_{T^{\circ}},\qquad
\omega_{S^{\circ}}=\mathcal{N}^c|_{S^{\circ}}.
\]
Since $U$ is smooth of dimension $n$, duality for $i_T$ gives $i_T^!\omega_U[n]=\omega_T[d-1]$. Thus, for $G\in D^-_{\mathrm{coh}}(T)$,
\begin{equation}\label{eq:filtered-comparison-finite-duality}
\RHom_U(i_{T*}G,\omega_U[n])[-(d-1)]\simeq i_{T*}\RHom_T(G,\omega_T).
\end{equation}
Theorem~\ref{thm:dubois-realization} and filtered duality identify $\mathsf{A}_U$ with the filtered realization of $i_{S*}(\DD\mathbb{Q}_S^H[-(d-1)])$. Hence $\mathsf{A}^u$ is the pullback of $h_*(\mathsf{A}|_{S^{\circ}})$, as asserted in (5).

We prove the bounds in (4). Locally, $\LL_{T/\CC}$ has a locally free model $P^{\bullet}$ in nonpositive degrees with $P^0=\Omega_U^1|_T$ of rank $n$. A summand of $\mathbf{L}\!\bigwedge^sP^{\bullet}$ contains at most $n$ factors from $P^0$. Each remaining factor has degree at most $-1$. Therefore
\begin{equation}\label{eq:filtered-comparison-exterior-bound}
\mathbf{L}\!\bigwedge^s\LL_{T/\CC}\in D^{\leq\min(0,n-s)}(T),\qquad
\mathbf{L}\!\bigwedge^s\LL_{T/\CC}[-s]\in D^{\leq\min(s,n)}(T).
\end{equation}
By Theorems~\ref{thm:derived-de-rham} and~\ref{thm:filtered-duality}(1), and \eqref{eq:filtered-comparison-finite-duality}, the $s$-th grade of $C_T^{\vee}$ is $i_{T*}Q_s$, where
\begin{equation}\label{eq:filtered-comparison-dual-bound}
Q_s:=\RHom_T(\mathbf{L}\!\bigwedge^s\LL_{T/\CC}[-s],\omega_T)
\in D^{\geq-\min(s,n)}(T)\subset D^{\geq-n}(T).
\end{equation}
For $k\geq0$, the order filtration on the linearization gives finite filtrations whose pieces are
\[
\begin{array}{ll}
\gr_k^F\mathsf{B}_{T,U}:& i_{T*}Q_s\otimes\operatorname{Sym}^rT_U,\quad s,r\geq0,\ s+r=k,\\[2pt]
F_k\mathsf{B}_{T,U}:& i_{T*}Q_s\otimes\operatorname{Sym}^rT_U,\quad s,r\geq0,\ s+r\leq k.
\end{array}
\]
These are complexes of $\mathcal{O}_U$-modules. Their lower bound is $-n$, independently of $k$, and they vanish away from $S$. Passing to the filtered union proves the assertion about the underlying $\mathcal{D}_U$-complex.

The term of relative Spencer degree $-\alpha$ in $F_k\pi_{U+}K$ is computed from
\[
\bigwedge^{\alpha}T_{U/V'}\otimes F_{k-\alpha}K,\qquad0\leq\alpha\leq e.
\]
For $K=\mathsf{B}_{T,U}$ the resulting total complex is bounded below by $-n-e$. The relative Spencer complex is finite, and $R\pi_{U*}$ has amplitude $[0,e]$ on quasi-coherent sheaves. Consequently, each fixed filtration index and cohomological degree can be computed using a bounded truncation. Composition of the relative and absolute Spencer complexes gives
\[
\DR_{V'}\pi_{U+}\mathsf{B}_{T,U}\simeq R\pi_{U*}\DR_U\mathsf{B}_{T,U}
\]
\cite[Lemma~2.3.6 and Section~2.3.7]{Sai88}. This proves (4).

For an \'etale map $V''\rightarrow V'$, use the embedding $T\times_{V'}V''\hookrightarrow U\times_{V'}V''$. The cotangent complexes, filtered duals and linearizations pull back along this flat map, and the finite Spencer complexes above pull back term by term. Proper flat base change for $\pi_U$ then identifies the direct images and their comparison morphism. This proves the remaining assertion of (5). We use $\mathsf{B}^u$ only for this chosen construction on $V'$.

We prove (6). The zeroth Hodge grades of $\dR_T$ and $\underline{\Omega}_S^{\bullet}$ are $\mathcal{O}_T$ and $\mathcal{O}_S$. Their duals have no negative filtration pieces, so only operators of order zero contribute to $F_0$ of the linearization. Thus
\begin{equation}\label{eq:filtered-comparison-source-zero}
F_{<0}\mathsf{A}_U=F_{<0}\mathsf{B}_{T,U}=0,\qquad
F_0\mathsf{A}_U=i_{S*}\omega_S,\qquad F_0\mathsf{B}_{T,U}=i_{T*}\omega_T.
\end{equation}
The map on $F_0$ is the dual of $\mathcal{O}_T\rightarrow\iota_*\mathcal{O}_S$, hence is the trace
\[
\omega_S=\mathcal{H}om_{\mathcal{O}_T}(\mathcal{O}_S,\omega_T)\longrightarrow\omega_T.
\]
Locally $T=(w^j=0)$ and $S=(w=0)$. Since $w$ is a nonzerodivisor in the ambient ring,
\[
\mathcal{H}om_{\mathcal{O}_T}(\mathcal{O}_S,\omega_T)=w^{j-1}\omega_T\subset\omega_T.
\]
Only relative Spencer degree zero contributes to $F_0$ after direct image. We obtain
\begin{equation}\label{eq:filtered-comparison-relative-zero}
F_{<0}\mathsf{A}^u=F_{<0}\mathsf{B}^u=0,\qquad
F_0\mathsf{A}^u=Rh'_*\omega_S,\qquad F_0\mathsf{B}^u=Ru'_*\omega_T,
\end{equation}
with comparison induced by the trace.

We next prove (7). After forgetting the filtrations,
\[
\mathsf{A}_U\simeq R\Gamma_S(\omega_U)[n-d+1],\qquad
\DR_U^{-1}\DD_U(\Omega_U^{\bullet})[-(d-1)]\simeq\omega_U[n-d+1].
\]
By (4), $\mathsf{B}_{T,U}$ is bounded below and supported on $S$. The support factorization in Theorem~\ref{thm:cdo-summand} therefore gives
\begin{equation}\label{eq:filtered-comparison-retraction}
R\Gamma_S(\omega_U)[n-d+1]\longrightarrow\mathsf{B}_{T,U}
\longrightarrow R\Gamma_S(\omega_U)[n-d+1]\longrightarrow\omega_U[n-d+1].
\end{equation}
The composition of the first two arrows, followed by the counit, is the counit: before dualizing, this is the restriction $\Omega_U^{\bullet}\rightarrow\underline{\Omega}_S^{\bullet}$. Uniqueness of the support factorization implies that the first two arrows compose to the identity. Applying $\pi_{U+}$ gives the required unfiltered retraction.

To prove (8), use strictness only for $\mathsf{A}^u$, which realizes a bounded complex of rational mixed Hodge modules by (5) and Theorem~\ref{thm:hodge-image}. Thus
\[
H^q(F_k\mathsf{A}^u)\longrightarrow H^q(\mathsf{A}^u)
\]
is injective. By (7), $H^q(\mathsf{A}^u)\rightarrow H^q(\mathsf{B}^u)$ is injective. Their composition factors as
\[
H^q(F_k\mathsf{A}^u)\longrightarrow H^q(F_k\mathsf{B}^u)\longrightarrow H^q(\mathsf{B}^u),
\]
so the first arrow is injective. No strictness assertion for $\mathsf{B}^u$ is needed. For $q=1$ and $k=0$, (6) identifies this arrow with $R^1h'_*\omega_S\rightarrow R^1u'_*\omega_T$. Proper flat base change identifies it with the pullback of the canonical trace on $V$. Since the construction applies on an \'etale covering, that canonical map is injective.

Finally, apply Theorem~\ref{thm:hodge-image} to $h_*(\mathsf{A}|_{S^{\circ}})$. Its cohomology in degree one is a single mixed Hodge module, and $F_0\mathsf{A}|_{S^{\circ}}=\omega_{S^{\circ}}$ gives (9).
\end{proof}

\begin{lem}\label{lem:filtered-comparison-first-grade}
Notation and conditions as in Construction-Theorem~\ref{consthm:nonreduced-trace}. Let $du':u'^*\Omega_{V'}^1\rightarrow\LL_{T/\CC}$ be the differential. Then:
\begin{enumerate}
\item On $V'$ there is a distinguished triangle
\begin{equation}\label{eq:filtered-comparison-first-grade}
Ru'_*\RHom_T(\LL_{T/\CC},\omega_T)\xrightarrow{(du')^{\vee}}
T_{V'}\otimes Ru'_*\omega_T\longrightarrow\gr_1^F\mathsf{B}^u\xrightarrow{+1}.
\end{equation}
The second arrow is $v\otimes s\rightarrow[s\cdot v]$. The triangle commutes with \'etale base change of the chosen construction, and the square
\[
\begin{tikzcd}
T_{V'}\otimes F_0\mathsf{A}^u \arrow[r] \arrow[d] & \gr_1^F\mathsf{A}^u \arrow[d]\\
T_{V'}\otimes F_0\mathsf{B}^u \arrow[r] & \gr_1^F\mathsf{B}^u
\end{tikzcd}
\]
commutes. The vertical arrows are induced by $\mathsf{A}^u\rightarrow\mathsf{B}^u$, and the horizontal arrows are the right $\mathcal{D}_{V'}$-actions modulo $F_0$.
\item Let $W\rightarrow\mathcal{X}$ be an \'etale chart on which $w$ is a local equation of $\mathcal{S}$, and suppose $T_W:=W\times_{\mathcal{X}}T^{\circ}=W\times_{\mathcal{X}}\mathcal{S}_j$. Then $T_W=(w^j=0)$, and the transitivity triangle is
\[
\mathcal{R}^{-j}|_{T_W}\longrightarrow\LL_{W/\CC}\otimes^{\mathbf{L}}\mathcal{O}_{T_W}
\longrightarrow\LL_{T^{\circ}/\CC}|_{T_W}\longrightarrow(\mathcal{R}^{-j}|_{T_W})[1].
\]
The first arrow sends the conormal generator to $d(w^j)=jw^{j-1}dw$.
\end{enumerate}
\end{lem}
\begin{proof}
By filtered linearization and Construction-Theorem~\ref{consthm:nonreduced-trace}(4),
\[
\begin{aligned}
\gr_1^F\DR_U(\mathsf{B}_{T,U})&\simeq i_{T*}\RHom_T(\LL_{T/\CC},\omega_T)[1],\\
\gr_1^F\DR_{V'}(\mathsf{B}^u)&\simeq Ru'_*\RHom_T(\LL_{T/\CC},\omega_T)[1].
\end{aligned}
\]
Since $F_{<0}\mathsf{B}^u=0$, the left hand side of the second line is the total complex
\[
\left[T_{V'}\otimes F_0\mathsf{B}^u\longrightarrow\gr_1^F\mathsf{B}^u\right]
\]
with terms in degrees $-1$ and $0$. Rotate its triangle and use $F_0\mathsf{B}^u=Ru'_*\omega_T$ to obtain \eqref{eq:filtered-comparison-first-grade}.

To identify the first arrow, compute $\dR_T$ by a simplicial polynomial resolution. The first Hodge differential
\[
\gr_{\mathrm H}^0\dR_T\longrightarrow\gr_{\mathrm H}^1\dR_T[1]
\]
is the universal derivation $\mathcal{O}_T\rightarrow\LL_{T/\CC}$, a differential operator. For a function $x$ on $V'$,
\[
d(u'^*x)=du'(u'^*dx).
\]
Dualizing and linearizing this equality identifies the first, $\mathcal{O}_{V'}$-linear arrow with $(du')^{\vee}$. Linearization gives the symbol $v\otimes s\rightarrow[s\cdot v]$. Since $\mathsf{A}^u\rightarrow\mathsf{B}^u$ is filtered and $\mathcal{D}_{V'}$-linear, it commutes with these symbols. The calculations pull back under \'etale maps of $V'$, proving (1).

For (2), $w^j$ is a nonzerodivisor on $W$, so the conormal sheaf of $T_W\hookrightarrow W$ is $(w^j)/(w^{2j})=\mathcal{R}^{-j}|_{T_W}$. Its map to the cotangent complex is the universal derivation, which sends $w^j$ to $jw^{j-1}dw$.
\end{proof}
\subsection{The obstruction class lies in the kernel of the Kodaira-Spencer map}

We now use the comparison to study $\eta_j$. The injection on $F_1$ from Construction-Theorem~\ref{consthm:nonreduced-trace} does not by itself give an injection on $\gr_1^F$. We will obtain the required injection by using the traces from $\mathcal{S}_{\nu}$ to $\mathcal{S}_j$ for $1\leq\nu\leq j$.

\begin{thm}\label{thm:tangential-comparison}
Notation and conditions as in Construction-Theorem \ref{consthm:obstruction}. Let $j\geq b$, let $\sss$ be a $j$-order lift, and put
\[
M:=\mathcal{H}^1h_*\mathsf{A},
\]
with underlying filtered right $\mathcal{D}_{\mathcal{V}}$-module $(\mathcal{M},F)$. Then:
\begin{enumerate}
\item Choose an affine \'etale chart $W\rightarrow\mathcal{V}$ trivializing $\mathcal{T}$, and set $Y_W:=Y\times_{\mathcal{V}}W=W\times\mathbb{A}^1$. Pull back $\psi_{j,\sss}$ to $Y_W$ and choose one projective embedding over $Y_W$ as in Construction-Theorem~\ref{consthm:nonreduced-trace}. Let $\mathsf{B}_Y$ denote the resulting filtered complex on $Y_W$. Suppressing pullback notation, the map
\[
z_*\gr_1^F\mathcal{M}\longrightarrow\gr_1^F(z_+\mathcal{M})
\longrightarrow H^1(\gr_1^F\mathsf{B}_Y)
\]
is injective.
\item For the Kodaira-Spencer map $\sigma:T_{\mathcal{V}}\otimes F_0\mathcal{M}\rightarrow\gr_1^F\mathcal{M}$ of Definition-Lemma~\ref{deflem:kodaira-spencer}, we have
\[
\eta_j\in H^0(\mathcal{V},T_{\mathcal{V}}\otimes F_0\mathcal{M}\otimes\mathcal{T}^{-(c+j)}),
\qquad (\sigma\otimes\mathcal{T}^{-(c+j)})(\eta_j)=0.
\]
\end{enumerate}
\end{thm}
\begin{proof}
We work on a chart $W$ as in (1). All the spaces, maps and line bundles in the proof of (1) are pulled back to $Y_W$ or $W$; we omit these pullbacks from the notation. Thus $h:\mathcal{S}\rightarrow W$, $z:W\hookrightarrow Y_W$, and $\psi_j:=\psi_{j,\sss}:\mathcal{S}_j\rightarrow Y_W$. These base changes are projective by Lemma~\ref{lem:root-neighborhoods-framed-projectivity}. Write $\pi:Y_W\rightarrow W$ for the projection, and keep $\theta$ for the tautological section. We have
\[
\psi_j^*\theta=w|_{\mathcal{S}_j},\qquad
\psi_j^*\pi^*\mathcal{T}=\mathcal{R}|_{\mathcal{S}_j},\qquad
\psi_1=z\circ h.
\]
Here $\psi_{\nu}:=\psi_j|_{\mathcal{S}_{\nu}}$ for $1\leq\nu\leq j$.

For $1\leq\nu\leq j$, the trace
\[
\operatorname{tr}_{\nu,j}:\omega_{\mathcal{S}_{\nu}}\longrightarrow\omega_{\mathcal{S}_j}
\]
identifies $\omega_{\mathcal{S}_{\nu}}$ with the subsheaf annihilated by $w^{\nu}$. In a frame of $\mathcal{R}$, it is multiplication by $w^{j-\nu}$. Adjunction therefore gives, for $2\leq\nu\leq j$,
\begin{equation}\label{eq:filtered-comparison-trace-layer}
0\longrightarrow\omega_{\mathcal{S}_{\nu-1}}\xrightarrow{\operatorname{tr}_{\nu-1,\nu}}
\omega_{\mathcal{S}_{\nu}}\longrightarrow\omega_{\mathcal{S}}\otimes\mathcal{N}^{\nu-1}\longrightarrow0,
\end{equation}
where the last arrow is restriction. Put $D_{\nu}:=R^1\psi_{\nu*}\omega_{\mathcal{S}_{\nu}}$. The long exact sequence and $h^*\mathcal{T}=\mathcal{N}$ give
\begin{equation}\label{eq:filtered-comparison-trace-les}
\begin{aligned}
\psi_{\nu*}\omega_{\mathcal{S}_{\nu}}&\longrightarrow z_*(h_*\omega_{\mathcal{S}}\otimes\mathcal{T}^{\nu-1})
\xrightarrow{\delta_{\nu}}D_{\nu-1}\xrightarrow{\ell_{\nu}}D_{\nu}\\
&\xrightarrow{\rho_{\nu}}z_*(R^1h_*\omega_{\mathcal{S}}\otimes\mathcal{T}^{\nu-1})
\longrightarrow R^2\psi_{\nu-1*}\omega_{\mathcal{S}_{\nu-1}}.
\end{aligned}
\end{equation}
In particular,
\begin{equation}\label{eq:filtered-comparison-layer-kernels}
\ker\ell_{\nu}=\operatorname{im}\delta_{\nu},\qquad
\ker\rho_{\nu}=\operatorname{im}\ell_{\nu}.
\end{equation}

Let $\lambda_{\nu}:D_{\nu}\rightarrow D_j$ be induced by $\operatorname{tr}_{\nu,j}$. Then $\lambda_j=\mathrm{id}$ and $\lambda_{\nu-1}=\lambda_{\nu}\ell_{\nu}$. For a local section $y\in D_{\nu}$, we have
\begin{equation}\label{eq:filtered-comparison-normal-identity}
\theta^{\nu-1}\lambda_{\nu}(y)
=(\lambda_1\otimes\pi^*\mathcal{T}^{\nu-1})\rho_{\nu}(y)
\quad\text{in }D_j\otimes\pi^*\mathcal{T}^{\nu-1}.
\end{equation}
Indeed, before direct image, the left hand side multiplies a local section of $\omega_{\mathcal{S}_{\nu}}$ by
\[
w^{j-\nu}w^{\nu-1}=w^{j-1}.
\]
This factors through restriction to $\omega_{\mathcal{S}}\otimes\mathcal{N}^{\nu-1}$ and then through $\operatorname{tr}_{1,j}\otimes\mathcal{N}^{\nu-1}$, which is the right hand side. Applying direct image and the projection formula proves the identity. The sequences and maps just defined are pullbacks of the corresponding global trace maps by proper flat base change.

Let $\mathsf{A}_Y:=z_+h_*\mathsf{A}$ on this chart. Construction-Theorem~\ref{consthm:nonreduced-trace} supplies a filtered morphism $\mathsf{A}_Y\rightarrow\mathsf{B}_Y$ with
\[
F_0\mathsf{A}_Y=z_*Rh_*\omega_{\mathcal{S}},\qquad
F_0\mathsf{B}_Y=R\psi_{j*}\omega_{\mathcal{S}_j},\qquad
\mathcal{H}^1\mathsf{A}_Y=z_+\mathcal{M}.
\]
The complex $\mathsf{A}_Y$ is strict. Hence $H^1(F_0\mathsf{A}_Y)\rightarrow H^1(F_1\mathsf{A}_Y)$ is injective; composing with the injection in Construction-Theorem~\ref{consthm:nonreduced-trace}(8) gives
\begin{equation}\label{eq:filtered-comparison-trace-detector}
D_1=z_*R^1h_*\omega_{\mathcal{S}}=H^1(F_0\mathsf{A}_Y)
\lhook\joinrel\longrightarrow H^1(F_1\mathsf{B}_Y).
\end{equation}
This map factors through $\lambda_1:D_1\rightarrow D_j$.

We prove (1). In the chosen frame, write $\theta=t$. By the closed-embedding formula in Definition~\ref{defn:direct-image},
\[
F_1(z_+\mathcal{M})=z_*F_1\mathcal{M}\oplus z_*F_0\mathcal{M}\,\partial_t.
\]
Multiplication by $t$ kills $z_*F_1\mathcal{M}$, while
\[
(\gamma\partial_t)t=\gamma\qquad(\gamma\in F_0\mathcal{M}),
\]
since $\partial_tt=t\partial_t+1$ and $\gamma t=0$. Thus $z_*F_1\mathcal{M}$ is the kernel of multiplication by $t$. Quotienting by $z_*F_0\mathcal{M}$ gives $z_*\gr_1^F\mathcal{M}\hookrightarrow\gr_1^F(z_+\mathcal{M})$.

Let $x$ be a local section of $z_*\gr_1^F\mathcal{M}$ with zero image in $H^1(\gr_1^F\mathsf{B}_Y)$. Lift it locally to
\[
\xi\in z_*F_1\mathcal{M}\subset H^1(F_1\mathsf{A}_Y),\qquad\theta\xi=0.
\]
The exact sequence
\[
H^0(\gr_1^F\mathsf{B}_Y)\longrightarrow D_j\longrightarrow H^1(F_1\mathsf{B}_Y)
\longrightarrow H^1(\gr_1^F\mathsf{B}_Y)
\]
shows that the image of $\xi$ is the image of some local section $y_j\in D_j$.

Suppose $2\leq\nu\leq j$ and $y_{\nu}\in D_{\nu}$ satisfies $\lambda_{\nu}(y_{\nu})=y_j$. Its image in $H^1(F_1\mathsf{B}_Y)$ is the image of $\xi$, which is killed by $\theta$. Multiplying by $\theta^{\nu-1}$ and using \eqref{eq:filtered-comparison-normal-identity}, we find that the image of $\rho_{\nu}(y_{\nu})$ under \eqref{eq:filtered-comparison-trace-detector}, tensored with $\pi^*\mathcal{T}^{\nu-1}$, is zero. That map is injective, so $\rho_{\nu}(y_{\nu})=0$. By \eqref{eq:filtered-comparison-layer-kernels}, we can shrink the open set and choose $y_{\nu-1}$ with
\[
\ell_{\nu}(y_{\nu-1})=y_{\nu},\qquad
\lambda_{\nu-1}(y_{\nu-1})=y_j.
\]
Starting with $y_j$ and repeating this step gives $y_1\in D_1$ with $\lambda_1(y_1)=y_j$. Choices differ by $\ker\ell_{\nu}=\operatorname{im}\delta_{\nu}$ and have the same image in $D_j$.

Now $\xi$ and $y_1$ have the same image in $H^1(F_1\mathsf{B}_Y)$. The map $H^1(F_1\mathsf{A}_Y)\rightarrow H^1(F_1\mathsf{B}_Y)$ is injective, so
\[
\xi=y_1\in H^1(F_0\mathsf{A}_Y)=z_*F_0\mathcal{M}.
\]
Thus $x=0$. This proves (1).

We prove (2). Put $k:=c+j$. Since $\omega_{\mathcal{S}}=\mathcal{N}^c$ and $h^*\mathcal{T}=\mathcal{N}$, Construction-Theorem~\ref{consthm:nonreduced-trace}(9) and the projection formula give, globally,
\begin{equation}\label{eq:filtered-comparison-coefficient}
\mathcal{N}^{-j}=\omega_{\mathcal{S}}\otimes h^*\mathcal{T}^{-k},\qquad
R^1h_*\mathcal{N}^{-j}=F_0\mathcal{M}\otimes\mathcal{T}^{-k}.
\end{equation}
This identifies the group containing $\eta_j$ with the group in (2). It remains to show that its image under $\sigma$ vanishes, which can be checked on the charts $W$.

Continue on the chosen chart and put
\[
\mathsf{E}_j:=R\psi_{j*}\RHom_{\mathcal{S}_j}(\LL_{\mathcal{S}_j/\CC},\omega_{\mathcal{S}_j}),
\qquad\mathcal{L}_Y:=\pi^*\mathcal{T}^{-k}.
\]
Using \eqref{eq:filtered-comparison-coefficient}, tensor $\operatorname{tr}_{1,j}$ with $\psi_j^*\mathcal{L}_Y$ to obtain
\[
\tau:=\operatorname{tr}_{1,j}\otimes\psi_j^*\mathcal{L}_Y:
i_{j*}\mathcal{N}^{-j}\longrightarrow\omega_{\mathcal{S}_j}\otimes\psi_j^*\mathcal{L}_Y.
\]
Compose $\alpha_j$ with $\tau$. Since $Y_W$ is affine, the resulting Ext class defines a section
\[
\widetilde\alpha_j\in H^0(Y_W,H^1(\mathsf{E}_j)\otimes\mathcal{L}_Y).
\]
Here $\LL_{\mathcal{S}_j/\CC}\in D^{\leq0}$, so $\mathsf{E}_j\in D^{\geq0}$; the identification is also the degree-one Leray edge for this complex.

Lemma~\ref{lem:filtered-comparison-first-grade}(1), tensored with $\mathcal{L}_Y$, gives a sequence of sheaves whose successive composition is zero:
\[
H^1(\mathsf{E}_j)\otimes\mathcal{L}_Y\xrightarrow{(d\psi_j)^{\vee}}
T_{Y_W}\otimes D_j\otimes\mathcal{L}_Y\longrightarrow H^1(\gr_1^F\mathsf{B}_Y)\otimes\mathcal{L}_Y.
\]
By the definition $\operatorname{ob}_j(\psi_j)=\alpha_j\circ d\psi_j$, functoriality of the edge map, and Construction-Theorem~\ref{consthm:oj}(4),
\[
(d\psi_j)^{\vee}(\widetilde\alpha_j)
=(\mathrm{id}\otimes\lambda_1\otimes\mathrm{id})\bigl(z_*(\eta_j,0)\bigr).
\]
The cotangent complex in this calculation is that of $\mathcal{S}_j$: its conormal map contains $d(w^j)=jw^{j-1}dw$, as in Lemma~\ref{lem:filtered-comparison-first-grade}(2).

Apply the second arrow of the preceding sequence to the right hand side of this equality. By the symbol square in Lemma~\ref{lem:filtered-comparison-first-grade}(1), its image is the image of
\[
z_*(\sigma\otimes\mathcal{T}^{-k})(\eta_j)
\]
under the map of (1), tensored with $\mathcal{L}_Y$. Exactness makes this image zero. By (1), that map is injective, so $(\sigma\otimes\mathcal{T}^{-k})(\eta_j)=0$ on $W$. These charts cover $\mathcal{V}$, which proves (2).
\end{proof}
\section{A vanishing theorem on root stacks}\label{sec:07-root-vanishing}

By Theorem~\ref{thm:tangential-comparison}, $\eta_j$ belongs to
\[
H^0(\mathcal{V},\ker\sigma\otimes\mathcal{T}^{-(c+j)}).
\]
We show that this group is zero. We will apply the positive-twist vanishing theorem of Davis--Vilonen to the dual Hodge module and then use Serre duality to obtain the required negative-twist vanishing.

\subsection{Twisted Hodge modules}

Let $\pi_T:\widetilde U\rightarrow U$ be a torsor under an algebraic torus $H$. We use the $\lambda$-twisted complex mixed Hodge modules of \cite[Definitions~2.3 and~2.5]{DV23}. Such an object is a single complex mixed Hodge module on $\widetilde U$, in the sense of \cite[Section~2.1]{DV23}, with a weak $H$-action. If $i$ is the differential action and $\theta$ the weak infinitesimal action, the twist condition is
\[
i(\xi)-\theta(\xi)=\lambda(\xi)\,\mathrm{id}\qquad(\xi\in\operatorname{Lie}(H)),
\]
for $\lambda\in\operatorname{Lie}(H)^{\vee}$. For a character $\chi$ of $H$, the associated line bundle is
\[
\mathcal{O}(\chi):=(\pi_{T*}\mathcal{O}_{\widetilde U}\otimes\mathbb{C}_{-\chi})^H
\]
\cite[Section~4.1]{DV23}. We denote the graded de Rham complexes on $U$ by $C_p$, with the cohomological convention of Definition~\ref{defn:de-rham-complex}.

\begin{thm}[{\cite[Proposition~4.1 and Theorem~4.2]{DV23}}]\label{thm:dv-vanishing}
Let $U$ be a smooth projective variety, let $\widetilde U\rightarrow U$ be a torsor under an algebraic torus $H$, and let $M$ be a $\lambda$-twisted complex mixed Hodge module. Suppose $\lambda$ is a positive real linear combination of characters whose associated line bundles on $U$ are ample. Then
\[
\HH^i(U,C_p(M))=0\qquad(p\in\mathbb{Z},\ i>0).
\]
\end{thm}

The theorem allows modules supported on proper closed subsets. We will use it for polarizable complex Hodge modules. The pure definition in \cite[Section~2.1, pp.~10--11]{DV23} is explicitly compared there with the filtered triples of \cite{SS26}: a polarization identifies the second filtration with the image of the first filtration shifted by the weight. Thus a polarizable complex Hodge module of weight $w$, equipped with $W_{w-1}=0$ and $W_w=M$, is an object of the category used in Theorem~\ref{thm:dv-vanishing}. This uses only the pure categories.

The conventions in \cite{DV23} are for left $\mathcal{D}$-modules. We pass to our right convention by
\[
F_b\mathcal{M}^{\mathrm r}=\omega_U\otimes F_{b+\dim U}\mathcal{M}^{\mathrm l}.
\]
This conversion gives the same graded de Rham complexes.

\begin{thm}\label{thm:hodge-duality-vanishing}
Let $U$ be smooth and let $M$ be a rational mixed Hodge module or a polarizable complex Hodge module on $U$. Then:
\begin{enumerate}
\item Its complex-linear dual $\DD M$ is a Hodge module in the same respective category, and its underlying filtered right $\mathcal{D}_U$-module is the filtered dual. For every integer $p$,
\[
C_p(\DD M)\simeq\RHom_U(C_{-p}(M),\omega_U[\dim U]).
\]
\item If $M$ is supported on a reduced closed subvariety $Z\subset U$, then $C_p(M)$ is a complex of coherent $\mathcal{O}_Z$-modules, and its quasi-isomorphism class on $Z$ is independent of the smooth ambient embedding.
\end{enumerate}
\end{thm}
\begin{proof}
For rational mixed Hodge modules, duality is \cite[Proposition~2.6]{Sai90}. The coherent formula in (1) is filtered differential duality \cite[Sections~2.4.3 and~2.4.11]{Sai88}; see also \cite[Lemma~7.4]{Sch16}. For polarizable complex Hodge modules, strict holonomicity, nondegeneracy and the extension of duality to filtered triples are established in \cite[Theorem~14.7.2, Theorem~14.7.3 and Corollary~14.7.4]{SS26}. The resulting complex-linear dual is again a polarizable complex Hodge module by \cite[Proposition~16.3.5]{SS26}. Applying the same filtered differential duality to its components gives (1). We keep $\DD M$ in this formula; a complex character summand need not be isomorphic to its dual up to a Tate twist.

The rational assertion in (2) is \cite[Lemma~7.3]{Sch16}; see also \cite[Theorem~28(1)]{Pop16}. We give the ambient comparison for the complex case. If a local function $g$ vanishes on $Z$, filtered regularity along $g=0$ gives
\[
gF_k\mathcal{M}\subset F_{k-1}\mathcal{M}.
\]
This is the filtered $\mathcal{D}$-module calculation of \cite[Lemma~3.2.6]{Sai88}; in the complex category its hypotheses follow from \cite[Definition~14.2.2(1)$_g$ and Lemma~9.6.3(2)]{SS26}. Hence $\mathcal{I}_Z\gr^F\mathcal{M}=0$, so the terms of $C_p(M)$ are $\mathcal{O}_Z$-modules.

Work locally with two embeddings $Z\hookrightarrow\mathbb{A}^a$ and $Z\hookrightarrow\mathbb{A}^b$. Extend the coordinate functions of the second embedding to the first ambient space, and conversely. The resulting graph embeddings
\[
\gamma_f:\mathbb{A}^a\hookrightarrow\mathbb{A}^a\times\mathbb{A}^b,
\qquad
\gamma_g:\mathbb{A}^b\hookrightarrow\mathbb{A}^a\times\mathbb{A}^b
\]
are closed embeddings between smooth varieties and agree on $Z$. Their direct images remain polarizable complex Hodge modules \cite[Propositions~9.6.2 and~14.2.9]{SS26}. They identify on the common support: decompose into strict-support summands \cite[Theorem~14.2.19(1)]{SS26} and use the uniqueness of the extension of the same variation on each dense smooth open subset \cite[Theorem~16.2.1]{SS26}.

For a graph embedding $\gamma$ between smooth varieties of codimension $c$, the normal-coordinate formula is
\[
\gr^F(\gamma_+\mathcal{M})=\gr^F\mathcal{M}[\xi_1,\dots,\xi_c].
\]
The normal part of its graded Spencer complex is the Koszul complex on $\xi_1,\dots,\xi_c$. Its augmentation therefore gives
\[
C_p(\gamma_+M)\simeq\gamma_*C_p(M)
\]
as complexes on $Z$ \cite[Lemma~2.3.6 and Section~2.3.7]{Sai88}. Applying this to both graphs identifies the two complexes $C_p(M)$. The graph identities and composition of direct images make these identifications compatible with a third embedding and with restriction. They therefore glue. This proves (2).
\end{proof}

\subsection{The vanishing theorem}

\begin{setup}\label{setup:root-vanishing}
Let $W$ be a reduced projective scheme, let $L$ be a line bundle on $W$, and let $n\geq1$. Let $\rho:\mathcal{G}\rightarrow W$ be the stack of $n$-th roots of $L$, with tautological line bundle $\mathcal{H}$, so $\mathcal{H}^n=\rho^*L$ \cite[Definition~2.2.6]{Cad07}. On an open set $W'$ trivializing $L$, we have $\rho^{-1}(W')=W'\times B\mu_n$ \cite[Example~2.2.7]{Cad07}. A rational mixed Hodge module on $\mathcal{G}$ is given on these charts by rational mixed Hodge modules with a $\mu_n$-action and descent isomorphisms of the full Hodge objects. If $W$ is singular, we compute $C_p$ in smooth ambient varieties, using Theorem~\ref{thm:hodge-duality-vanishing}(2).
\end{setup}

\begin{thm}\label{thm:negative-root-vanishing}
Notation and conditions as in Set-up~\ref{setup:root-vanishing}. Assume that $L$ is ample. Let $P$ be a line bundle on $\mathcal{G}$ such that $P^s\cong\rho^*J$ for some $s>0$ and some ample line bundle $J$ on $W$. Then, for every rational mixed Hodge module $M$ on $\mathcal{G}$,
\[
\HH^i(\mathcal{G},C_p(M)\otimes P^{-1})=0\qquad(p\in\mathbb{Z},\ i<0).
\]
If $W$ is smooth and $F_{<0}\mathcal{M}=0$, then
\[
H^0\left(\mathcal{G},\ker\left[T_{\mathcal{G}}\otimes F_0\mathcal{M}\longrightarrow\gr_1^F\mathcal{M}\right]\otimes P^{-1}\right)=0,
\]
where the displayed map is the Kodaira-Spencer map.
\end{thm}
\begin{proof}
First let $U$ be smooth projective and let $K$ be a single $\lambda$-twisted complex mixed Hodge module on an $H$-torsor $\pi_T:\widetilde U\rightarrow U$. Put $e_H:=\dim H$. We use the complex-linear dual, normalized on $U$; its representative on $\widetilde U$ is
\[
\DD K:=\DD_{\widetilde U}K(-e_H).
\]
To check this normalization, recall that smooth pullback is $\pi_T^{\circ}Q=\pi_T^*Q[e_H]$ \cite[Notation~2.2]{DV23}. On a trivial torsor,
\[
\DD_{\widetilde U}(\pi_T^*Q[e_H])
\simeq\pi_T^*(\DD_UQ)[e_H](e_H).
\]
Thus the Tate twist $(-e_H)$ removes the contribution from the fiber. The local description of twisted objects in \cite[Lemmas~2.7--2.8]{DV23} gives the same formula with the rank-one fiber factor replaced by its dual.

The local duals have opposite twist. If two trivializations differ by $t_{ab}:U_{ab}\rightarrow H$, the horizontal operators differ by the scalar term $\lambda(d\log t_{ab})$. The dual action is determined by
\[
\xi\langle\varphi,v\rangle=\langle\xi\varphi,v\rangle+\langle\varphi,\xi v\rangle.
\]
Its scalar term is therefore $-\lambda(d\log t_{ab})$. Filtered differential duality, applied in each trivialization, now glues to
\begin{equation}\label{eq:root-vanishing-twisted-duality}
C_{-p}(\DD K)\simeq\RHom_U(C_p(K),\omega_U[\dim U]),\qquad
\lambda(\DD K)=-\lambda(K).
\end{equation}
For the local identity, let $C$ be the filtered de Rham complex on a trivializing open set. Resolve $\omega_U[\dim U]$ by a complex $I^{\bullet}$ of right $\mathcal{D}_U$-modules injective as $\mathcal{O}_U$-modules, and put the annihilator filtration on $\mathcal{H}om(C,I^{\bullet})$. Then
\[
\gr_{-p}^F\mathcal{H}om(C,I^{\bullet})
=\mathcal{H}om(\gr_p^FC,I^{\bullet}).
\]
The normalized torsor descent leaves precisely $\omega_U[\dim U]$ in this formula, with no further cohomological shift. Duality in the complex mixed category is as in \cite[Section~2.1]{DV23}; for the pure objects used below, the filtered dual is the one in Theorem~\ref{thm:hodge-duality-vanishing}(1). If $\mu_n$ acts on $v$ by $\zeta^{\alpha}$, it acts on the dual functional by $\zeta^{-\alpha}$, so this duality also reverses finite characters.

Suppose $-\lambda$ is a positive real linear combination of characters with ample associated line bundles. By Serre duality and \eqref{eq:root-vanishing-twisted-duality},
\[
\HH^i(U,C_p(K))^{\vee}\simeq\HH^{-i}(U,C_{-p}(\DD K)).
\]
Theorem~\ref{thm:dv-vanishing} applies to $\DD K$ and kills the right hand side for $i<0$. We have proved
\begin{equation}\label{eq:root-vanishing-negative-twist}
\HH^i(U,C_p(K))=0\qquad(p\in\mathbb{Z},\ i<0).
\end{equation}

Return to $\mathcal{G}\rightarrow W$. The stabilizer acts trivially on $P^n$, so $P^n$ descends to a line bundle $L_P$ on $W$. Since $\rho_*\mathcal{O}_{\mathcal{G}}=\mathcal{O}_W$, pullback by $\rho$ is faithful on line bundles. Therefore
\begin{equation}\label{eq:root-vanishing-positive-line}
P^n\cong\rho^*L_P,\qquad L_P^s\cong J^n.
\end{equation}
In particular, $L_P$ is ample.

We reduce $M$ to pure complex Hodge modules. Its finite weight filtration is preserved by the stabilizer action and the descent isomorphisms. Morphisms of rational mixed Hodge modules are strict for $F$ and $W$ \cite[Proposition~5.1.14]{Sai88}. Hence each weight sequence gives an exact sequence of complexes
\[
0\longrightarrow C_p(W_{a-1}M)\otimes P^{-1}\longrightarrow C_p(W_aM)\otimes P^{-1}
\longrightarrow C_p(\gr_a^WM)\otimes P^{-1}\longrightarrow0.
\]
Induction on the finitely many weights reduces the assertion to the rational polarizable pure modules $\gr_a^WM$. Decompose each into its canonical strict-support summands, and apply \cite[Corollary~16.3.9]{SS26} to each summand. This identifies the rational pure object with a polarizable complex Hodge module having the same filtered $\mathcal{D}$-module, after forgetting the rational structure. The comparison is through the same generic variation and its unique extension, so it respects morphisms, stabilizer actions and descent. Thus it suffices to prove the assertion for a polarizable complex Hodge module of one weight. Such an object belongs to the pure subcategory of \cite[Section~2.1]{DV23}, as explained above.

Assume first that $W$ is smooth and connected. The stabilizer character of $P$ is constant. Cancelling a power of $\mathcal{H}$ gives
\[
P\cong\rho^*L_1\otimes\mathcal{H}^k,\qquad
k=\alpha+tn,\qquad0\leq\alpha<n,
\]
for a line bundle $L_1$ on $W$. We have $L_P=L_1^n\otimes L^k$. The idempotents
\[
e_{\beta}:=\frac1n\sum_{\zeta\in\mu_n}\zeta^{-\beta}\,\zeta
\]
acting on $M$ give $M=\bigoplus_{\beta=0}^{n-1}M_{\beta}$. They are morphisms in the complex category, so the summands are polarizable complex Hodge modules \cite[Sections~14.2.7--14.2.8 and Theorem~14.2.19(4)]{SS26}. On $W'\times B\mu_n$, the functor $\rho_*$ takes invariants, which is exact in characteristic zero. Only $M_{\alpha}$ contributes after tensoring by $P^{-1}$.

We compute that contribution on
\[
\mathcal{G}=[L^{\times}/\mathbb{G}_m],\qquad t_0\cdot v=t_0^{-n}v.
\]
The pullback of $\mathcal{H}$ is the trivial line with action $t_0\cdot z=t_0z$, and $z^nv$ identifies its $n$-th power with $L$. In this calculation only, use left $\mathcal{D}$-modules as in \cite{DV23}. Pull $M_{\alpha}$ back by normalized smooth pullback. Let $i_0$ and $\theta_0$ be its differential and weak infinitesimal actions for $t_0$, and let $i$ be the differential action for ordinary scaling $h_0=t_0^{-n}$. Strong equivariance gives
\[
\theta_0=i_0=-n i.
\]
Change the weak action by $t_0^{-\alpha}$. The new infinitesimal action is $\theta_0-\alpha$, and the subgroup $\mu_n$ acts trivially. Thus this weak action factors through $h_0$. Its infinitesimal action $\theta$ and twist are
\[
\theta=-\frac{\theta_0-\alpha}{n}=i+\frac{\alpha}{n},\qquad
\lambda=i-\theta=-\frac{\alpha}{n}.
\]
The underlying Hodge module on $L^{\times}$ has not changed.

Let $N_{\alpha}$ denote the descended twisted Hodge module on $W$. More generally, for an $H$-torsor of dimension $e$ and an upstairs object $K$, the left filtrations descend by
\[
F_bN=(\pi_{T*}F_bK)^H,\qquad
W_aN=(\pi_{T*}W_{a+e}K)^H
\]
\cite[Proposition~2.9 and (2.6)]{DV23}. Normalized smooth pullback leaves the left Hodge filtration unchanged and shifts the weight by $e$ \cite[Notation~2.2]{DV23}; the displayed formula removes this weight shift. Scalar changes of the weak action preserve both Hodge filtrations and the weight filtration. Descent is exact. Converting back to right modules, we use $F_bN^{\mathrm r}=\omega_W\otimes F_{b+\dim W}N^{\mathrm l}$.

The cotangent bundle of $\mathcal{G}$ is $\rho^*\Omega_W^1$. In local frames, the horizontal operators before and after descent differ by an order-zero scalar logarithmic derivative. This disappears from the graded Spencer differential, giving
\[
\rho_*(C_p(M_{\alpha})\otimes\mathcal{H}^{-\alpha})\simeq C_p(N_{\alpha}).
\]
Since $\rho_*$ is exact and $\mathcal{H}^{-tn}=\rho^*L^{-t}$, we obtain
\begin{equation}\label{eq:root-vanishing-character-descent}
R\rho_*(C_p(M)\otimes P^{-1})\simeq C_p(N_{\alpha})\otimes L_1^{-1}\otimes L^{-t}.
\end{equation}
On the product torsor $L^{\times}\times_WL_1^{\times}$, the right hand side comes from the object with twist
\[
\lambda=\left(-\frac{\alpha}{n}-t,-1\right)=\left(-\frac{k}{n},-1\right).
\]
Its opposite is $\frac1n(k,n)$. The character $(k,n)$ has associated line bundle
\[
L^k\otimes L_1^n=L_P,
\]
which is ample by \eqref{eq:root-vanishing-positive-line}. Therefore \eqref{eq:root-vanishing-negative-twist} applies, even if $k$ is negative. This proves the assertion when $W$ is smooth. The calculation also includes $n=1$, with $\alpha=0$.

We explain the direct images needed to treat singular $W$. For a morphism $\nu:\mathcal{G}_1\rightarrow\mathcal{G}_2$ of stacks of roots of line bundles over the same base, choose \'etale local presentations $W_0\times BK_1$ and $W_0\times BK_2$. If $\varphi:K_1\rightarrow K_2$ is the induced homomorphism, define the image of a polarizable complex Hodge module $Q$ by
\[
\nu_*Q=\operatorname{Ind}_{\operatorname{im}\varphi}^{K_2}(Q^{\ker\varphi}).
\]
Invariants under the finite kernel are the image of an idempotent, and induction is a finite direct sum. These operations are exact and preserve polarizable complex Hodge modules \cite[Sections~14.2.7--14.2.8]{SS26}. They commute with both filtrations and with $C_p$. Changes of trivialization act by the given isomorphisms on $Q$, hence give descent isomorphisms on $Q^{\ker\varphi}$ and its induced module. For a representable finite \'etale map, this agrees with the finite direct sum on \'etale charts. For a closed embedding of stacks of roots of line bundles, we use the closed-embedding image on charts in the complex pure category \cite[Propositions~9.6.2 and~14.2.9, Corollary~14.3.2]{SS26}. On singular charts, Theorem~\ref{thm:hodge-duality-vanishing}(2) supplies the ambient comparison. For all these maps,
\begin{equation}\label{eq:root-vanishing-exact-image}
C_p(\nu_*Q)\simeq R\nu_*C_p(Q),
\end{equation}
and the projection formula holds. There is no cohomological shift. If $\ker\varphi\neq1$, the map is not representable, and its image is computed by the kernel invariants in the displayed formula.

The line bundle $P$ and the isomorphism $P^n\cong\rho^*L_P$ define a morphism
\[
v:\mathcal{G}\longrightarrow\mathcal{G}_1,
\]
where $\mathcal{G}_1$ is the stack of $n$-th roots of $L_P$, with tautological line $R_1$. We have $v^*R_1=P$; where $P$ has character $k$, the stabilizer map is $\zeta\rightarrow\zeta^k$. Put $M_1:=v_*M$. By \eqref{eq:root-vanishing-exact-image} and the projection formula,
\[
R\Gamma(\mathcal{G},C_p(M)\otimes P^{-1})
\simeq R\Gamma(\mathcal{G}_1,C_p(M_1)\otimes R_1^{-1}).
\]
Taking the kernel invariants gives this equality also when $v$ is not representable, since $P=v^*R_1$ is trivial on that kernel.

Choose $b_0>0$ such that $L_P^{b_0}$ is very ample, and let $\iota:W\hookrightarrow\mathbb{P}^e$ be its embedding. Let $\mathcal{G}_2$ be the stack of $nb_0$-th roots of $L_P^{b_0}$, with tautological line $R_2$. The map
\[
u_1:\mathcal{G}_1\longrightarrow\mathcal{G}_2,\qquad
(R_0,\phi:R_0^n\simeq L_P)\longmapsto(R_0,\phi^{\otimes b_0})
\]
induces the inclusion $\mu_n\hookrightarrow\mu_{nb_0}$ on stabilizers. On \'etale trivializations its fibers are $\mu_{nb_0}/\mu_n$, so it is representable and finite \'etale of degree $b_0$, and $u_1^*R_2=R_1$. Put $M_2:=u_{1*}M_1$.

Finally, $\iota$ induces a closed embedding $\iota_{\mathcal{G}}:\mathcal{G}_2\hookrightarrow\mathcal{G}_{\mathbb{P}}$, where $\mathcal{G}_{\mathbb{P}}$ is the stack of $nb_0$-th roots of $\mathcal{O}_{\mathbb{P}^e}(1)$, with tautological line $R_{\mathbb{P}}$. Put $M_{\mathbb{P}}:=\iota_{\mathcal{G}*}M_2$. These are single polarizable complex Hodge modules. We obtain
\begin{equation}\label{eq:root-vanishing-complete-support}
R\Gamma(\mathcal{G},C_p(M)\otimes P^{-1})
\simeq R\Gamma(\mathcal{G}_{\mathbb{P}},C_p(M_{\mathbb{P}})\otimes R_{\mathbb{P}}^{-1}).
\end{equation}
The smooth case applies to the right hand side because $R_{\mathbb{P}}^{nb_0}$ is the pullback of $\mathcal{O}_{\mathbb{P}^e}(1)$. It applies equally to summands supported on a proper closed subset of $\mathbb{P}^e$. Thus the argument includes the summands of $M$ supported on the singular locus or on intersections of components of $W$. This proves the first assertion.

For the second assertion, $F_{<0}\mathcal{M}=0$ makes the graded Spencer complex
\[
C_1(M)=\left[T_{\mathcal{G}}\otimes F_0\mathcal{M}\longrightarrow\gr_1^F\mathcal{M}\right]
\]
concentrated in degrees $-1$ and $0$. Hence
\[
\HH^{-1}(\mathcal{G},C_1(M)\otimes P^{-1})
=H^0\left(\mathcal{G},\ker\left[T_{\mathcal{G}}\otimes F_0\mathcal{M}\longrightarrow\gr_1^F\mathcal{M}\right]\otimes P^{-1}\right).
\]
The first assertion makes this group zero.
\end{proof}

\begin{rem}\label{rem:root-vanishing-derived-range}
Theorem~\ref{thm:negative-root-vanishing} applies to one Hodge module in degree zero. If $K$ is a bounded complex whose Hodge-module cohomology is concentrated in $[a_0,b_0]$, the spectral sequence of its perverse truncations has terms
\[
\HH^v(\mathcal{G},C_p(\mathcal{H}^uK)\otimes P^{-1})
\]
of total degree $u+v$. The theorem kills the terms with $v<0$ and therefore gives vanishing only in degrees $i<a_0$. Below we apply it to $M=\mathcal{H}^1h_*\mathsf{A}$, a single module, rather than to $h_*\mathsf{A}$.
\end{rem}

\subsection{Application to the obstruction class}

Return to Construction-Theorem~\ref{consthm:root-cover} and Set-up~\ref{setup:framed-map}. The rational mixed Hodge module $M=\mathcal{H}^1h_*\mathsf{A}$ on $\mathcal{V}$ satisfies
\[
F_{<0}\mathcal{M}=0,\qquad F_0\mathcal{M}=R^1h_*\omega_{\mathcal{S}}
\]
by Construction-Theorem~\ref{consthm:nonreduced-trace}(9). Since $\omega_{\mathcal{S}}=\mathcal{N}^c$ and $h^*\mathcal{T}=\mathcal{N}$, the projection formula gives
\begin{equation}\label{eq:root-vanishing-obstruction-coefficient}
R^1h_*\mathcal{N}^{-j}
=R^1h_*(\omega_{\mathcal{S}}\otimes h^*\mathcal{T}^{-(c+j)})
=F_0\mathcal{M}\otimes\mathcal{T}^{-(c+j)}.
\end{equation}

\begin{cor}\label{cor:eta-vanishes}
Notation and conditions as in Construction-Theorem \ref{consthm:obstruction}. Let $M:=\mathcal{H}^1h_*(\DD\mathbb{Q}_{\mathcal{S}}^H[-(d-1)])$, with underlying filtered right $\mathcal{D}_{\mathcal{V}}$-module $(\mathcal{M},F)$, and let $\sigma:T_{\mathcal{V}}\otimes F_0\mathcal{M}\rightarrow\gr_1^F\mathcal{M}$ be its Kodaira-Spencer map. For every $j\geq1$, we have
\begin{equation}\label{eq:root-vanishing-application}
H^0(\mathcal{V},\ker\sigma\otimes\mathcal{T}^{-(c+j)})
=\HH^{-1}(\mathcal{V},C_1(M)\otimes\mathcal{T}^{-(c+j)})=0.
\end{equation}
In particular, $\eta_j=0$ for every $j\geq b$ and every $j$-order lift $\sss$.
\end{cor}
\begin{proof}
Apply Theorem~\ref{thm:negative-root-vanishing} with
\[
W=\mathbb{P}^r,\quad L=\mathcal{O}_{\mathbb{P}^r}(1),\quad n=Na,\quad
\mathcal{G}=\mathcal{V},\quad P=\mathcal{T}^{c+j}.
\]
Since $c+j>0$, the line bundle $P^{Na}=\rho^*\mathcal{O}_{\mathbb{P}^r}(c+j)$ is pulled back from an ample line bundle. This proves \eqref{eq:root-vanishing-application}. Theorem~\ref{thm:tangential-comparison}(2) places $\eta_j$ in the first group, so $\eta_j=0$.
\end{proof}
\section{Lifting through one order}\label{sec:08-framed-lifting}

We keep the notation of Construction-Theorem~\ref{consthm:root-cover} and Set-up~\ref{setup:framed-map}. The equality $\eta_j=0$ gives lifts of $h_{j,\sss}$ over affine open subsets of $\mathbb{P}^r$. To extend the sections $s_0,\dots,s_r$, we must also identify the pullback of $\mathcal{T}$ with the prescribed line bundle $\mathcal{R}$, extending the given identification on $\mathcal{S}_j$. We first prove that this identification extends. We then use $H^1(Z,C_j\otimes\mathcal{O}_Z(aA))=0$ to obtain a global extension of the tuple.

\begin{nota}\label{not:framed-lifting-restrictions}
Let $\rho:\mathcal{V}\rightarrow\mathbb{P}^r$ be the coarse moduli map, and let $p_S:=p_1:\mathcal{S}\rightarrow B$ be as in Theorem~\ref{thm:root-cover}(1). For $U\subset\mathbb{P}^r$ open, put
\[
Z_U:=Z\cap U,\qquad B_U:=f^{-1}(Z_U),\qquad\mathcal{V}_U:=\rho^{-1}(U).
\]
For $t\geq1$, let $\mathcal{S}_{t,U}$ be the open substack of $\mathcal{S}_t$ with underlying topological space $h^{-1}(\mathcal{V}_U)$, and put $\mathcal{S}_U:=\mathcal{S}_{1,U}$.
\end{nota}

\begin{lem}\label{lem:framed-lifting-frame-correction}
Notation and conditions as in Construction-Theorem \ref{consthm:obstruction}. Let $j\geq b$, let $\sss$ be a $j$-order lift, and let $U\subset\mathbb{P}^r$ be a standard affine open subset. Suppose
\[
u:\mathcal{S}_{j+1,U}\longrightarrow\mathcal{V}_U
\]
is a lift of $h_{j,\sss}|_{\mathcal{S}_{j,U}}$. Then $u$ is projective, and there is an isomorphism
\[
u^*\mathcal{T}\cong\mathcal{R}|_{\mathcal{S}_{j+1,U}}
\]
extending the specified isomorphism $h_{j,\sss}^*\mathcal{T}=\mathcal{R}|_{\mathcal{S}_{j,U}}$.
\end{lem}
\begin{proof}
We first show that $u$ is projective. A nilpotent closed immersion does not change geometric-point stabilizers. At every geometric point, the stabilizer map induced by $u$ is therefore the one induced by $h$, which is injective by Lemma~\ref{lem:root-neighborhoods-framed-projectivity}. Thus $u$ is representable by algebraic spaces \cite[Lemma~3.3.2]{Cad07}. Its restriction to $\mathcal{S}_{j,U}$ is proper, so $u$ is proper as well, since $\mathcal{S}_{j+1,U}$ is a nilpotent thickening.

Let $T\rightarrow\mathcal{V}_U$ be a morphism from a qcqs scheme, and put $\mathcal{S}_{j+1,T}:=\mathcal{S}_{j+1,U}\times_{\mathcal{V}_U}T$. This is an algebraic space proper over $T$. The natural map
\[
q:\mathcal{S}_{j+1,T}\longrightarrow X\times T
\]
is proper: $X\times T$ is separated over $T$, so the graph of $q$ is closed, and the projection from $\mathcal{S}_{j+1,T}\times_T(X\times T)$ is proper. It has finite fibers. Indeed, $\mathcal{S}_{j+1,T}\rightarrow\mathcal{S}_{j+1,U}\times T$ is a base change of the finite diagonal of $\mathcal{V}_U$ \cite[Theorem~2.3.2 and the remark following Theorem~2.3.3]{Cad07}, and $\mathcal{S}_{j+1,U}\rightarrow X$ has finite fibers by Construction-Theorem~\ref{consthm:root-cover}(11)(14). Thus $q$ is finite \cite[Tag~0A4X]{Sta26}. A finite algebraic space over the scheme $X\times T$ is a scheme \cite[Tags~03ZP and~03JH]{Sta26}, and its composition with the projective map $X\times T\rightarrow T$ is projective. This proves the assertion about $u$.

Put
\[
L_u:=u^*\mathcal{T},\qquad\mathcal{P}:=\mathcal{R}\otimes L_u^{-1},\qquad k:=c+j.
\]
The prescribed isomorphism on $\mathcal{S}_{j,U}$ gives a nowhere vanishing section $1$, or frame, of $\mathcal{P}|_{\mathcal{S}_{j,U}}$. We must extend this frame. Tensor the sequence
\[
0\longrightarrow\mathcal{N}^{-j}\xrightarrow{\ w^j\ }\mathcal{O}_{\mathcal{S}_{j+1}}
\longrightarrow\mathcal{O}_{\mathcal{S}_j}\longrightarrow0
\]
with $\omega_{\mathcal{S}_{j+1}}=\mathcal{R}^{c+j}|_{\mathcal{S}_{j+1}}$, and restrict to $\mathcal{S}_{j+1,U}$. This gives
\[
0\longrightarrow\omega_{\mathcal{S}_U}\longrightarrow\omega_{\mathcal{S}_{j+1,U}}
\longrightarrow\mathcal{R}^{c+j}|_{\mathcal{S}_{j,U}}\longrightarrow0.
\]
The first arrow is the trace, locally multiplication by $w^j$. Tensoring with $L_u^{-k}$ and using $L_u|_{\mathcal{S}_U}=\mathcal{N}|_{\mathcal{S}_U}$ gives
\begin{equation}\label{eq:framed-lifting-unit-sequence}
0\longrightarrow\mathcal{N}^{-j}|_{\mathcal{S}_U}\longrightarrow\mathcal{P}^k
\longrightarrow\mathcal{O}_{\mathcal{S}_{j,U}}\longrightarrow0.
\end{equation}
The last identification uses the $k$-th power of the prescribed frame. Sheaves on the closed substacks are regarded here as sheaves on $\mathcal{S}_{j+1,U}$.

Let $h_U:=u|_{\mathcal{S}_U}$. Since $u$ is projective, Construction-Theorem~\ref{consthm:nonreduced-trace}(8) gives the trace injection
\[
R^1(h_U)_*\omega_{\mathcal{S}_U}\lhook\joinrel\longrightarrow R^1u_*\omega_{\mathcal{S}_{j+1,U}}.
\]
Tensoring with $\mathcal{T}^{-k}|_{\mathcal{V}_U}$ and applying the projection formula gives
\[
R^1(h_U)_*(\mathcal{N}^{-j}|_{\mathcal{S}_U})\lhook\joinrel\longrightarrow R^1u_*\mathcal{P}^k.
\]
The connecting map from $(h_{j,\sss})_*\mathcal{O}_{\mathcal{S}_{j,U}}$ in the long exact sequence of \eqref{eq:framed-lifting-unit-sequence} has image in the kernel of this injection. Hence it is zero, and
\[
u_*\mathcal{P}^k\longrightarrow(h_{j,\sss})_*\mathcal{O}_{\mathcal{S}_{j,U}}
\]
is surjective. Since $U$ is affine and $\rho_*$ is exact, taking global sections on $\mathcal{V}_U$ is exact on quasi-coherent sheaves. The unit therefore lifts to
\[
v_{\mathcal{P}}\in H^0(\mathcal{S}_{j+1,U},\mathcal{P}^k),\qquad
v_{\mathcal{P}}|_{\mathcal{S}_{j,U}}=1.
\]
It is a frame because its restriction to the nilpotent closed substack is a frame.

Choose local frames $\xi_{\lambda}$ of $\mathcal{P}$ extending $1$. Write
\[
v_{\mathcal{P}}/\xi_{\lambda}^k=1+\epsilon_{\lambda},\qquad\epsilon_{\lambda}^2=0.
\]
Since $k=c+j>0$, the sections
\[
\widetilde\xi_{\lambda}:=\xi_{\lambda}\left(1+\frac{\epsilon_{\lambda}}{k}\right)
\]
satisfy $\widetilde\xi_{\lambda}^k=v_{\mathcal{P}}$ and restrict to $1$. On an overlap, the ratio of two such sections is $1+\epsilon$, with $\epsilon^2=0$ and $(1+\epsilon)^k=1$. Thus $k\epsilon=0$, so the ratio is $1$. The local sections glue to a frame of $\mathcal{P}$, also respecting the \'etale descent isomorphisms by the same uniqueness. This is the required extension of $L_u\cong\mathcal{R}|_{\mathcal{S}_{j+1,U}}$.
\end{proof}

\begin{thm}\label{thm:fixed-target-lift}
Notation and conditions as in Set-up~\ref{setup:final-setup}. Let $j\geq b$ and let $\sss$ be a $j$-order lift. If
\[
H^1(Z,C_j\otimes\mathcal{O}_Z(aA))=0,
\]
then there is a $(j+1)$-order lift $\sss'$ with $\sss'|_{\mathcal{S}_j}=\sss$.
\end{thm}
\begin{proof}
Let $o_j\in H^1(\mathcal{S},h^*\At(\mathcal{T})\otimes\mathcal{N}^{-j})$ be the class of Construction-Theorem~\ref{consthm:oj}(4). We first identify it with the obstructions to extending the individual sections.

Let $x_0,\dots,x_r$ be the coordinate sections of $\mathcal{T}^{Na}=\rho^*\mathcal{O}_{\mathbb{P}^r}(1)$. An operator $D\in\At(\mathcal{T})$ acts on $\mathcal{T}^{Na}$ by the Leibniz rule. On the chart $x_0\neq0$, choose an \'etale local frame $\xi$ with $x_0=\xi^{Na}$ and write $x_l=z_l\xi^{Na}$. If $\zeta$ is the symbol of $D$ and $D(\xi)=\lambda\xi$, then evaluation on the coordinates gives
\[
(D(x_0),\dots,D(x_r))
=\left(Na\lambda,\ \zeta(z_1)+Na\lambda z_1,\dots,\zeta(z_r)+Na\lambda z_r\right)\xi^{Na}.
\]
These entries determine $\lambda$ and $\zeta$, and every tuple of entries occurs. Evaluation is independent of the frame, so the local isomorphisms glue to
\begin{equation}\label{eq:framed-lifting-euler}
\At(\mathcal{T})\cong(\mathcal{T}^{Na})^{\oplus(r+1)}.
\end{equation}
The scalar operator $1$ maps to $Na(x_0,\dots,x_r)$, and the symbol map becomes the Euler map to $T_{\mathcal{V}}$.

Consider
\begin{equation}\label{eq:framed-lifting-tuple-sequence}
0\longrightarrow\mathcal{N}^{Na-j}\longrightarrow\mathcal{R}^{Na}|_{\mathcal{S}_{j+1}}
\longrightarrow\mathcal{R}^{Na}|_{\mathcal{S}_j}\longrightarrow0,
\end{equation}
and let $\delta_j$ be its connecting map from $H^0(\mathcal{S}_j,\mathcal{R}^{Na}|_{\mathcal{S}_j})$ to $H^1(\mathcal{S},\mathcal{N}^{Na-j})$. On affine \'etale charts, choose lifts of $\psi_{j,\sss}$ over $\mathcal{K}$. These give local extensions of the sections $s_l$, because their pullback line bundle is identified with $\mathcal{R}$. Evaluation in \eqref{eq:framed-lifting-euler} sends the differences of these lifts to the differences of the extended sections. Hence
\[
o_j=(\delta_j(s_0),\dots,\delta_j(s_r))
\in H^1(\mathcal{S},\mathcal{N}^{Na-j})^{\oplus(r+1)}.
\]

We show that this class vanishes over every standard affine open $U\subset\mathbb{P}^r$. By Construction-Theorem~\ref{consthm:oj}(4),
\[
\sym_{\mathcal{T}}(o_j)=\operatorname{ob}_j(h_{j,\sss}),\qquad
e_h(\operatorname{ob}_j(h_{j,\sss}))=\eta_j=0,
\]
where the last equality is Corollary~\ref{cor:eta-vanishes}. Since higher cohomology of quasi-coherent sheaves on $\mathcal{V}_U$ vanishes, the Leray edge is an isomorphism
\[
H^1(\mathcal{S}_U,h^*T_{\mathcal{V}}\otimes\mathcal{N}^{-j})
\cong H^0(\mathcal{V}_U,T_{\mathcal{V}}\otimes R^1h_*\mathcal{N}^{-j}).
\]
Thus $\operatorname{ob}_j(h_{j,\sss})|_{\mathcal{S}_U}=0$. By Construction-Theorem~\ref{consthm:oj}(1) applied to $V=\mathcal{S}_{j+1,U}$, a lift of $h_{j,\sss}$ on $\mathcal{S}_{j+1,U}$ exists. It gives
\[
u:\mathcal{S}_{j+1,U}\longrightarrow\mathcal{V}_U
\]
extending $h_{j,\sss}|_{\mathcal{S}_{j,U}}$. By Lemma~\ref{lem:framed-lifting-frame-correction}, the specified isomorphism with $\mathcal{R}$ extends to $u^*\mathcal{T}\cong\mathcal{R}|_{\mathcal{S}_{j+1,U}}$. The section $w$, viewed through this isomorphism, gives a lift to $Y$ over $\mathcal{K}$. Therefore
\[
o_j|_{\mathcal{S}_U}=0.
\]

It remains to kill the class globally. By Lemma~\ref{lem:root-cover-basic}(2), Construction-Theorem~\ref{consthm:jj}(6) and Theorem~\ref{thm:root-cover}(1)--(3),
\[
p_{S*}\mathcal{N}^{-j}=\overline{J}_j,\qquad
\mathcal{N}^{Na}=p_S^*f^*\mathcal{O}_Z(aA),
\]
and $p_{S*}$ is exact. The projection formula identifies the group containing $o_j$ with
\begin{equation}\label{eq:framed-lifting-coarse-obstruction}
H^1(\mathcal{S},\mathcal{N}^{Na-j})^{\oplus(r+1)}
=H^1(B,\overline{J}_j\otimes f^*\mathcal{O}_Z(aA))^{\oplus(r+1)}.
\end{equation}
Since $C_j=f_*\overline{J}_j$, Leray for $f:B\rightarrow Z$ gives
\[
\begin{aligned}
0\longrightarrow H^1(Z,C_j\otimes\mathcal{O}_Z(aA))&\longrightarrow H^1(B,\overline{J}_j\otimes f^*\mathcal{O}_Z(aA))\\
&\longrightarrow H^0(Z,R^1f_*\overline{J}_j\otimes\mathcal{O}_Z(aA))
\longrightarrow H^2(Z,C_j\otimes\mathcal{O}_Z(aA)).
\end{aligned}
\]
Each component of $o_j$ restricts to zero on $B_U$, because $o_j|_{\mathcal{S}_U}=0$. The edge map commutes with restriction, so its image vanishes on every $Z_U$, hence on $Z$. Each component therefore comes from $H^1(Z,C_j\otimes\mathcal{O}_Z(aA))$, which is zero by hypothesis. Thus $o_j=0$.

We obtain sections $s'_0,\dots,s'_r$ of $\mathcal{R}^{Na}|_{\mathcal{S}_{j+1}}$ restricting to $s_0,\dots,s_r$. They generate this line bundle: if $\mathcal{Q}$ is the cokernel of the corresponding evaluation map and $\mathcal{I}$ is the square-zero ideal of $\mathcal{S}_j$, then generation on $\mathcal{S}_j$ gives
\[
\mathcal{Q}=\mathcal{I}\mathcal{Q}=\mathcal{I}^2\mathcal{Q}=0.
\]
Their restrictions to $\mathcal{S}_b$ are $(f\circ p_b)^*x_l$. Hence $\sss':=(s'_0,\dots,s'_r)$ is a $(j+1)$-order lift extending $\sss$.
\end{proof}

\section{Lifting through finitely many orders}\label{sec:09-abundance}

\begin{thm}\label{thm:finite-window}
Notation and conditions as in Set-up \ref{setup:final-setup}. Then there exist integers $\tau\geq 1$ and $a_0\geq\tau$ satisfying the following properties. For any $a\geq a_0$, we have:
\begin{enumerate}
    \item $h^0(Q_{N(a-\tau)},\mathcal{O}_X(aE)|_{Q_{N(a-\tau)}})\geq h^0(Z,\mathcal{O}_Z(aA))$.
\item $h^0(X,\mathcal{O}_X(aE))\geq h^0(X,\mathcal{O}_X(\tau E))+h^0(Z,\mathcal{O}_Z(aA))-h^1(X,\mathcal{O}_X(\tau E)).$
\end{enumerate}
\end{thm}
\begin{proof}
By the Serre vanishing, there exists a positive integer $\tau$ such that
$$H^1(Z,C_k\otimes\mathcal{O}_Z(sA))=0$$
for any $1\leq k\leq N$ and any $s\geq\tau$. In the following we fix such $\tau$. Set $a_0:=\tau+1$. 

For any integer $a\geq a_0$, we choose a basis $x_0,\dots,x_r$ of $|aA|$ and let $j_a:=N(a-\tau)$. Since $a>\tau$ and $b\leq N$ (Construction-Theorem \ref{consthm:root-cover}(9)), we have
$$j_a\geq N\geq b.$$
Thus
$$\sss_b:=\left((f\circ p_b)^*x_0,\dots,(f\circ p_b)^*x_r\right)$$
is a $b$-order lift. For any $b\leq j<j_a$, write $j=tN+k$ where $t,k$ are integers and $1\leq k\leq N$. Then we have $t\leq a-\tau-1$. By Lemma \ref{lem:jj}(2), we have
$$H^1(Z,C_j\otimes\mathcal{O}_Z(aA))=H^1(Z,C_k\otimes\mathcal{O}_Z((a-t)A))=0$$
as $a-t\geq\tau+1\geq\tau$. 

Applying Theorem~\ref{thm:fixed-target-lift} successively for $j=b,b+1,\dots,j_a-1$, we obtain a $j_a$-order lift $\sss=(s_0,\dots,s_r)$. Thus $H^0(\mathcal{S}_{j_a},\mathcal{R}^{Na}|_{\mathcal{S}_{j_a}})\rightarrow H^0(\mathcal{S}_b,\mathcal{R}^{Na}|_{\mathcal{S}_{b}})$ is surjective. By Theorem~\ref{thm:root-cover}(4)(5),
$$H^0(Q_{j_a},\mathcal{O}_X(aE)|_{Q_{j_a}})\cong H^0(\mathcal{S}_{j_a},\mathcal{R}^{Na}|_{\mathcal{S}_{j_a}})\rightarrow H^0(\mathcal{S}_b,\mathcal{R}^{Na}|_{\mathcal{S}_{b}})\cong H^0(Z,\mathcal{O}_Z(aA))$$
is surjective. This implies (1). We consider the short exact sequence
$$0\longrightarrow\mathcal{O}_X(\tau E)\xrightarrow{\ \cdot s_E^{a-\tau}\ }\mathcal{O}_X(aE)\longrightarrow\mathcal{O}_X(aE)|_{Q_{j_a}}\longrightarrow 0.$$
Then by (1), we have
\begin{align*}
    h^0(X,\mathcal{O}_X(aE))&\geq h^0(X,\mathcal{O}_X(\tau E))+h^0(Q_{j_a},\mathcal{O}_X(aE)|_{Q_{j_a}})-h^1(X,\mathcal{O}_X(\tau E))\\
    &\geq h^0(X,\mathcal{O}_X(\tau E))+h^0(Z,\mathcal{O}_Z(aA))-h^1(X,\mathcal{O}_X(\tau E)).
\end{align*}
This implies (2).
\end{proof}

\begin{thm}\label{thm:no-supported-core}
Let $d$ be a positive integer and $2\leq\nu\leq d-1$ an integer. Then there does not exist an integral supported core of type $(d,\nu)$.
\end{thm}
\begin{proof}
    Suppose that there exists an integral supported core $(X,B)$ of type $(d,\nu)$. Then there exist data and structures satisfying the conditions of Set-up \ref{setup:final-setup}. In the following we adopt the notation and conditions as in Set-up \ref{setup:final-setup}, then we have
$$h^0(X,\mathcal{O}_X(aE))\geq h^0(X,\mathcal{O}_X(\tau E))+h^0(Z,\mathcal{O}_Z(aA))-h^1(X,\mathcal{O}_X(\tau E))$$
by Theorem \ref{thm:finite-window}(2) for some positive integer $\tau$ and any integer $a\gg\tau$. This implies that $h^0(X,\mathcal{O}_X(aE))\geq 2$ for any $a\gg 0$, so $$\kappa(E)=\kappa(D)=\kappa(K_X+B)\geq 1.$$
Since $(X,B)$ is an integral supported core, $\kappa(K_X+B)=0$, a contradiction.
\end{proof}

\section{Proof of the main theorems}\label{sec:main-proofs}

The following theorem can be considered as a unification of Theorems \ref{thm: main i} and \ref{thm: main ii}.

\begin{thm}\label{thm:main-unified}
Assume Conjecture \ref{conj:nv-smooth} in dimension $d-1$. Let $(X,B)$ be a projective lc pair of dimension $\leq d$. Assume that one of the following two conditions hold.
\begin{enumerate}
    \item $(X,B)$ is klt and $\kappa_{\iota}(K_X+B)\geq 0$.
    \item Conjecture \ref{conj:nv-smooth} holds in dimension $d$.
\end{enumerate}
Then $(X,B)$ has a good minimal model or a Mori fiber space. In particular, we may run a $(K_X+B)$-MMP with scaling of an ample divisor which terminates with a good minimal model or a Mori fiber space.
\end{thm}
\begin{proof}
We may assume that $K_X+B$ is pseudo-effective, otherwise we are done by \cite[Theorem 1.7]{HH20}. We denote by $\nu:=\nu(K_X+B)$. By Theorem \ref{thm:main-unified}(2) for $d-1$, Conjecture \ref{conj:lc-gmm} holds in dimension $\leq d-1$, and we may assume that $\dim X=d$. Under either condition (1) or (2), we have $\kappa_{\iota}(K_X+B)\geq 0$. By Proposition \ref{prop: nv implies eomm}, we only need to show that $(X,B)$ has a good minimal model.

In the following we prove by using contradiction and assume that $(X,B)$ does not have a good minimal model. By Theorem \ref{thm:real-reduction} and Proposition \ref{prop:integral-model}, we may assume that $(X,B)$ is a supported core and $2\leq \nu\leq d-1$. This contradicts Theorem \ref{thm:no-supported-core}.
\end{proof}

\begin{proof}[Proof of Theorem \ref{thm: main i}]
It follows from Theorem \ref{thm:main-unified}(2).
\end{proof}

\begin{proof}[Proof of Theorem \ref{thm: main ii}]
It follows from Theorem \ref{thm:main-unified}(1).   
\end{proof}

\section{Applications}\label{sec: applications}

In this section we give some applications of Theorem \ref{thm: main ii}. 

\subsection{Existence of good minimal models}

We show that good minimal models exist in many special cases.

\begin{cor}\label{cor: nv imply nv dim 4}
    Let $(X,B)$ be a projective klt pair of dimension $\leq 4$. If $\kappa_{\iota}(K_X+B)\geq 0$, then $(X,B)$ has a good minimal model.
\end{cor}
\begin{proof}
It follows from Theorem \ref{thm: main ii} and \cite[Theorem 1.1]{Miy88a}.
\end{proof}

\begin{cor}\label{cor: cy abundance dim 4}
    Let $(X,B)$ be a projective pair of dimension $\leq 4$ such that $(X,\Delta)$ is klt and $K_X+B\equiv 0$ for some $\Delta$. Then $(X,B)$ has a good minimal model or a Mori fiber space.
\end{cor}
\begin{proof}
By \cite[Lemma 6.1]{Gon11}, $K_X+B\sim_{\mathbb R}0$. Pick $0<\epsilon\ll 1$, then
$$K_X+\epsilon B+(1-\epsilon)\Delta\sim_{\mathbb R}\epsilon(K_X+B)+(1-\epsilon)(K_X+\Delta)$$
and $(X,\epsilon B+(1-\epsilon)\Delta)$ is klt. Possibly replacing $B$ with $\epsilon B+(1-\epsilon)\Delta$, we may assume that $(X,B)$ is klt. The corollary follows from Corollary \ref{cor: nv imply nv dim 4} and \cite[Theorem 1.3]{Has19b}.
\end{proof}

\begin{cor}\label{cor: cy nv nef is sa}
    Let $(X,B)$ be a projective klt Calabi-Yau pair of dimension $\leq 4$ and $D\geq 0$ a nef $\mathbb R$-divisor on $X$. Then $D$ is semi-ample.
\end{cor}
\begin{proof}
We have that $(X,B+\epsilon D)$ is klt for some $0<\epsilon\ll 1$. By \cite[Lemma 6.1]{Gon11}, $K_X+B+\epsilon D\sim_{\mathbb R}\epsilon D$ is nef. By Corollary \ref{cor: cy abundance dim 4}, $K_X+B+\epsilon D\sim_{\mathbb R}\epsilon D$ is semi-ample, so $D$ is semi-ample.
\end{proof}

Corollary \ref{cor: cy nv nef is sa} is tightly related to the hyper-Kähler SYZ conjecture (cf. \cite[Conjecture 1.7]{Ver10}) for fourfolds. In particular, we have the following:
\begin{cor}
    Let $X$ be a hyper-Kähler fourfold and $\mathcal{L}$ a nef line bundle on $X$. If $H^0(\mathcal{L}^{\otimes m})\not=0$ for some positive integer $m$, then $\mathcal{L}$ is semi-ample. 
\end{cor}
\begin{proof}
This is a special case of Corollary \ref{cor: cy nv nef is sa}.
\end{proof}

\begin{cor}
    Let $(X,B)$ be a projective uniruled lc pair of dimension $4$. Then $(X,B)$ has a good minimal model or a Mori fiber space.
\end{cor}
\begin{proof}
    It follows from Theorem \ref{thm: main ii} and \cite[Corollary 1.2]{LM21}.
\end{proof}

\subsection{Equivalence with the Campana-Peternell's conjecture}

We can also show that the non-vanishing conjecture implies Campana-Peternell's conjecture. In particular, they are equivalent. This improves a result of Schnell \cite{Sch25}.

\begin{conj}[{Campana-Peternell conjecture, \cite[Page 43, Conjecture]{CP11}}]\label{conj: cp11}
Let $X$ be a smooth projective variety and let $A\geq 0$ be a divisor on $X$ such that $NK_X-A$ is pseudo-effective for some positive integer $N$. Then $\kappa(X)\geq\kappa(A)$.
\end{conj}

\begin{cor}
  Conjecture \ref{conj:nv-smooth} in dimension $d$ and Conjecture \ref{conj: cp11} in dimension $d$ are equivalent.
\end{cor}
\begin{proof}
Conjecture \ref{conj: cp11} in dimension $d$ implies   Conjecture \ref{conj:nv-smooth} in dimension $d$ by considering $A=0$. Now assume Conjecture \ref{conj:nv-smooth} in dimension $d$. Let $X$ be a smooth projective variety of dimension $d$ and let $A\geq 0$ be a divisor on $X$ such that $NK_X-A$ is pseudo-effective for some positive integer $N$. By Theorem \ref{thm: main i}, 
$$\kappa(X)=\nu(X)=\nu(NK_X)=\nu(A+(NK_X-A))\geq\nu(A)\geq\kappa(A).$$
The corollary follows.
\end{proof}

\subsection{Application to the Iitaka conjecture}

We conclude with paper with an application to the Iitaka conjecture:

\begin{conj}[Iitaka]\label{conj: iitaka conjecture}
    Let $f: X\rightarrow Z$ be a contraction between smooth projective varieties with general fiber $F$ such that $\dim Z>0$. Then $\kappa(X)\geq\kappa(F)+\kappa(Z)$.
\end{conj}

Conjecture \ref{conj: iitaka conjecture} was known when $\dim X\leq 6$ \cite[Theorems 1.1 and 1.2]{Bir09} or when $\dim X-\dim Z\leq 3$ (cf. \cite[Theorem 1.1 and Corollary 1.2]{Kaw85b} with \cite{Kaw92}). We deduce Conjecture \ref{conj: iitaka conjecture} when $\dim X\leq 7$ or when $\dim X-\dim Z\leq 4$.

\begin{cor}
   Conjecture \ref{conj: iitaka conjecture} holds when $\dim F\leq 4$ or when $\dim X\leq 7$. 
\end{cor}
\begin{proof}
    First we assume that $\dim F\leq 4$. In this case, either $\kappa(F)=-\infty$ and the conjecture is trivial, or $\kappa(F)\geq 0$ and we are done by \cite[Theorem 1.1 and Corollary 1.2]{Kaw85b} and Corollary \ref{cor: nv imply nv dim 4}. 
    
    Now we assume that $\dim X\leq 7$. If $\dim F\leq 4$ then we are done, so $\dim F\geq 5$. Thus $\dim Z\leq 2$. If $\dim Z=1$, then we are done by \cite[Theorem 2]{Kaw82}. If $\dim Z=2$ and $\kappa(F)=0$ then we are done by \cite[Theorem 1.4]{Bir09}. Thus we may assume that $\dim Z=2$ and $\kappa(F)\geq 1$. By Theorem \ref{thm:kappa-positive-case} and Corollary \ref{cor: nv imply nv dim 4}, $F$ has a good minimal model. We are done by \cite[Theorem 1.1 and Corollary 1.2]{Kaw85b}.
\end{proof}

\section{Final remarks}\label{rem:final-remarks}

\begin{rem}
    Of course there could be more application and we leave them to the readers of interest. One particular type of applications, provided that the abundance conjecture will not be proven in a short period of time in full generality, is the following: Consider a specific category of pairs $\mathcal{P}=\cup_d\mathcal{P}_d$ so that $|K_X+B|_{\mathbb R}\not=\emptyset$ for any $(X,B)\in\mathcal{P}$ and $\dim X=d$ for any $(X,B)\in\mathcal{P}_d$. Assume that, following the same lines of the proof of our main theorems,the condition on the existence of good minimal models in lower dimensions is only used for pairs in $\mathcal{P}_{d-1}$. Then the same induction shall imply that existence of good minimal models for any $(X,B)\in\mathcal{P}$ unconditionally.
\end{rem}

\begin{rem}
    It will be interesting to see if there is a way to prove our main theorems using extension theorems (cf. \cite{DHP13}) instead of the current Hodge theoretic approach. It seems that the $\nu=d-1$ case can be proven using plt extension theorems and using vanishing theorems for $h^i$ with $i\geq 2$, yet when $\nu\leq d-2$ the arguments could not go through.
\end{rem}

\begin{rem}
    The authors expect that the final proof of the non-vanishing conjecture, which is now equivalent to the abundance conjecture, should still be via analytic methods. This should also be more difficult, as it seems that zero essential progress has been made since \cite{Miy88a}.
\end{rem}

\subsection*{AI declaration} We have used Deepseek, Fable, GLM, Chatgpt, the Danus system \cite{Liu+26}, and the Rethlas system \cite{Ju+26} in the preparation of the work.

\subsection*{Acknowledgements}
The work was partially supported by the National Key R\&D Program of China \#\allowbreak 2024YFA1014400. 

The author would like to thank Zheng Xu for extremely useful and helpful discussions with him, and his contribution in the preparation of the paper, particularly during the period of September and October of 2025. The paper is soloed by the author only because Zheng voluntarily declined to have the authorship as he does not have the time to look through the paper.

The author would like to thank Yoshinori Gongyo, Christopher D. Hacon, and Gang Tian for suggesting the Hodge theoretic approach in 2025, thank Mircea Musta\c{t}\u{a}, Chuanhao Wei, Lei Wu for useful discussions on mixed Hodge modules and vanishing theorems in March 2025, May 2026, and June 2026 in Hangzhou, Ann Arbor, and Xiamen. The authors would also like to thank Caucher Birkar, Paolo Cascini, Omprokash Das, Vladimir Lazi\'c, Zhan Li, Wenfei Liu, Wenhao Ou, Junchao Shentu, Vyachaslav V. Shokurov, Nikolaos Tsakanikas, Zhiyu Tian, Jakub Witaszek, Lingyao Xie, and Lei Zhang for useful discussions and comments.

A large part of the paper was written when the author attended the conference ``Boundedness and Moduli Problems in Birational Geometry and Foliation Theory'' conference in Centro di Ricerca Matematica Ennio De Giorgi, Pisa, Italy from October 27th 2025 to October 31st 2025. The first author would like to thank the organizers Laura Pertusi, Roberto Svaldi, Luca Tasin, and also to Fabrizio Catanese for their hospitality.

The author would like to thank Bin Dong, Guoxiong Gao, Zeming Sun, Bin Wu, and Leheng Chen for AI assistance, and would like to thank Christohper D. Hacon, Ruochuan Liu, and Gang Tian for consistence support and encouragement.

\end{document}